\documentclass[12pt]{article}
\pdfoutput=1
\usepackage{amsmath,amsfonts,amssymb,amsthm}
\usepackage{txfonts}%
\usepackage{graphicx}   %
\usepackage{bm,color}
\usepackage{braket}

\theoremstyle{plain}
\newtheorem{theorem}{Theorem}[section]
\newtheorem{lemma}[theorem]{Lemma}

\newtheorem{proposition}[theorem]{Proposition}
\theoremstyle{definition}
\newtheorem{definition}[theorem]{Definition}

\theoremstyle{remark}
\newtheorem{remark}[theorem]{Remark}
\theoremstyle{conjecture}
\newtheorem{conjecture}[theorem]{Conjecture}

\newcommand{\kerz}{\operatorname{\ker_{\mathbb Z}}}

\newcommand{\swapj}[1]{\stackrel{#1}{\leftrightarrow}}
\newcommand{\swapoperation}[1]{\stackrel{#1}{\longleftrightarrow}}
\newcommand{\rank}{\operatorname{rank}}

\newcommand{\N}{{\mathbb N}}

\newcommand{\Z}{{\mathbb Z}}

\newcommand{\cF}{{\mathcal F}}
\newcommand{\cI}{{\mathcal I}}

\newcommand{\cP}{{\mathcal P}}
\newcommand{\cS}{{\mathcal S}}

\newcommand{\idp}{i^{\prime\prime}}
\newcommand{\kdp}{k^{\prime\prime}}
\newcommand{\kpr}{k_{\mathrm{pr}}}
\newcommand{\kim}{k_{\mathrm{im}}}

\bmdefine{\Ba}{a}
\bmdefine{\Bb}{b}
\bmdefine{\Bc}{c}
\bmdefine{\Be}{e}
\bmdefine{\Bg}{g}
\bmdefine{\Br}{r}
\bmdefine{\Bt}{t}
\bmdefine{\Bw}{w}
\bmdefine{\Bx}{x}
\bmdefine{\By}{y}
\bmdefine{\Bz}{z}

\newcommand{\bmb}{{\bm b}}

\newcommand{\weight}{v}
\newcommand{\fbsz}[1]{\nu(#1)}

\begin{document}

\title{Markov degree of configurations defined by fibers of a configuration}

\author{
Takayuki Koyama\thanks{Graduate School of Information Science and Technology, University of Tokyo}, \ 
Mitsunori Ogawa\footnotemark[1] \ 
and Akimichi Takemura\footnotemark[1]\ %
}
\date{July, 2014}
\maketitle

\begin{abstract}
We consider a series of configurations defined by fibers of a given base configuration.  We prove
that Markov degree of the configurations is bounded from above by the Markov complexity of the
base configuration.  
As important examples of base configurations we consider incidence matrices of graphs and 
study the maximum Markov degree of configurations defined by fibers of the incidence matrices.
In particular we give a proof that the Markov degree for two-way transportation polytopes is three.
\end{abstract}

\noindent
{\it Keywords and phrases:} \ 
algebraic statistics,
Markov basis,
transportation polytopes

\section{Introduction}
\label{sec:intro}
The study of Markov bases has been developing rapidly since
the seminal paper of Diaconis and Sturmfels (\cite{diaconis1998algebraic}), which established the equivalence of 
a Markov basis for a discrete exponential model in statistics
and a generating set of a corresponding toric ideal.
See \cite{aoki2012markov}, \cite{drton-sturmfels-sullivant-lecture} and \cite{hibi_book_13} 
for terminology of algebraic statistics and toric ideals used in this paper.

When we study Markov bases for a specific problem, usually we are not faced with a single
configuration, but rather with a series of configurations, possibly parameterized by a few
parameters. For example, Markov bases associated with complete bipartite graphs $K_{I,J}$ 
(in statistical terms, independence model of $I\times J$ two-way contingency tables) are parameterized 
by $I$ and $J$.  In this case, Markov bases consist of moves of degree two
irrespective of $I$ and $J$.  In more general cases, some measure of complexity of Markov bases
grows with the parameter and we are interested in bounding the growth.

There are some typical procedures to generate a series of configurations based on a given set of configurations.
Perhaps the most important construction is the higher Lawrence lifting of a configuration, for which 
Santos and Sturmfels  (\cite{santos2003higher}) described the growth by the notion of Graver complexity.
Another important  construction is the nested configuration (\cite{ohsugi2009toric}), where generated
series of configurations basically inherit nice properties of original configurations.
In this paper we define a new procedure to generate a series of configurations using 
fibers of a given  configuration, which we call the {\em base} configuration.  
This  construction is closely related to the higher
Lawrence lifting of the base configuration and using this fact we 
prove that {\em Markov degree} of the configurations is bounded from above by the {\em Markov complexity} of the
base configuration.  

There are some nice problems, such as the
complete bipartite graphs,  where the moves  of degree two forms a Markov basis.  
When a minimal Markov basis contains a move of degree three or higher, 
it is usually very hard to control measures of complexity of Markov bases.
A notable exception is the conjecture by \cite{diaconis2006markov} that the Markov degree associated with the Birkhoff polytope is three, i.e., the toric ideal associated with the Birkhoff polytope is generated by
binomials of degree at most three. This conjecture was proved in \cite{YamaguchiOgawaTakemura}.
In view of \cite{haase2009poly} and \cite{YamaguchiOgawaTakemura}, 
Christian Haase (personal communication, 2013) suggested that the Markov degree associated with two-way transportation polytopes 
and flow polytopes is three.
Very recently Domokos and  Jo\'o (\cite{domokos-joo}) gave a proof of this general conjecture.
Adapting the arguments in \cite{YamaguchiOgawaTakemura},
we give a proof that the Markov degree associated with two-way transportation polytopes 
is three in Section \ref{subsec:MD-for-two-way}.
Two-way transportation polytopes are important examples in our framework,
since  they are fibers of the incidence matrix of a complete bipartite graph.

The organization of this paper is as follows. 
In  Section \ref{sec:main} we set up the framework of this paper and
prove the main theorem that the Markov degree of the configurations
defined by fibers of a base configuration is 
bounded from above by the Markov complexity of the base configuration.  
In the remaining sections of this paper we investigate 
the maximum Markov degree and the Markov complexity
of some important base configurations.
In Section \ref{sec:complete-graphs} we study incidence matrices of complete graphs and
in Section \ref{sec:complete-bipartite-graphs} we study those of complete bipartite
graphs as base configurations.
We end the paper with some discussions in Section \ref{sec:discussion}.

\section{Main result}
\label{sec:main}

Let $A$ be a $d\times n$ configuration matrix.
Elements of the integer kernel $\kerz A$ of $A$ are called {\em moves} for $A$.
As in Section 1.5.1  of \cite{hibi_book_13} we assume that
there exists a $d$-dimensional row vector ${\bm \weight}$ such that ${\bm \weight}A = (1,1,\dots,1)$.
Let $\N=\{0,1,2,\dots\}$ denote the set of non-negative integers and let
$\N A = \{ A\Bx \mid \Bx\in \N^n\}$. For $\Bb \in \N A$
\[
\cF_{A,\Bb}=\{ \Bx\in \N^n \mid A\Bx=\Bb\}
\]
is the $\Bb$-fiber of $A$.  
Each fiber $\cF_{A,\Bb}$ is a finite set and non-empty for $\Bb \in \N A$. We denote the size of $\cF_{A,\Bb}$ by $\fbsz{\Bb}=|\cF_{A,\Bb}|$.
Hence with an appropriate term order the elements of $\cF_{A,\Bb}$ are enumerated as
\[
\cF_{A,\Bb}=\{ \Bx_1, \dots, \Bx_{\fbsz{\Bb}} \}.
\]
We look at $\Bx_i$, $i=1,\dots,\fbsz{\Bb}$, as $n$-dimensional column vectors and we define
an $n\times \fbsz{\Bb}$ matrix as
\[
A_\Bb = (\Bx_1, \dots, \Bx_{\fbsz{\Bb}}).
\]
Note that $\Bx\in \cF_{A,\Bb}$ implies 
\[{\bm \weight} \Bb = {\bm \weight} A\Bx= |\Bx|=x_1 + \dots + x_n > 0
\]
if 
$\Bx\neq 0$. Hence for $\tilde {\bm \weight}={\bm \weight}A/({\bm \weight}\Bb)$
\begin{equation}
\label{eq:Ab-configuration}
\tilde {\bm \weight} A_\Bb = (1,1,\dots,1),
\end{equation}
and $A_\Bb$ is a configuration.

Consider the set of moves for $A_\Bb$ of degree at most $m$.  The {\em Markov degree} ${\rm MD}(A_\Bb)$ of $A_\Bb$
is the minimum value of $m$ such that the moves of degree at most $m$ form a Markov basis (cf.\ \cite{YamaguchiOgawaTakemura}, \cite{haws-etal}).  
We are interested in the maximum of ${\rm MD}(A_\Bb)$ when $\Bb$ ranges over $\N A$:
\[
\max_{\Bb\in \N A} {\rm MD}(A_\Bb).
\]

Let $A^{(N)}$ denote the $N$-th Lawrence lifting of $A$ (cf.\ \cite{santos2003higher}).
The moves for $A^{(N)}$ are written as $\Bz=(\Bz_1, \dots \Bz_N)$, such that $\sum_{k=1}^N \Bz_k=0$ 
and $\Bz_k\in \kerz A$, $k=1,\dots,N$.  In this paper, we call $\Bz_k$ the $k$-th  {\em layer}
or {\em slice} of $\Bz$. The {\em type} of $\Bz$ is the number of non-zero layers
among $\Bz_1, \dots, \Bz_N$:
\[
{\rm type}(\Bz)= | \{ k \mid \Bz_k \neq 0 \} |.
\]
Let ${\mathcal G}\left(A^{(N)}\right)$ denote the Graver basis of $A^{(N)}$.  Then
the {\em Graver complexity} of $A$ is 
defined (cf.\ \cite{santos2003higher}, \cite{berstein09}, \cite{kudo-takemura}) as
\begin{equation*}
{\rm GC}(A)=\sup\left(\{0\}\cup\Set{{\rm type}(x) | x\in
\bigcup_{N\geq 1}{\mathcal G}\left(A^{(N)}\right)}\right),
\end{equation*}
where $\{0\}$ is needed for the case that the columns of $A$ are linearly independent.
Santos and Sturmfels (\cite{santos2003higher}) gave an explicit expression for the Graver complexity,
which we will use for computing the Graver complexity of some configurations.
The {\em Markov complexity} ${\rm MC}(A)$ of $A$ is defined as
the minimum value of $m$ such that the moves of type at most $m$ form a Markov basis
for every $A^{(N)}$.
Note that ${\rm MC}(A)\le {\rm GC}(A)$
since a minimal Markov basis is contained in the Graver basis.

Now we are ready to state our main theorem.

\begin{theorem} \label{thm:main}
The Markov degree of $A_\Bb$ is bounded from above by the Markov complexity of $A$:
\begin{equation}
\label{eq:main-result}
\max_{\Bb\in \N A} {\rm MD}(A_\Bb)  \le {\rm MC}(A).
\end{equation}
\end{theorem}

Before giving a proof, we discuss how a fiber of $A_\Bb$ is embedded in a fiber of
some $A^{(N)}$.  For $\Bc \in \N A_\Bb$ consider an element $\By=(y_1,\dots,y_{\fbsz{\Bb}})$
of $\cF_{A_\Bb, \Bc}$. By
\[
\Bc = A_\Bb \By = \Bx_1 y_1 + \dots + \Bx_{\fbsz{\Bb}} y_{\fbsz{\Bb}}
\]
and by \eqref{eq:Ab-configuration},  we see that
$
|\By|=y_1 + \dots + y_{\fbsz{\Bb}}=\tilde {\bm \weight} \Bc
$
is common for all $\By\in \cF_{A_\Bb, \Bc}$. Let $N= |\By|, \By\in \cF_{A_\Bb, \Bc}$.
Then $\By \in \cF_{A_\Bb, \Bc}$ is identified with a multiset $\{ \Bw_1, \dots, \Bw_N\}$
of elements (columns) of $A_\Bb$, where $\Bx_i$ is repeated $y_i$ times, e.g.:
\begin{equation*}
\Bx_1 = \Bw_1= \dots = \Bw_{y_1}, \ \ \Bx_2 = \Bw_{y_1+1}= \dots = \Bw_{y_1+y_2},\  \dots
\end{equation*}
In this notation 
\begin{equation}
\label{eq:Ab-fiber}
\Bw_k\in \cF_{A,\Bb}, \ k=1,\dots,N, \ \  \text{and}  \ \ 
\Bc = \Bw_1 + \dots + \Bw_N.
\end{equation}

Define a $(dN+n)$-dimensional integer vector $(\Bb^{(N)},\Bc)$ as
\begin{equation}
\label{eq:brc}
(\Bb^{(N)},\Bc) = \begin{pmatrix} \Bb \\ \vdots \\ \Bb \\ \Bc
\end{pmatrix},
\end{equation}
where $\Bb$ is repeated $N$ times on the right-hand side. For $(\Bb^{(N)},\Bc)\in \N A^{(N)}$, an 
element of the fiber $\cF_{A^{(N)}, (\Bb^{(N)},\Bc)}$ of $A^{(N)}$ 
is written as $\Bw=(\Bw_1,\dots,\Bw_N)$, where 
$\Bw_k \in \cF_{A,\Bb}, k=1,\dots,N$,  and $\Bw_1 + \dots + \Bw_N = \Bc$. This is the same
as \eqref{eq:Ab-fiber}.  Hence any element of the fiber $\cF_{A_\Bb, \Bc}$ of $A_\Bb$ corresponds
to an element of the fiber $\cF_{A^{(N)}, (\Bb^{(N)},\Bc)}$ of $A^{(N)}$. This correspondence
between $\cF_{A_\Bb,\Bc}$ and $\cF_{A^{(N)}, (\Bb^{(N)},\Bc)}$  is one-to-one except
for the permutation of vectors $\Bw_1,\dots,\Bw_N$.  Note that the same $N$ 
$\Bb$'s on the right-hand side of \eqref{eq:brc} may be different for general fibers of $A^{(N)}$.
Hence the set of fibers $\N A_\Bb$ for $A_\Bb$ is a subset of the set of fibers
$\cup_{N\ge 1} \N A^{(N)}$. As discussed in \cite{hara-takemura-yoshida-logistic}, Markov bases
for a subset of fibers may be smaller than the full Markov bases. This fact is reflected in the
inequality in \eqref{eq:main-result}.

Now we give a proof of Theorem \ref{thm:main}.
\begin{proof}
Define a map 
$f_\bmb:\mathcal F_{A^{(N)},(\bmb^{(N)},\bm c)}\rightarrow\mathcal F_{A_\bmb,\bm c}$ by 
\begin{align*}
f_\bmb(\bm w)=\bm y=(y_1,\dots,y_{\fbsz{\Bb}}), %
\quad y_i = | \{k \mid \Bw_k = \Bx_i \}|.
\end{align*}
Then $f_\bmb$ is a surjection and furthermore 
\[
f_\bmb(\bm w)=\sum_{k=1}^N f_\bmb((\bm 0,\dots,\bm 0,\bm w_k,\bm 0,\dots,\bm 0)) = 
\sum_{k=1}^N (0,\dots,0,\underset{\scriptsize i:\bm x_i = \bm w_k}{1}, 0,\dots,0).
\]
For any $\bm y^{(s)}, \bm y^{(t)}\in\mathcal F_{A_\bmb,\bm c}$ we choose
\begin{align*}
\bm w^{(s)}\in f_\bmb^{-1}(\bm y^{(s)}), \quad \bm w^{(t)}\in f_\bmb^{-1}(\bm y^{(t)})
\end{align*}
and we connect  $\bm w^{(s)}$ and $\bm w^{(t)}$ by a Markov basis consisting of moves of type at most ${\rm MC}(A)$ of $A^{(N)}$.
Denote the path from $\bm w^{(s)}$ to $\bm w^{(t)}$ in $\cF_{A^{(N)},(\bmb^{(N)},\bm c)}$ as
\begin{align*}
\bm w^{(s)}=\bm w^{(0)}\rightarrow\bm w^{(1)}\rightarrow\cdots\rightarrow\bm w^{(T)}=\bm w^{(t)}.
\end{align*}
Let $\bm y^{(l)}=f_\bmb(\bm w^{(l)})$, $l=0,1\dots,T$.
Then 
\begin{align*}
A_\bmb \bm y^{(l)}=\sum_{i=1}^{\fbsz{\Bb}}y^{(l)}_i\bm x_i=\bm w_1^{(l)}+\cdots+\bm w_N^{(l)}=\bm c
\end{align*}
and %
$\bm y^{(l)}\in \mathcal F_{A_\bmb,\bm c}$. Hence 
$\bm y^{(l+1)}-\bm y^{(l)}$ is a move for  $A_\bmb$.  Its degree is bounded as
\begin{align*}
\frac{1}{2} \, |\bm y^{(l+1)}-\bm y^{(l)}|
&= \frac{1}{2} \, |f_\bmb(\bm w^{(l+1)})-f_\bmb(\bm w^{(l)})| \nonumber \\
&= \frac{1}{2} \, |\sum_{k:\bm w^{(l+1)}_k\neq\bm w^{(l)}_k}f_\bmb((\bm 0,\dots,\bm 0,\bm w^{(l+1)}_k,\bm 0,\dots,\bm 0))-f_\bmb((\bm 0,\dots,\bm 0,\bm w^{(l)}_k,\bm 0,\dots,\bm 0))| \nonumber
  \\
&\leq \frac{1}{2} \, \sum_{k:\bm w^{(l+1)}_k\neq\bm w^{(l)}_k}|f_\bmb((\bm 0,\dots,\bm 0,\bm w^{(l+1)}_k,\bm 0,\dots,\bm 0))-f_\bmb((\bm 0,\dots,\bm 0,\bm w^{(l)}_k,\bm 0,\dots,\bm 0))| \\
&= |\{ k \mid \bm w^{(l+1)}_k\neq \bm w^{(l)}_k\} | 
= {\rm type}(\bm w^{(l+1)}-\bm w^{(l)}) \\
&\le {\rm MC}(A).
\end{align*}
Thus $\bm y^{(s)}$ and $\bm y^{(t)}$
can be connected by moves of degree less than or equal to ${\rm MC}(A)$.
\end{proof}

In Theorem \ref{thm:main} an interesting question is when 
\eqref{eq:main-result} holds with equality.  At this point we give a simple but important example.
As the base configuration consider a $1\times n$ row vector $A=(1,1,\dots,1)$.  
Then for any positive integer $b$,
the fiber $A_b$ is the configuration of Veronese-type (Chapter 14 of \cite{sturmfels1996}), whose
Markov degree is two.  Hence $\max_{\Bb\in \N A} {\rm MD}(A_\Bb)=2$.  
On the other hand, $A^{(N)}$ is the configuration matrix of the 
complete bipartite graph $K_{n,N}$.  Since $A^{(N)}$, $N\ge 2$, has a Markov basis consisting of
moves of degree two, we have ${\rm MC}(A)=2$.  Hence the equality in \eqref{eq:main-result} holds
for this case.  Also note that  ${\rm GC}(A)=n$, since the elements of Graver basis corresponds
to cycles of $K_{n,N}$.

For bounding the Markov complexity ${\rm MC}(A)$ from below, we will find an indispensable move
for the higher Lawrence lifting $A^{(N)}$ of $A$.  The following proposition is useful for this purpose.
We use the notation $[N]=\{1,2,\dots,N\}$.

\begin{proposition}
\label{lem:indispensable-for-lawrence}
Let $\Bz=(\Bz_1,\dots,\Bz_N)$ be a move for $A^{(N)}$ such that
each slice $\Bz_k$ is a non-zero indispensable move for $A$.  Then
$\Bz$ is indispensable if and only if 
\[
\sum_{k\in M} \Bz_k \neq 0
\]
for every non-empty proper subset $M$ of $[N]$.
\end{proposition}
\begin{proof}  Write $\Bz$ by its positive part and negative part as $\Bz=\Bz^+ - \Bz^-$ and
let $\Bb^{(N)}= A^{(N)} \Bz^+$. $\Bz$ is an indispensable move if and only if
${\mathcal F}_{A^{(N)},\Bb^{(N)}}=\{ \Bz^+, \Bz^-\}$ is a two-element set.
Also write each slice $\Bz_k$ as $\Bz_k = \Bz_k^+ - \Bz^-_k$ and 
let $\Bb_k = A \Bz_k^+$. 
We are assuming that ${\mathcal F}_{A,\Bb_k}=\{ \Bz_k^+, \Bz_k^-\}$
is a two-element set for each $k$. 
Let $\Bx=(\Bx_1, \dots, \Bx_N)\in {\mathcal F}_{A^{(N)},\Bb^{(N)}}$.
Then $A \Bx_k=\Bb_k$ for each $k$ and hence $\Bx_k$ is either $\Bz_k^+$ or $\Bz_k^-$.
Let $M=\{ k \mid \Bx_k = \Bz^+_k\}$. %
Then $\Bx$ is different from both $\Bz^+$ and $\Bz^-$ 
if and only if $M$ is a non-empty proper subset of $[N]$.  
Now $\sum_{k=1}^N \Bx_k = \sum_{k=1}^N \Bz_k^- \ = \Bc$ (say) implies
\begin{equation}
\label{eq:lawrence-indispensable}
0=\sum_{k=1}^N (\Bx_k- \Bz_k^-) = \sum_{k\in M} (\Bz_k^+ - \Bz_k^-)=\sum_{k\in M} \Bz_k.
\end{equation}
Hence $\Bz$ is indispensable if and only if \eqref{eq:lawrence-indispensable} hold only for
$M=\emptyset$ or $M=[N]$.
\end{proof}

Note that $\sum_{k\in M} \Bz_k=0$ if and only if $\sum_{k\in M^C} \Bz_k=0$ and
any slice $k$ is either in $M$ or in $M^C$. Hence in order to prove that $\Bz$ is indispensable,
we can start from arbitrary  slice $\Bz_k$ and show that any sum of slices including $k$
does not vanish except for the sum of all slices.

\section{Complete graphs as base configurations}
\label{sec:complete-graphs}

In this section we study the maximum Markov degree and
the Markov complexity when the base configuration $A$ is an
incidence matrix of a small complete graph without
self-loops (Section \ref{subsec:complete-graph}) or
with self-loops (Section \ref{subsec:with-loops}).

In $\Bb = A \Bx$, %
the elements of $\Bx$ are the non-negative integer weights of the edges and 
the elements of $\Bb$ are degrees of vertices, where the
degree of a vertex $v$ is the sum of weights of the edges having $v$ as an endpoint. 
Note that one self-loop $\{v,v\}$ gives two degrees to the vertex $v$.
  
In the following,  by $\Bg$ we denote
a graph with non-negative weights attached to the edges.
The elements of a fiber ${\mathcal F}_{A,\Bb}$
are the graphs $\Bg$ with the same degree sequence $\Bb$.
See Figure \ref{fig:2-2-2-examples} below for an example.

Elements of a fiber ${\mathcal F}_{A_\Bb,\Bc}$ can be identified with  multisets of graphs $\Bg$ 
such that the sum of weights of each edge is common.
A move of degree $k$ for the configuration $A_\Bb$ corresponds to replacing $k$ graphs
$\Bg_1, \dots, \Bg_k \in {\mathcal F}_{A,\Bb}$ with $\hat \Bg_1, \dots, \hat \Bg_k \in {\mathcal F}_{A,\Bb}$ such that the sum of weights of each edge is preserved.

\subsection{Complete graph on  four vertices without self-loops}
\label{subsec:complete-graph}

In this section we take the incidence matrix of the complete graph $K_4$ on four vertices 
without self-loops as the  base configuration $A$.   At the end of this section 
we give some comments on larger complete graphs.  In particular we present
a conjecture on $K_5$.

Let 
\begin{equation}
\label{eq:k4-noloop}
A=
\begin{pmatrix}
1&1&1&0&0&0\\
1&0&0&1&1&0\\
0&1&0&1&0&1\\
0&0&1&0&1&1
\end{pmatrix}.
\end{equation}
We prove that both sides of \eqref{eq:main-result} are two and the equality holds for this $A$.

\begin{theorem}
\label{thm:k4-noloop}
For $A$ in \eqref{eq:k4-noloop}
\[
\max_{\Bb\in \N A} {\rm MD}(A_\Bb)  = {\rm MC}(A) = 2.
\]
\end{theorem}

By {\tt 4ti2}(\cite{4ti2}) we easily obtain ${\rm GC}(A) = 3$, which equals 
the maximum 1-norm of ${\mathcal G}({\mathcal G}(A))$ (Theorem 3 of \cite{santos2003higher}).

Denote the four vertices as $a,b,c,d$, corresponding to the rows of $A$.
There are six edges corresponding to the columns  of $A$.
Let $E=\{ab,ac,ad,bc,bd,cd\}$ denote the edge set.
A graph $\Bg$ is identified with a 6-dimensional non-negative integer vector
\[
\Bg=(g(ab),g(ac),g(ad),g(bc),g(bd),g(cd))\in \N^6,
\]
whose elements
represent weights of the edges. 
For two graphs $\Bg, \hat \Bg$ in the same fiber of $A$, we write $\Bz=\Bg-\hat \Bg=(z(ab),\dots,z(cd))$, 
which is a move for $A$.

We prove two lemmas.
\begin{lemma}
\label{lem:lemma1-4v-noloop}
Let $\Bg, \hat\Bg$ be graphs in the same fiber of $A$ and let $\Bz=\Bg-\hat \Bg$. 
Then
\[
z(ab)=z(cd),\ z(ac)=z(bd), \ z(ad)=z(bc).
\]
\end{lemma}

\begin{proof} By symmetry it suffices to prove %
$z(ab)=z(cd)$.
Let $\deg(a)$ denote the degree of vertex $a$. We have 
\begin{align*}
\deg(a)&= g(ab) + g(ac)+g(ad) =  \hat g(ab) + \hat g(ac) + \hat g(ad),\\
\deg(b)&= g(ab) + g(bc)+g(bd) =  \hat g(ab) + \hat g(bc) + \hat g(bd) .
\end{align*}
Hence
\begin{align*}
\deg(a)+\deg(b)&= 2 g(ab)+ g(ac)+g(ad)+g(bc)+g(bd)\\
&=2 \hat g(ab)+ \hat g(ac)+ \hat g(ad)+\hat g(bc)+ \hat g(bd).
\end{align*}
Similarly
\begin{align*}
\deg(c)+\deg(d)&= 2 g(cd)+ g(ac)+g(ad)+g(bc)+g(bd)\\
&=2 \hat g(cd)+ \hat g(ac)+ \hat g(ad)+\hat g(bc)+ \hat g(bd).
\end{align*}
Then
\[
\deg(a)+\deg(b)-(\deg(c)+\deg(d))=2(g(ab)-g(cd))=2(\hat g(ab)-\hat g(cd))
\]
and 
\[
g(ab)-\hat g(ab)= g(cd)-\hat g(cd).
\]
\end{proof}

\begin{lemma}
\label{lem:lemma2-4v-noloop}
Let $\Bg, \hat\Bg$ in the same fiber of $A$ 
and let $g(e_1)\neq \hat g(e_1)$ for some $e_1\in E$. 
Then there exists a loop  $(e_1,e_2,e_3,e_4)$ of  
length 4 passing each vertex, such that $g(e_i)\neq \hat g(e_i)$, $i=1,\dots,4$, and 
the signs of $g(e_i) - \hat g(e_i)$ alternate.
\end{lemma}
\begin{proof}
By symmetry we may assume that $e_1=ab$ and $g(ab)-\hat g(ab)>0$.  Then by the previous lemma 
$g(cd)-\hat g(cd)>0$.
Since $\deg(a)$ is common in $\Bg$ and $\hat \Bg$, by symmetry we may assume
that $g(ad)-\hat g(ad)<0$. Again by the previous lemma $g(bc)-\hat g(bc)<0$. 
Then $(ab,bc,cd,ad)$ is the required loop.
\end{proof}

We now give a proof of Theorem \ref{thm:k4-noloop}
based on the idea of  distance reduction (cf.\ Chapter 6 of \cite{aoki2012markov}).

\begin{proof}[Proof of Theorem \ref{thm:k4-noloop}]
Obviously $\max_{\Bb\in \N A} {\rm MD}(A_\Bb) > 1$. Hence by Theorem \ref{thm:main}
it suffices to prove that ${\rm MC}(A) = 2$. 
Let $\{ \Bg_1, \dots, \Bg_N \}$ and $\{\hat \Bg_1,\dots, \hat \Bg_N\}$ be two elements of the
same fiber for $A^{(N)}$. Let 
\[
S = \sum_{k=1}^N |\Bz_k|, \qquad \Bz_k = \Bg_k - \hat \Bg_k, 
\]
where $|\cdot |$ denotes the 1-norm of a 6-dimensional vector.

Suppose $S>0$.  By symmetry we may assume that
$\Bg_1 \neq \hat \Bg_1$. By  Lemma \ref{lem:lemma2-4v-noloop} we may assume
\[
z_1(ab)>0, \ z_1(bc)<0, \ z_1(cd)>0, \ z_1(ad)<0.
\]
Because $\{ \Bg_1, \dots, \Bg_N \}$ and  $\{\hat \Bg_1,\dots, \hat \Bg_N\}$ belong to the same fiber, we have \[
\sum_{k=1}^N z_k(e)=0 
\]
for each $e\in E$ (in particular for $e=bc$).
Hence there exits $k$ such that  $z_k(bc)>0$.
Let $k=2$ without loss of generality.
By Lemma \ref{lem:lemma1-4v-noloop} $g_2(ad) > \hat g_2(ad)$.   Let 
\begin{equation}
\label{eq:standard-edge-basis}
\Be_{ab}=(1,0,0,0,0,0)
\end{equation}
denote the graph with weight 1 only on the edge $ab$.
Similarly define $\Be_{bc}, \Be_{cd}, \Be_{ad}$.  Now consider the move
\begin{equation}
\label{eq:exchange-edges-two}
(\Bg_1, \Bg_2) \rightarrow (\Bg_1 +\Be_{ab}- \Be_{bc} + \Be_{cd}-\Be_{ad}, 
\Bg_2 -\Be_{ab} + \Be_{bc} - \Be_{cd} + \Be_{ad}).
\end{equation}
Then the vectors on the right-hand side are non-negative and $S$ is strictly decreased.  This proves
${\rm MC}(A) = 2$. 
\end{proof}

\begin{remark}
\label{rem:ohsugi}
Hidefumi Ohsugi gave a simple direct proof of 
$\max_{\Bb\in \N A} {\rm MD}(A_\Bb) =2$ by identifying $A_\Bb$ with a
Segre--Veronese configuration.
\end{remark}

The move in \eqref{eq:exchange-edges-two} can be understood as an {\em exchange} or {\em swap}
of edges between
two  graphs $\Bg_1, \Bg_2$, i.e., edges $bc$ and $ad$ are given from $\Bg_1$ to $\Bg_2$, 
and edges $ab$ and $cd$ are taken from $\Bg_2$ to $\Bg_1$. A move of degree two
for $A_\Bb$ and a move of type two for $A^{(N)}$ is an exchange of edges 
between two graphs.  
Similarly 
a move of degree $k$
for $A_\Bb$ and a move of type $k$ for $A^{(N)}$ is an exchange of edges 
among $k$ graphs.

At this point, we make some remarks on larger complete graphs without self-loops.
Consider the complete graph $K_5$ of five vertices without self-loops and
let 
\begin{equation}
\label{eq:5-noloop}
A=
\begin{pmatrix}
1 & 1 & 1 & 1 & 0 & 0 & 0 & 0 & 0 & 0\\
1 & 0 & 0 & 0 & 1 & 1 & 1 & 0 & 0 & 0\\
0 & 1 & 0 & 0 & 1 & 0 & 0 & 1 & 1 & 0\\
0 & 0 & 1 & 0 & 0 & 1 & 0 & 1 & 0 & 1\\
0 & 0 & 0 & 1 & 0 & 0 & 1 & 0 & 1 & 1
\end{pmatrix}
\end{equation}
be its incidence matrix.
By {\tt 4ti2} we can check 
\[
{\rm MC}(A) \geq 6, \quad  {\rm GC}(A)=15.
\]

Concerning $\max_{\Bb\in \N A} {\rm MD}(A_\Bb)$ we make the following conjecture.
\begin{conjecture}
For $A$ in \eqref{eq:5-noloop}
\begin{equation}
\label{eq:k5-conjecture}
\max_{\Bb\in \N A} {\rm MD}(A_\Bb) =2.
\end{equation}
\end{conjecture}

\begin{table}[htbp]
\caption{Number of moves in minimal Markov bases for $A_{\Bb}$ in the case of $K_5$}
\label{tab:k5}
\begin{center}
{\footnotesize
\begin{tabular}{|c|c|c|} \hline
$\Bb$ & \# moves of deg 2 & \# moves of deg 3\\ \hline
(2,2,2,1,1)&9&0\\
(2,2,2,2,2)&95&0\\
(3,2,2,2,1)&39&0\\
(3,3,2,1,1)&9&0\\
(3,3,2,2,2)&16&0\\
(3,3,3,2,1)&105&0\\
(3,3,3,3,2)&741&0\\
(4,2,2,2,2)&105&0\\
(4,3,2,2,1)&39&0\\
(4,3,3,1,1)&9&0\\
(4,3,3,2,2)&413&0\\
(4,3,3,3,1)&225&0\\
(4,3,3,3,3)&1893&0\\
(4,4,2,1,1)&9&0\\
(4,4,2,2,2)&216&0\\
(4,4,3,2,1)&105&0\\
(4,4,3,3,2)&1179&0\\
(4,4,4,2,2)&710&0\\
(4,4,4,3,1)&420&0\\
(4,4,4,3,3)&4032&0\\
(4,4,4,4,2)&2718&0\\
(4,4,4,4,4)&10581&0\\ \hline
\end{tabular}
}
\end{center}
\end{table}

Our conjecture is based on the of numbers of moves of degrees two and three in minimal
Markov bases for various $A_{\Bb}$ in Table \ref{tab:k5} computed with {\tt 4ti2}.

For the case $K_6$ of 6 vertices, 
we can easily check that $\max_{\Bb\in \N A} {\rm MD}(A_\Bb) \ge 4$.

\subsection{Complete graph on three vertices with self-loops}
\label{subsec:with-loops}

We consider the incidence matrix of the complete graph on three vertices with self-loops
as the base configuration $A$:
\begin{align}
\label{eq:3-loop}
A=
\begin{pmatrix}
2 & 1 & 1 & 0 & 0 & 0\\
0 & 1 & 0 & 2 & 1 & 0\\
0 & 0 & 1 & 0 & 1 & 2
\end{pmatrix}.
\end{align}
The following theorem holds.
\begin{theorem}
\label{thmmd3}
For $A$ in \eqref{eq:3-loop}
\begin{equation}
\label{eq:3-loop-result}
\max_{\Bb\in \N A} {\rm MD}(A_\Bb)  = 3,\quad
 {\rm MC}(A) = 5.
\end{equation}
Furthermore $\max_{\Bb\in \N A \setminus \{(2,2,2)\}}{\rm MD}(A_\Bb)  = 2$.
\end{theorem}

Incidentally we obtained ${\rm GC}(A)=8$ by {\tt 4ti2} (\cite{4ti2})
and Theorem 3 of \cite{santos2003higher}.

As stated in Theorem \ref{thmmd3}, the fiber with 
$\bm b=(2,2,2)$ is special. %
${\mathcal F}_{A,(2,2,2)}$ consists of five vectors and $A_{(2,2,2)}$ is given as
\[
A_{(2,2,2)} = \begin{pmatrix}
 1 &  1 &  0 &  0 &  0 \\
 0 &  0 &  2 &  1 &  0 \\
 0 &  0 &  0 &  1 &  2 \\
 1 &  0 &  0 &  0 &  1 \\
 0 &  2 &  0 &  1 &  0 \\
 1 &  0 &  1 &  0 &  0   
\end{pmatrix}
=(\Ba_1, \Ba_2, \Ba_3, \Ba_4, \Ba_5).
\]
Columns of $A_{(2,2,2)}$ are displayed in 
Figure \ref{fig:2-2-2-examples}.
\begin{figure}[htbp]
\begin{center}
\includegraphics[width=40mm]{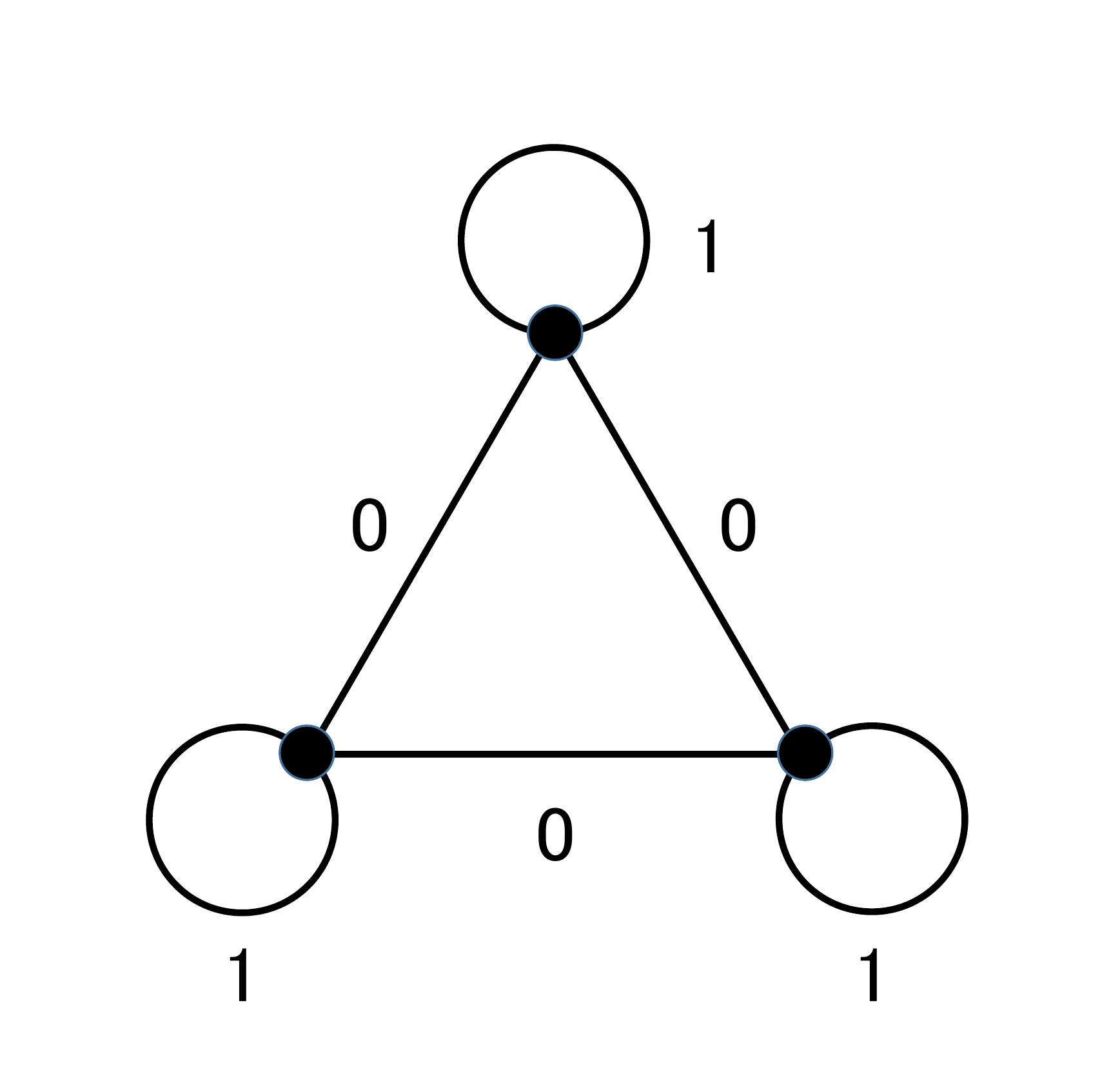} 
\includegraphics[width=40mm]{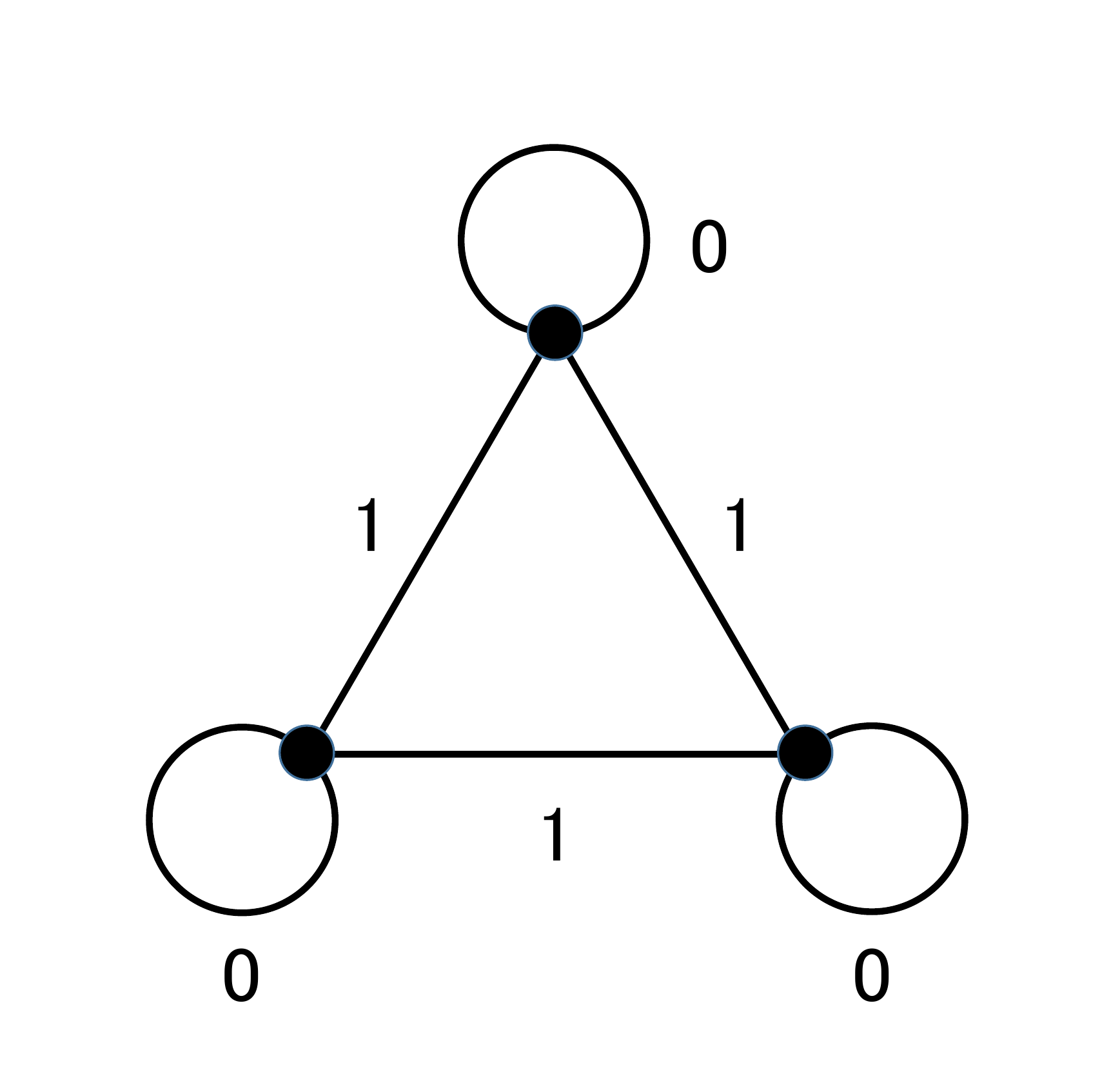} 
\includegraphics[width=40mm]{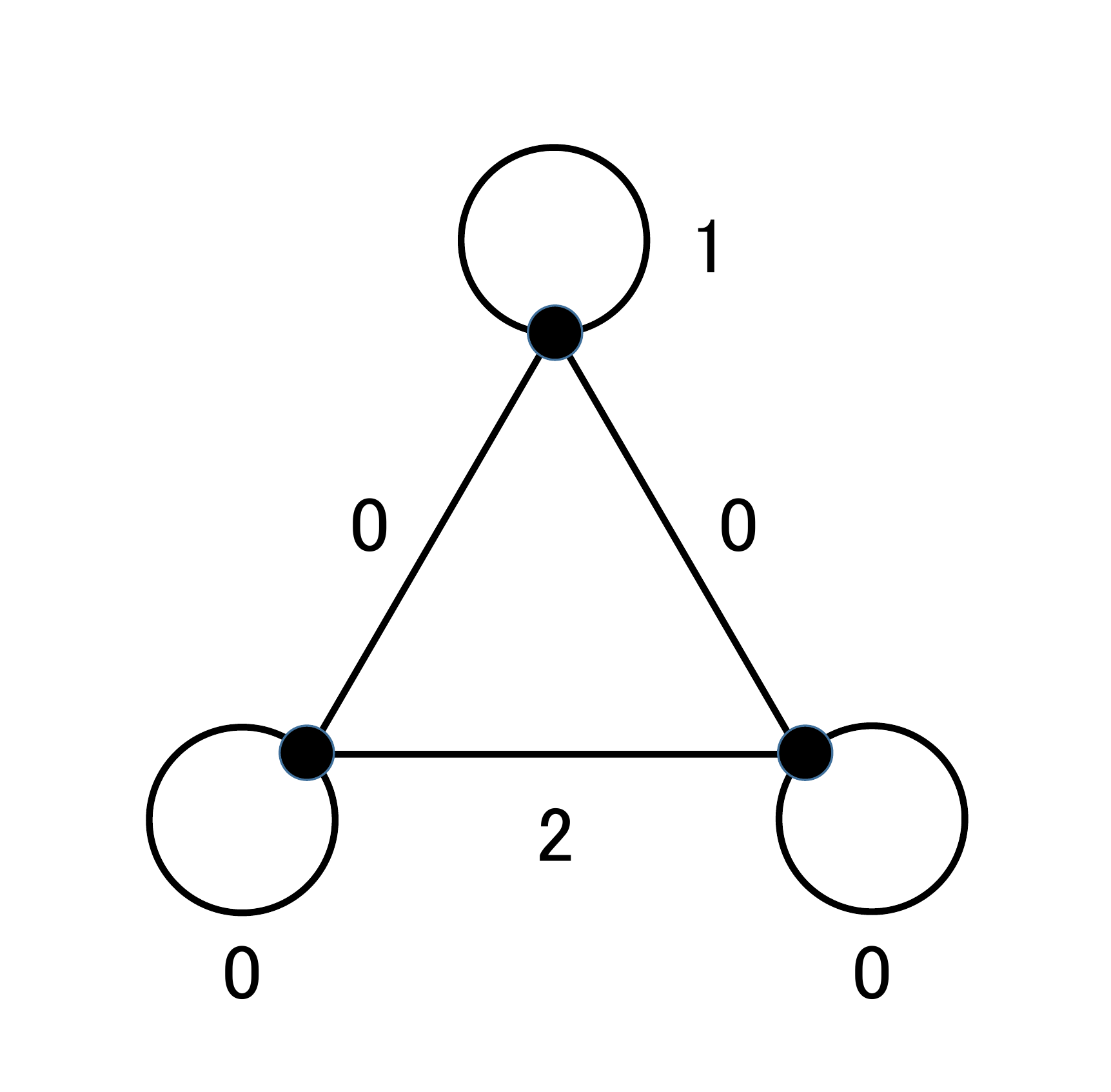}\\
Column 1 of $A_{(2,2,2)}$ \quad  Column 4 of $A_{(2,2,2)}$ \quad  Columns 2,3,5  of $A_{(2,2,2)}$
\caption{Graphs of the fiber ${\mathcal F}_{A,(2,2,2)}$}
\label{fig:2-2-2-examples}
\end{center}
\end{figure}

In this case $\rank A_{(2,2,2)}=4$ and the toric ideal  $I_{A_{(2,2,2)}}$ associated with $A_{(2,2,2)}$ is a principal ideal generated by the relation
\[
\Ba_1 + 2 \Ba_4 = \Ba_2 + \Ba_3 + \Ba_5.
\]
Hence 
\begin{equation}
\label{eq:a222}
{\rm MD}(A_{(2,2,2)})=3.
\end{equation}
\begin{figure}[htbp]
\begin{center}
   \includegraphics[width=2cm]{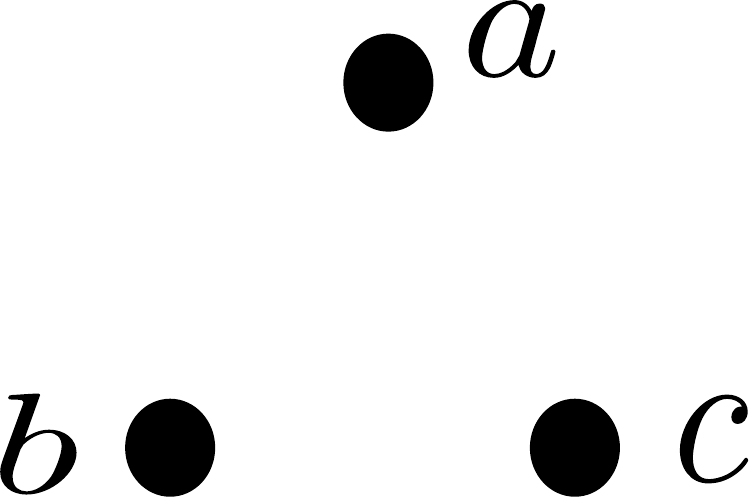}
 \caption{Labels of vertices}
\label{fig:vertex}
\end{center}
\end{figure}
To express vertices and edges, we label the vertices as Figure \ref{fig:vertex}.
Then for example we express the self-loop $\{a,a\}$ by an edge $aa$.

\bigskip
For the rest of this subsection we give a proof of Theorem \ref{thmmd3}.

It is easy to see that ${\rm MD}(A_\Bb)\le 2$ if $\min(\deg(a),\deg(b),\deg(c))\le 1$.
Hence from now on we assume that the degrees of three vertices are at least two.
For our proof we utilize the Graver basis ${\mathcal G}(A)$ of $A$ in \eqref{eq:3-loop}.  By {\tt 4ti2} (\cite{4ti2}) or by 
checking the moves for $A$, it is easily verified that 
${\mathcal G}(A)$ consists of ten column vectors in \eqref{eq:graver3a} and those with the minus sign. Hence $|{\mathcal G}(A)|=20$.  
There are four patterns of moves and  patterns $\bm B$ and $\bm C$ are indispensable moves.

\begin{equation}
\label{eq:graver3a}
\begin{array}{c|cccccccccc}
   & \bm A & \bm B(a) & \bm B(b) & \bm B(c) & \bm C(a) & \bm C(b) & \bm C(c) & \bm D(a) & \bm D(b) & \bm D(c) \\ \hline
aa &  1 &  1 & 0  & 0  & 0 & 1& 1 &  0 & 1 &-1\\
ab & -1 & -1 & -1 & 1  & 0 & 0&-2 &  2 & -2& 0\\
ac & -1 & -1 & 1  & -1 & 0 &-2& 0 & -2 & 0 & 2\\
bb &  1 & 0  & 1  & 0  & 1 & 0& 1 & -1 & 0 & 1\\
bc & -1 & 1  &-1  & -1 &-2 & 0& 0 & 0  & 2 &-2\\
cc & 1  & 0  & 0  & 1  & 1 & 1& 0 & 1  & -1& 0
\end{array}
\end{equation}

By using the notation in \eqref{eq:standard-edge-basis}, the move $\bm A$ is written as
\[
\bm A=\Be_{aa} + \Be_{bb} + \Be_{cc} - \Be_{ab} - \Be_{bc}-\Be_{ac}.
\]
We denote 20 moves of ${\mathcal G}(A)$ by $\bm A,\bm B(a),\dots,\bm D(c)$ and $-\bm A,-\bm B(a), \dots, -\bm D(c)$.
Moves $\bm A$, $\bm B(a)$, $\bm C(a)$, $\bm D(a), \bm D(b), \bm D(d)$ are displayed in Figure \ref{graver_basis_of_A1}.  
\begin{figure}[htbp]
  \begin{center}
  \includegraphics[width=40mm]{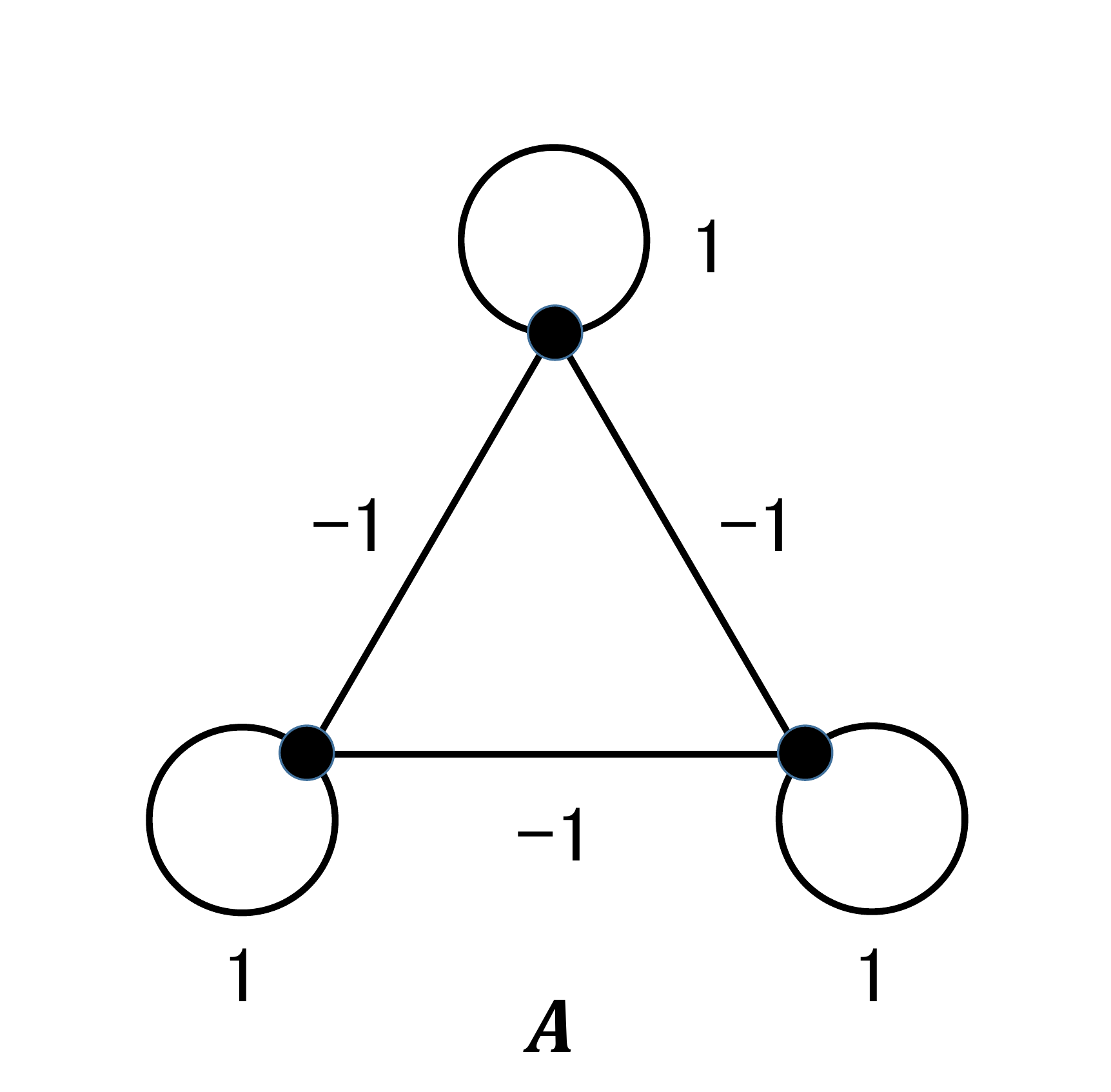}    
  \includegraphics[width=40mm]{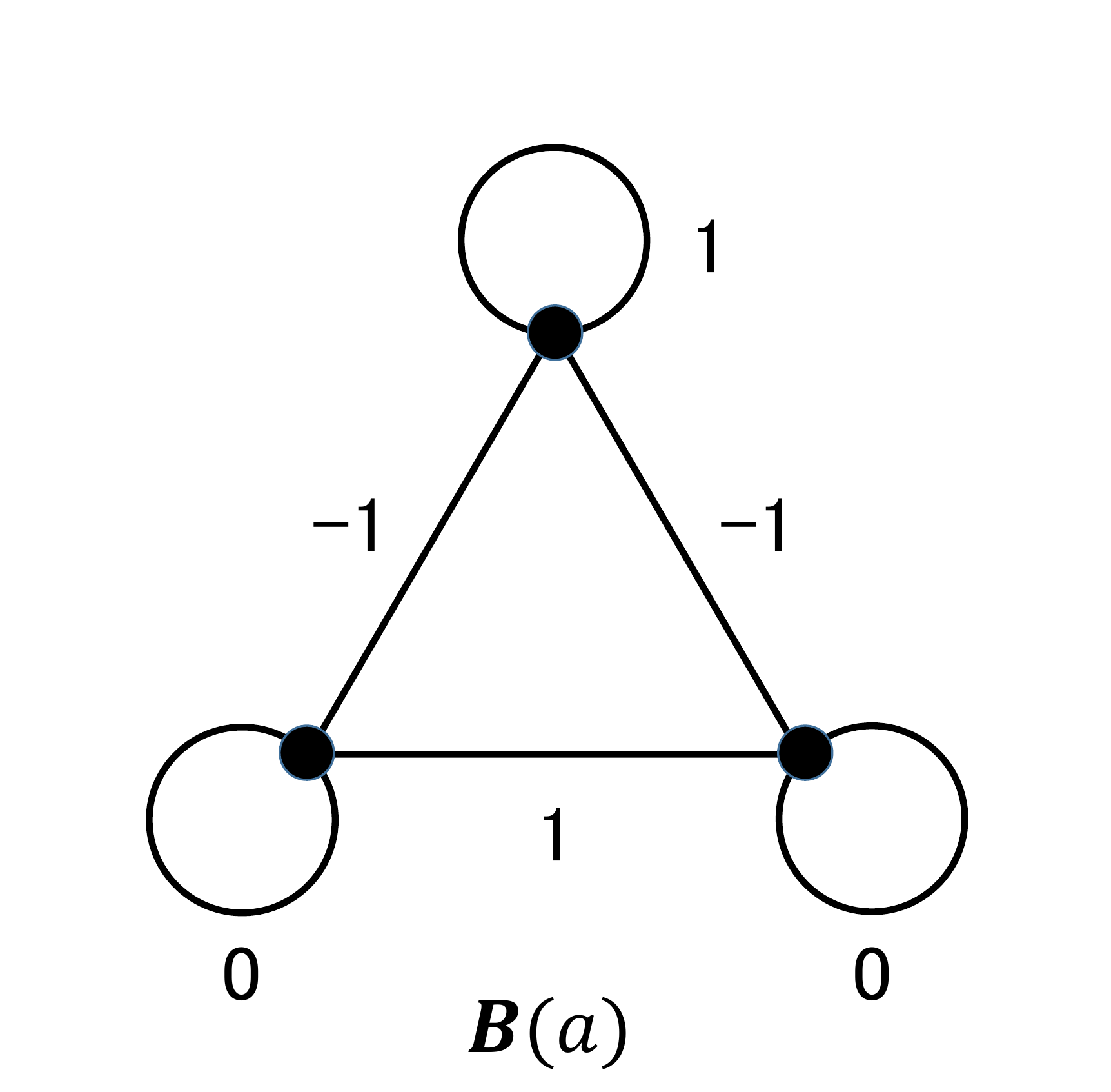} 
  \includegraphics[width=40mm]{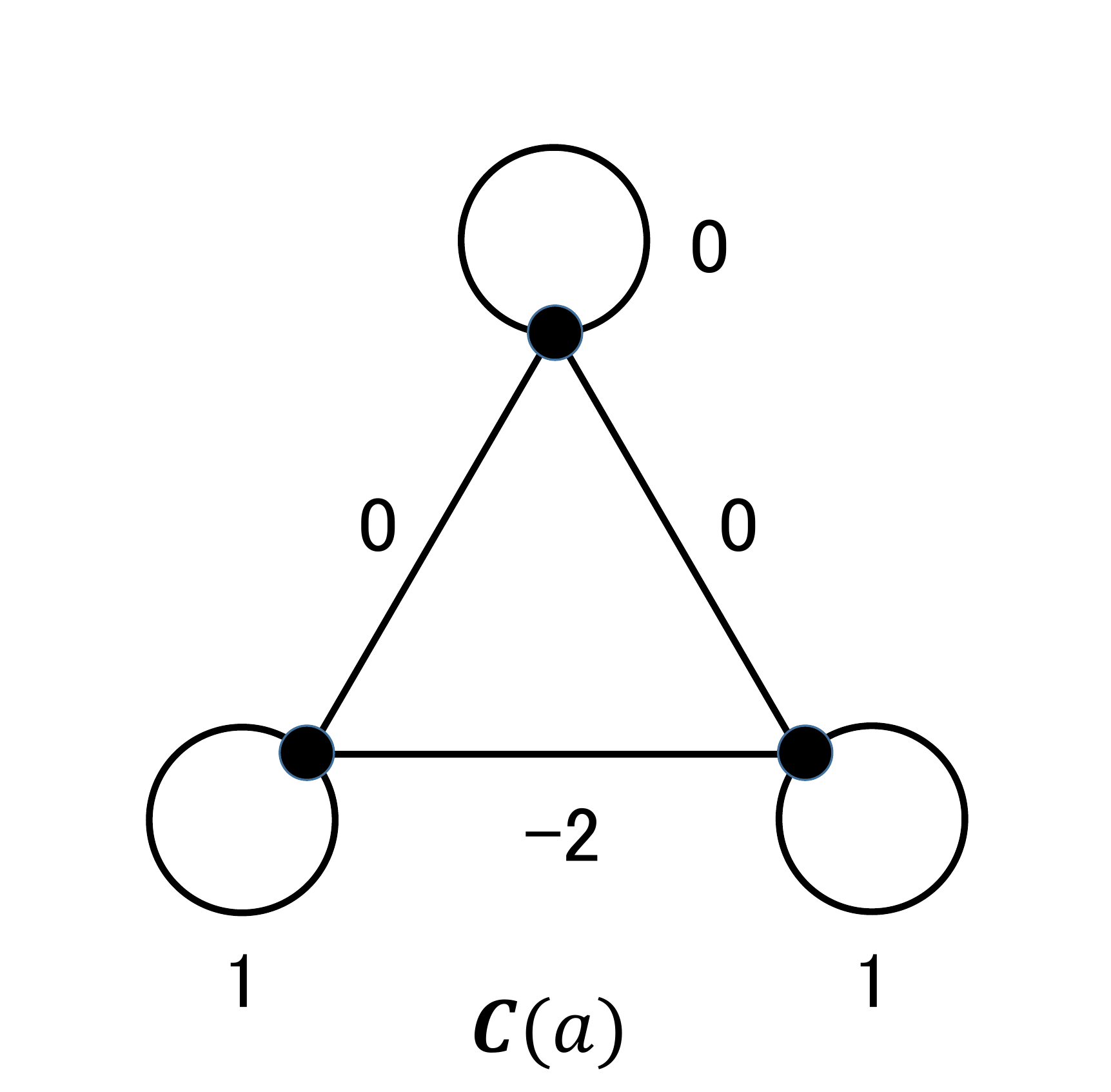} 
  \includegraphics[width=40mm]{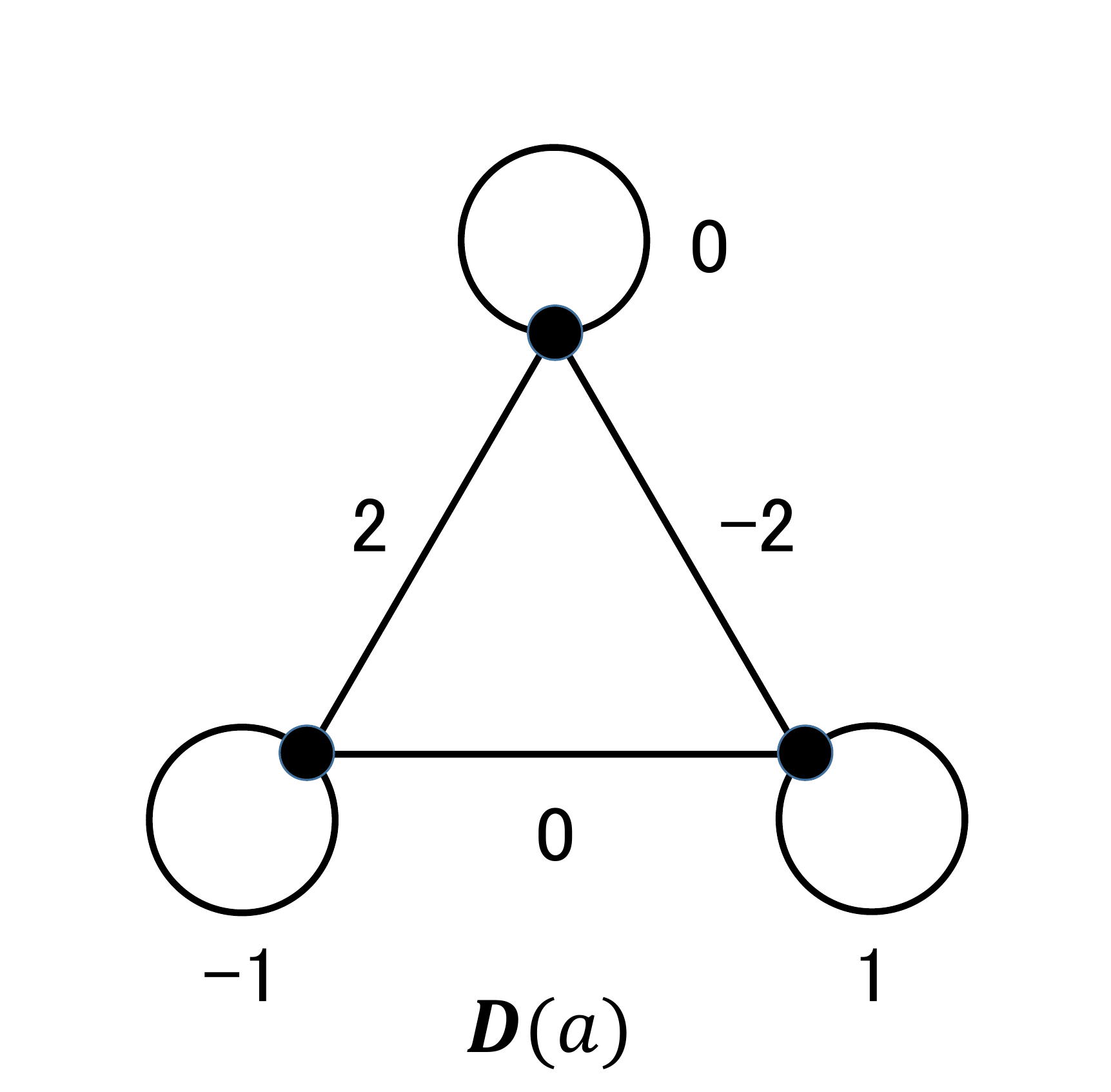} 
  \includegraphics[width=40mm]{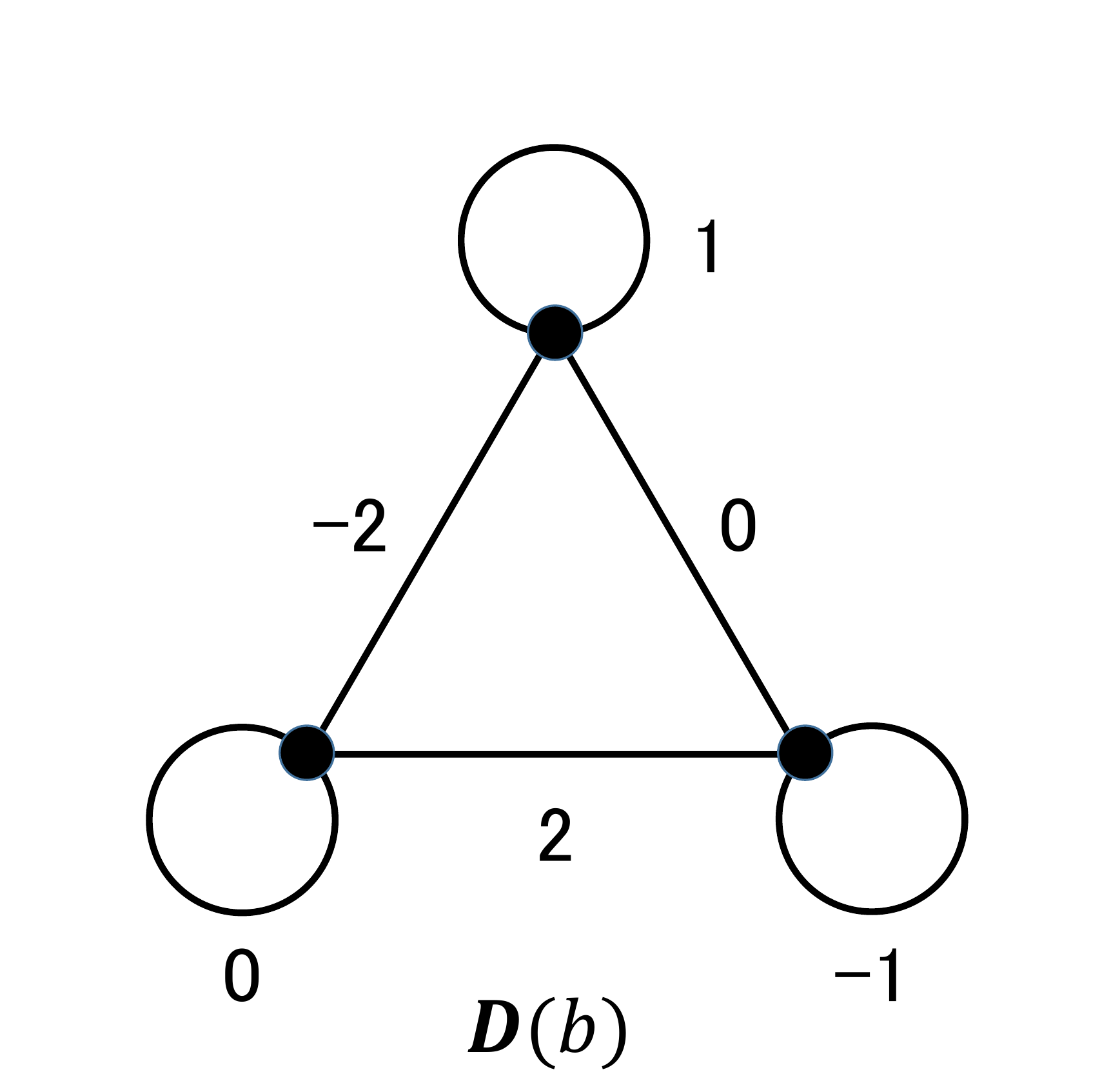}
  \includegraphics[width=40mm]{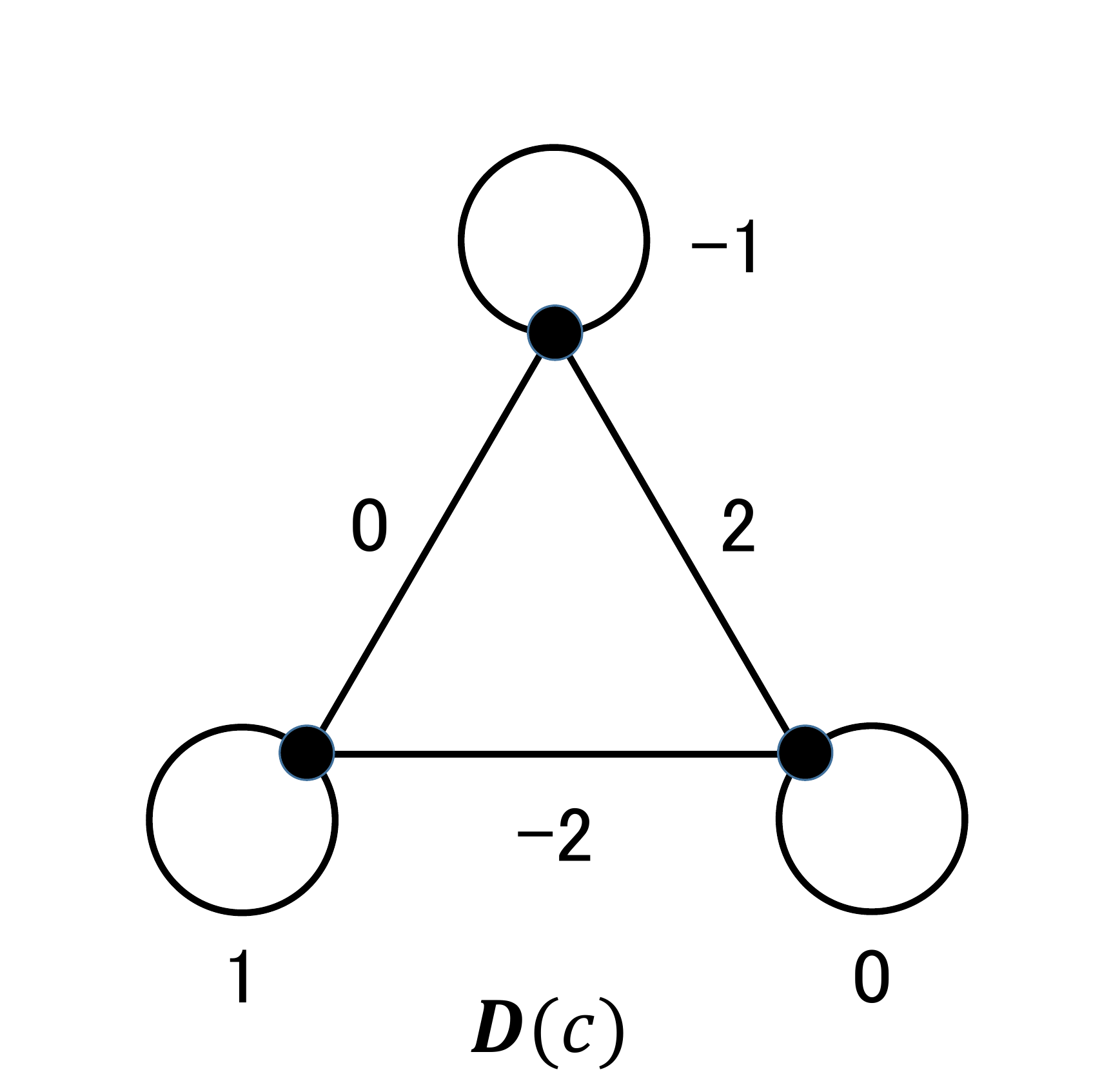}
  \caption{Moves $\bm A$, $\bm B(a)$, $\bm C(a)$, $\bm D(a)$, $\bm D(b)$, $\bm D(c)$}
  \label{graver_basis_of_A1}
  \end{center}
\end{figure}
For checking our proof of Theorem \eqref{thmmd3} it is convenient to have graphs for
$\bm B(b)$, $\bm B(c)$, $\bm C(b)$, $\bm C(c)$ in Figure \ref{graver_basis_of_A1a}.
\begin{figure}[hbtp]
  \begin{center}
  \includegraphics[width=35mm]{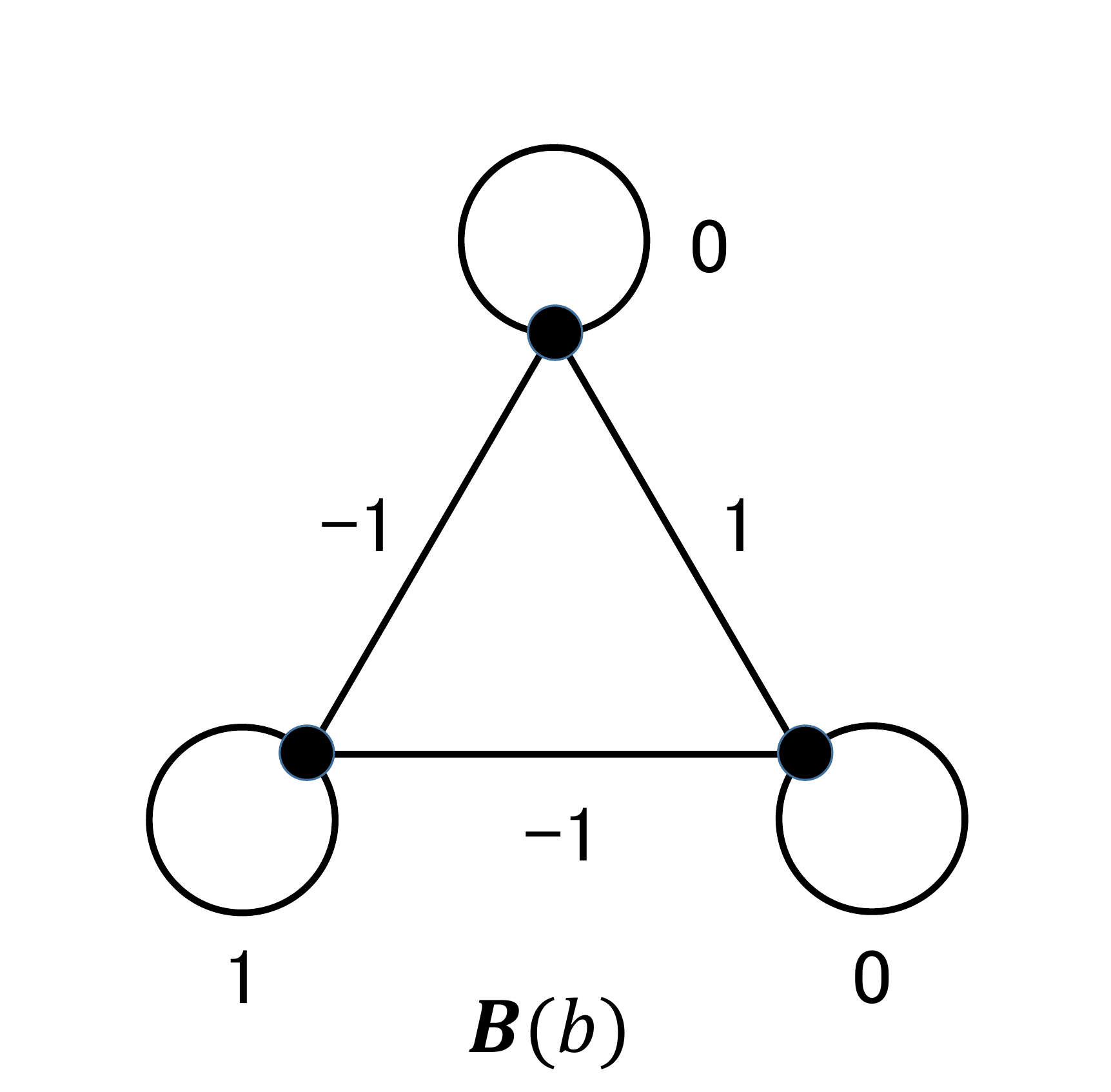}
  \includegraphics[width=35mm]{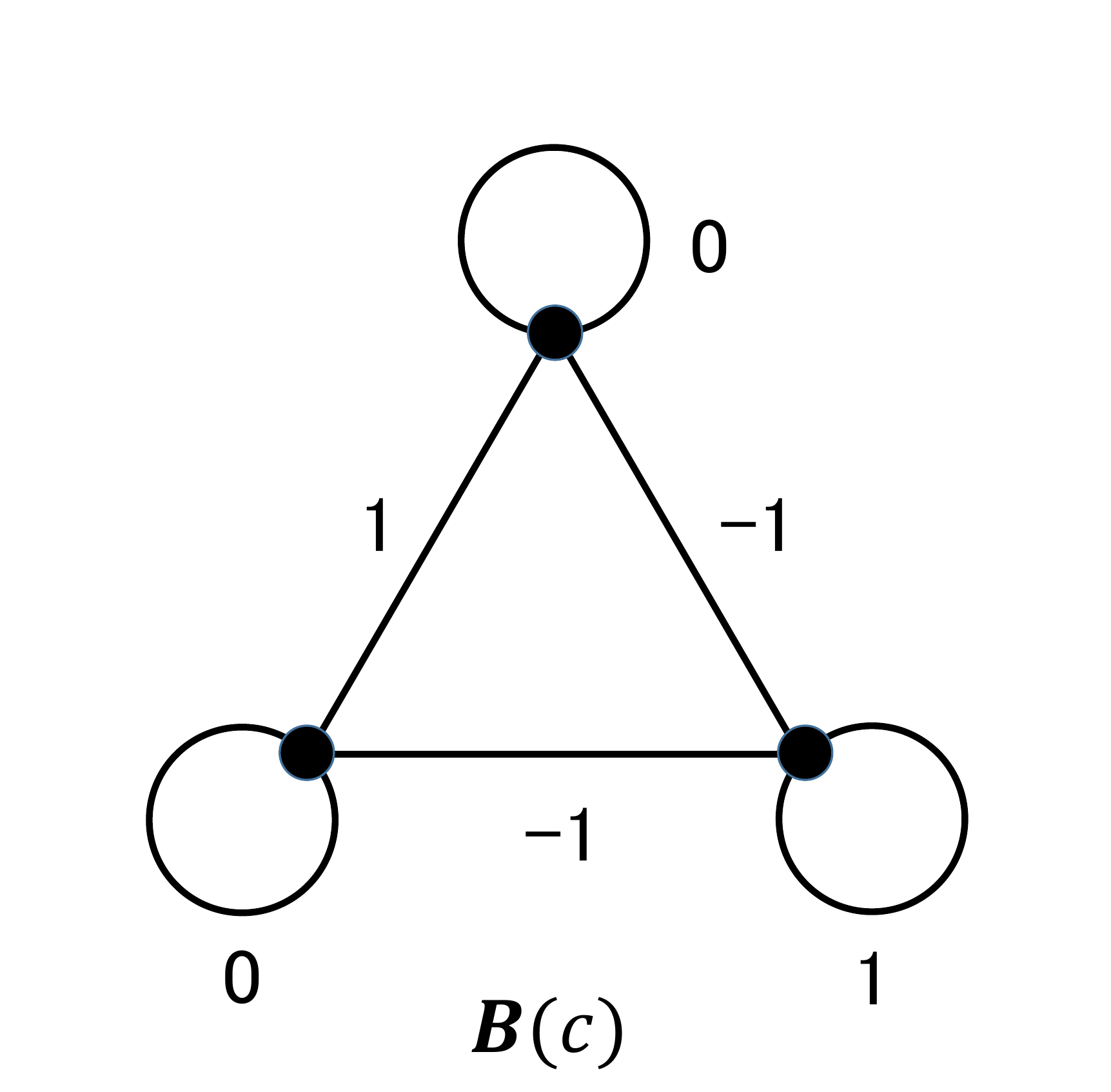} \quad 
  \includegraphics[width=35mm]{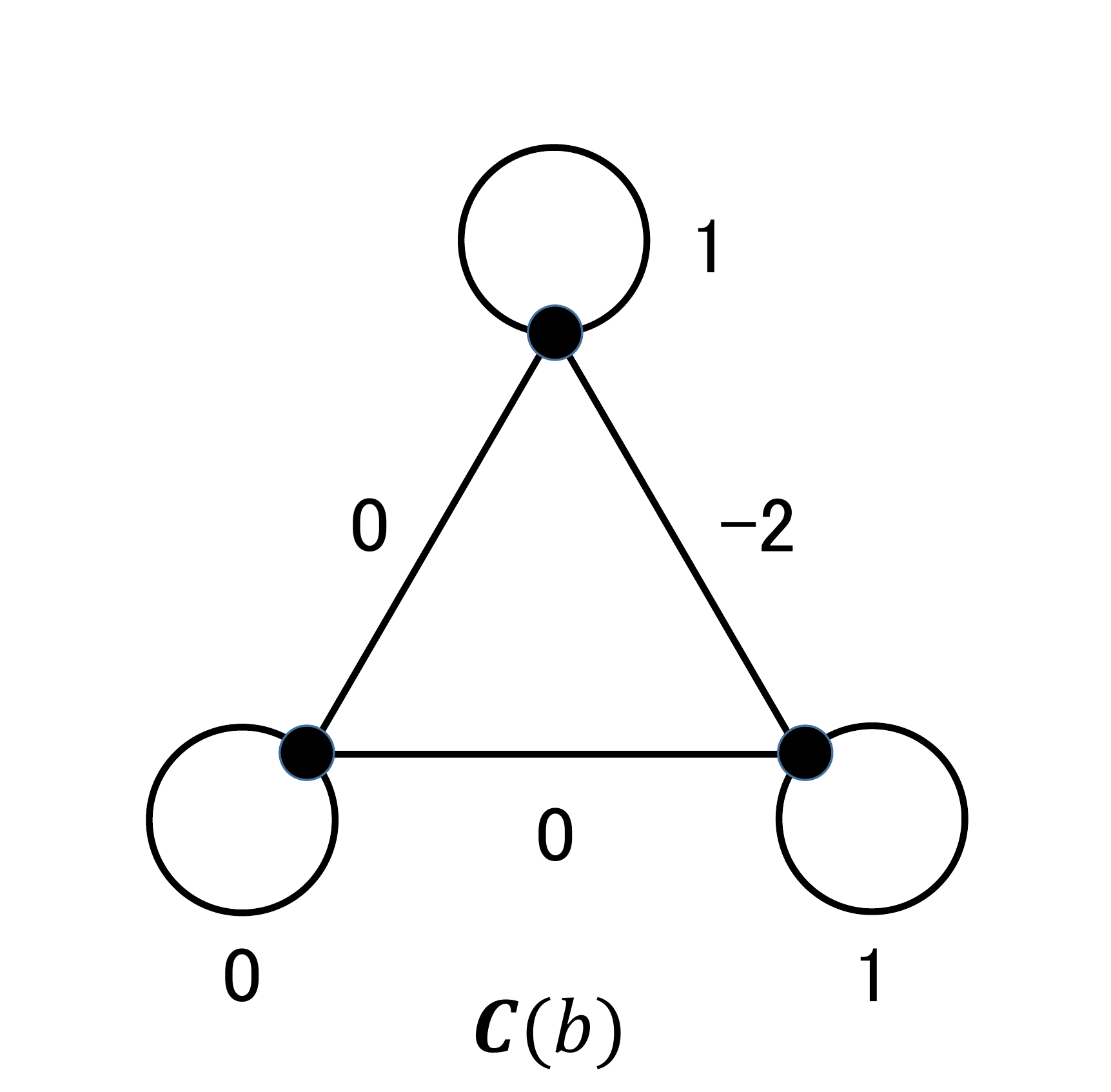}
  \includegraphics[width=35mm]{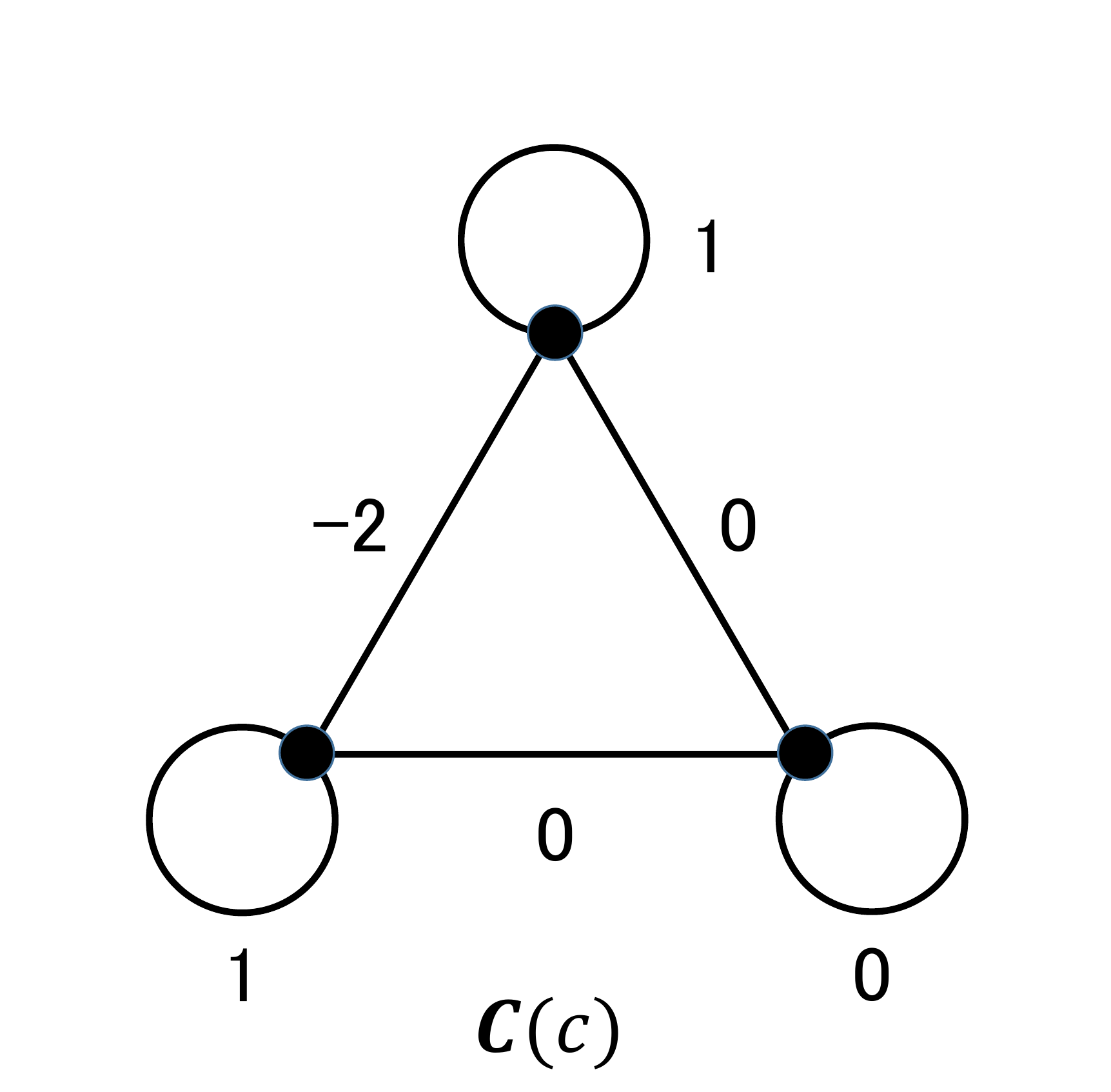}
  \caption{Moves $\bm B(b)$, $\bm B(c)$, $\bm C(b)$, $\bm C(c)$}
  \label{graver_basis_of_A1a}
  \end{center}
\end{figure}

For a move $\Bz \in \kerz A\subset \Z^6$, $\Bz\neq 0$, there exists $\Bw\in {\mathcal G}(A)$ such that
$\Bw + (\Bz- \Bw)=\Bz$
is a conformal sum, i.e., there is no cancellation of signs in this sum.  In this
case we write
\[
\Bw \sqsubseteq \Bz.
\]
Here we are allowing the case $\Bz = \Bw$.

Let $\Bg, \hat \Bg \in \N^6$ be two graphs in the same fiber $\cF_{A,\Bb}$ of $A$. 
Then $\Bz = \Bg-\hat\Bg$
is a move and there exists $\Bw\in {\mathcal G}(A)$ such that $\Bw \sqsubseteq \Bg-\hat\Bg$.
In this case we say that ``$(\Bg, \hat\Bg)$ {\em contains} $\bm w$''.  
Note that $(\Bg,\hat\Bg)$ contains $\bm w$ if and only if
$(\hat\Bg,\Bg)$ contains $-\bm w$. Also if
$(\Bg,\hat\Bg)$ contains $\bm w$ then $\bm g - \bm w \ge 0$ (elementwise) and
\[
|(\bm g - \bm w) - \hat{\bm g}| = |\bm g - \hat\Bg | - |\bm w|.
\]
When $(\Bg,\hat\Bg)$ contains $\bm w$, we denote $\Bg$ by $\Bg_{\Bw}$, provided that 
there is no confusion about $\hat \Bg$.
For example $\Bg_{-\bm A}$ denotes a graph $\Bg$ in $(\Bg, \hat \Bg)$ which contains the negative
of the first column of \eqref{eq:graver3a}.
Now we begin proving $\max_{\Bb\in \N A} {\rm MD}(A_\Bb)  = 3$.

\medskip
\noindent{I. \  Proof of $\max_{\Bb\in \N A} {\rm MD}(A_\Bb)  = 3$}. 

We choose two arbitrary elements of $\mathcal F_{A_{\bm b}, \bm c}=\{\bm y \mid A_{\bm b} \bm y=\bm c\}$ and denote them by $\bm y$ and $\hat{\bm{y}}$.
Although $\bm y$ and $\hat {\bm y}$ are multisets of graphs, by the embedding of a fiber of $A_\Bb$
into a fiber of $A^{(N)}$ discussed after Theorem \ref{thm:main}, 
we index the
graphs of $\bm y$ as $\bm g_{1},\bm g_2,\dots,\bm g_{N}$ and
graphs of $\hat {\bm y}$ as $\hat {\bm g}_1,\hat{\bm g}_2,\dots,\hat{\bm g}_{N}$.
Then %
\begin{align}
&A \bm g_k=A \hat{\bm g}_k=\bm b,  \quad k=1,\dots,N,   \label{omomiwa=b},\\
&\bm g_{1}+\bm g_2+\dots+\bm g_{N}=\hat{\bm g}_1+\hat{\bm g}_2+\dots+\bm{\hat g}_N=\bm{c}
\label{omomiwa=c}.
\end{align}
As in the proof of Theorem \ref{thm:k4-noloop}, let
\begin{align*}
S=\sum_{k=1}^{N}|\bm g_k-\hat{\bm g}_k|.
\end{align*}
Then, $S=0$ implies $\bm y=\hat{\bm y}$.
We will show that if $S>0$ 
there exists an exchange of edges among some fixed number graphs in $\bm y$  or in $\hat{\bm y}$
such that $S$ is decreased.

If $S>0$, there exist a layer $k$ satisfying $\bm g_k\neq\hat{\bm g}_k$.
By $A (\bm g_k-\hat{\bm g}_k)=0$, there exists $\bm w\in {\mathcal G}(A)$ such that 
$(\bm g_k,\hat\Bg_k)$ contains $\Bw$. 
In this case we say that there exists a {\em pattern} $\Bw$ among $\Bz_k=\Bg_k - \hat \Bg_k$, $k=1,\dots,N$.
For example, suppose that the pattern $\bm B(a)$ exists.  Then for some $k$, $z_k(aa)>0$ and $z_k(ab)<0$.
By $0=\sum_{k=1}^N \Bz_k$, 
there have to be some other layers $k', k''$ such that $z_{k'}(aa)<0$ and $z_{k''}(ab)>0$.
In this case we say that the edge $aa$ is ``in shortage'' and the edge $ab$ is ``in excess''
on some layers other than $k$.

At this point we consider an easy case to decrease $S$, where there are $\bm g_{\bm A}$ and $\bm g_{-\bm A}$, i.e., there are $k$ and $k'$ such that $(\bm g_k,\hat \Bg_k)$ contains the move $\bm A$ and 
$(\bm g_{k'},\hat \Bg_{k'})$ contains the move $-\bm A$.
Then we can apply an exchange of edges 
$(\Bg_{\bm A}, \Bg_{-\bm A}) \rightarrow (\Bg_{\bm A}', \Bg_{-\bm A}')$,
where 
\begin{align*}
\bm g_{A}'&=\bm g_{A}-\bm e_{aa}-\bm e_{bb}-\bm e_{cc}+\bm e_{ab}+\bm e_{bc}+\bm e_{ca}=\bm g_{A}-\bm A,\\
\bm g_{-A}'&=\bm g_{-A} +\bm e_{aa}+\bm e_{bb}+\bm e_{cc} -\bm e_{ab}-\bm e_{bc}-\bm e_{ca}=\bm g_{-A}+\bm A.
\end{align*}
By this degree-two move  (\ref{omomiwa=b}) and  (\ref{omomiwa=c}) are conserved.
Obviously $\bm g_{A}'$ and $\bm g_{-A}'$ are non-negative and
$S$ is immediately decreased.
Similar consideration applies to other nine pairs of moves
$(\bm B(a), -\bm B(a))$, $\dots$, $(\bm D(c), -\bm D(c))$.
Therefore, from now on, we ignore the case that there are two layers containing 
any of these 10 pairs.
Also note that by symmetry between $\bm y$ and $\hat{\bm y}$, we only need to consider
one of $\bm A$ or $- \bm A$.

We now distinguish various cases.
We first consider the case that the pattern $\bm A$ (or $-\bm A$) exists.
\begin{description}
\item[Case 1] $\bm A$ exists.\\
We are assuming that there exists some $k$ such that $(\Bg_k,\hat \Bg_k)$ contains $\bm A$. 
There are three subcases depending on whether the pattern  $\bm B$ exists or not on some other layer $k'\neq k$.
By symmetry among $a, b, c$, we only need to consider $\bm B(a)$.

\begin{description}
\item[Case 1-1] $\bm B(a)$ exists.\\
Because of the existence of $\bm A$ and $\bm B(a)$, the edge $aa$ is in shortage 
on some other layer $k''$.
The possible patterns are $-\bm C(c), -\bm C(b), -\bm D(b)$, or $\bm D(c)$.
By symmetry between $b$ and $c$, we only need to consider $-\bm C(c)$ or $-\bm D(b)$.
If $-\bm C(c)$ exists then $S$ is decreased by
\[
\bm g_{\bm A}'=\bm g_{\bm A}-\bm C(c), \qquad 
\bm g_{-\bm C(c)}'=\bm g_{-\bm C(c)}+\bm C(c)
\]
and if $-\bm D(b)$ exists then $S$ is decreased by
\[
\bm g_{\bm A}'=\bm g_{\bm A}-\bm C(c), \qquad 
\bm g_{-\bm D(b)}'=\bm g_{-\bm D(b)}+\bm C(c).
\]
Here note that $\bm g_{-\bm D(b)}+\bm C(c)\ge 0$. We omit this kind of remark on non-negativity
for the rest this proof.
\item[Case 1-2]  $-\bm B(a)$ exists.\\
In this case we look at $\hat \By$.  $S$ is decreased by 
\begin{align*}
\hat \Bg_{-A}'=\hat \Bg_{-A}+\bm B(a), \qquad
\hat \Bg_{B(a)}'=\hat \Bg_{B(a)}-\bm B(a).
\end{align*}

\item[Case 1-3]  None of $\bm B(a)$, $-\bm B(a)$ exists.\\
This case can be handled as in Case 1-1, since the edge $aa$ is in shortage.
\end{description}
\end{description}

From now on, we assume that pattern $\pm \bm A$ does not exist.
For Case 2, we consider the existence of the pattern $\pm \bm D$.

\begin{description}
\item[Case 2] $\bm D$ exists.\\
By symmetry we consider the case that there is some layer containing $\bm D(a)$.
Since there is $\bm D(a)$, the edge $bb$ is in excess on some other layer.
The possible patterns for this excess are  $\bm C(a), \bm C(c), \bm D(c)$, or $\bm B(b)$.
Also the edge $cc$ is in shortage.
The possible patterns for this shortage are $-\bm C(a), -\bm C(b), \bm D(b)$, or $-\bm B(c)$.

Note that we are assuming that $\bm C(a)$ and $-\bm C(a)$ do not simultaneously exist, i.e., 
at least one of $\bm C(a)$ and $-\bm C(a)$ does not exist.  If $\bm C(a)$ does not exist,
then at least one of $\bm C(c), \bm D(c)$, or $\bm B(b)$ exist.  Similarly if 
$-\bm C(a)$ dos not exist at least one of $-\bm C(b), \bm D(b)$ or $-\bm B(c)$ exist.
Hence at least one of $\bm C(c), \bm D(c), \bm B(b), -\bm C(b), \bm D(b)$, $-\bm B(c)$ exist.

Now by simultaneous symmetry $(b,\By,\bm D(a)) \leftrightarrow (c,\hat \By,-\bm D(a))$,  
we only need to consider one of $\bm C(c)$ and $-\bm C(b)$, one of $\bm D(c)$ and $\bm D(b)$,
and one of $-\bm B(c)$ and $\bm B(b)$.  Hence we will examine the cases
$\bm C(c)$,  $\bm D(c)$, $-\bm B(c)$, in turn.
\begin{description}
\item[Case 2-1] $\bm D(a)$ and $\bm C(c)$ exist.\\
$S$ is decreased by
\begin{align*}
\bm g_{\bm D(a)}'=\bm g_{\bm D(a)}+\bm C(c),\qquad 
\bm g_{\bm C(c)}'=\bm g_{\bm C(c)}-\bm C(c).
\end{align*}

\item[Case 2-2] $\bm D(a)$ and $\bm D(c)$ exist.\\
$S$ is decreased by 
\begin{align*}
\bm g_{\bm D(a)}'=\bm g_{\bm D(a)}-\bm D(a),\qquad
\bm g_{\bm D(c)}'=\bm g_{\bm D(c)}+\bm D(a).
\end{align*}

\item[Case 2-3]$\bm D(a)$ and $-\bm B(c)$ exist.\\
$S$ is decreased by 
\begin{align*}
\bm g_{\bm D(a)}'=\bm g_{\bm D(a)}-\bm B(c),\qquad 
\bm g_{-\bm B(c)}'=\bm g_{-\bm B(c)}+\bm B(c).
\end{align*}

\end{description}
\end{description}

We have now examined all possible cases where $\pm \bm D$ exists.
From now on, we may assume that pattern $\pm \bm D$ does not exist.

We now consider the case that the pattern $\pm \bm B$ exists.
\begin{description}
\item[Case 3]$\bm B$ exists.\\
By symmetry we assume that $\bm B(a)$ exists.
Then because of the shortage of  $aa$ on  other layers,
there exists pattern $-\bm C(c)$ or $-\bm C(b)$.
Because of symmetry of vertices  $b$ and $c$, it is enough to consider $-\bm C(c)$ only.
Then because of the excess of $bb$, there exists pattern $\bm B(b)$ or $\bm C(a)$.

\begin{description}
\item[Case 3-1]$\bm B(a)$, $-\bm C(c)$ and $\bm B(b)$ exist.\\
$S$ is decreased by 
\begin{equation*}
\bm g_{B(a)}'=\bm g_{B(a)}-\bm B(a),\qquad
\bm g_{-C(c)}'=\bm g_{-C(c)}+\bm C(c),\qquad 
\bm g_{B(b)}'=\bm g_{B(b)}-\bm B(b).
\end{equation*}

\item[Case 3-2] $\bm B(a)$, $-\bm C(c)$ and $\bm C(a)$ exist.\\
Note that already $\bm D(c)$ and $-\bm D(a)$ do not exist by our assumption.
Also in the previous case we considered the existence of $\bm B(b)$.
Hence here we consider the case that  $\bm D(c)$, $-\bm D(a)$ and $\bm B(b)$ do not exist, but
$\bm C(a)$ exists. Then by the shortage of $cc$, 
there is a pattern $-\bm C(b)$ or $-\bm B(c)$.

\begin{description}
\item[Case 3-2-1]$\bm B(a)$, $-\bm C(c)$, $\bm C(a)$ and $-\bm C(b)$ exist.\\
This case is difficult.  We renumber this case as Case 4 and will discuss this case below.

\item[Case 3-2-2] $\bm B(a)$, $-\bm C(c)$, $\bm C(a)$ and $-\bm B(c)$ exist.\\
This case is also difficult.  We renumber this case as Case 5 and will discuss this case below.
\end{description}
\end{description}
\end{description}

So far we did not use the fact that all graphs $\Bg_1, \dots, \Bg_N$ belong to the same fiber ${\mathcal F}_{A,\Bb}$ of $A$.  
Our argument before Case 3-2-1 apply not only to $A_{\Bb}$, but also
to the higher Lawrence lifting $A^{(N)}$.   However there is a gap between two sides of  \eqref{eq:3-loop-result}. In order to show the left-hand side
$\max_{\Bb\in \N A} {\rm MD}(A_\Bb)  = 3$ we need to use that fact that $\Bg_1, \dots, \Bg_N$ belong to the same fiber.

We now look at Case 4 from this viewpoint.

\begin{description}
\item[Case 4] $\bm B(a)$, $-\bm C(c)$, $\bm C(a)$ and $-\bm C(b)$ exist.\\
First note that the existence $\bm B(a)$ implies $\deg(a)\ge 2$.
Also the existence of $\bm C(a)$ implies $\deg(b)\ge 2$, $\deg(c)\ge 2$.
Hence the degree of each vertex is at least two.  Then $\bm g_{\bm B(a)}$ has additional edges
connecting to $b$ and to $c$. The possible combinations of edges are
\begin{quote}
1) $bc$ alone, \ 2) the pair $(ab, ac)$, \ 3) the pair  $(bb,ac)$, \\
4) the pair $(bc, cc)$, \  5) the pair $(bb, cc)$, or \  6) the case that   
$\Bg_{\bm B(a)}$ has two $\bm B(a)$.
\end{quote}
These six cases are depicted in Figure \ref{B(a)}.
Existence of an additional edge is shown as the weight of the form $+p-q$ in Figure \ref{B(a)}.
$+p$ means that we can subtract $p$ edges without producing a negative weight.

Consider the edge $aa$. The weight of $aa$ in $\bm C(a)$ is zero. On the other hand
in both $-\bm C(c)$ and $-\bm C(b)$ its weight is $-1$.  This extra shortage of $aa$
implies that there exists another pattern $\bm B(a)$ in addition to the already existing $\bm B(a)$, possibly on the same layer as the already existing one or on another layer.
The former case corresponds to 6) above.  

Also note that $-\bm C(c)$ and $-\bm C(b)$ may be on the same layer, but in this case
the weight of the self-loop $aa$ on the layer is less than or equal to $-2$ and our proof is
not affected.

\begin{figure}[htbp]
  \begin{center}
\begin{tabular}{ccc}
 \includegraphics[width=40mm]{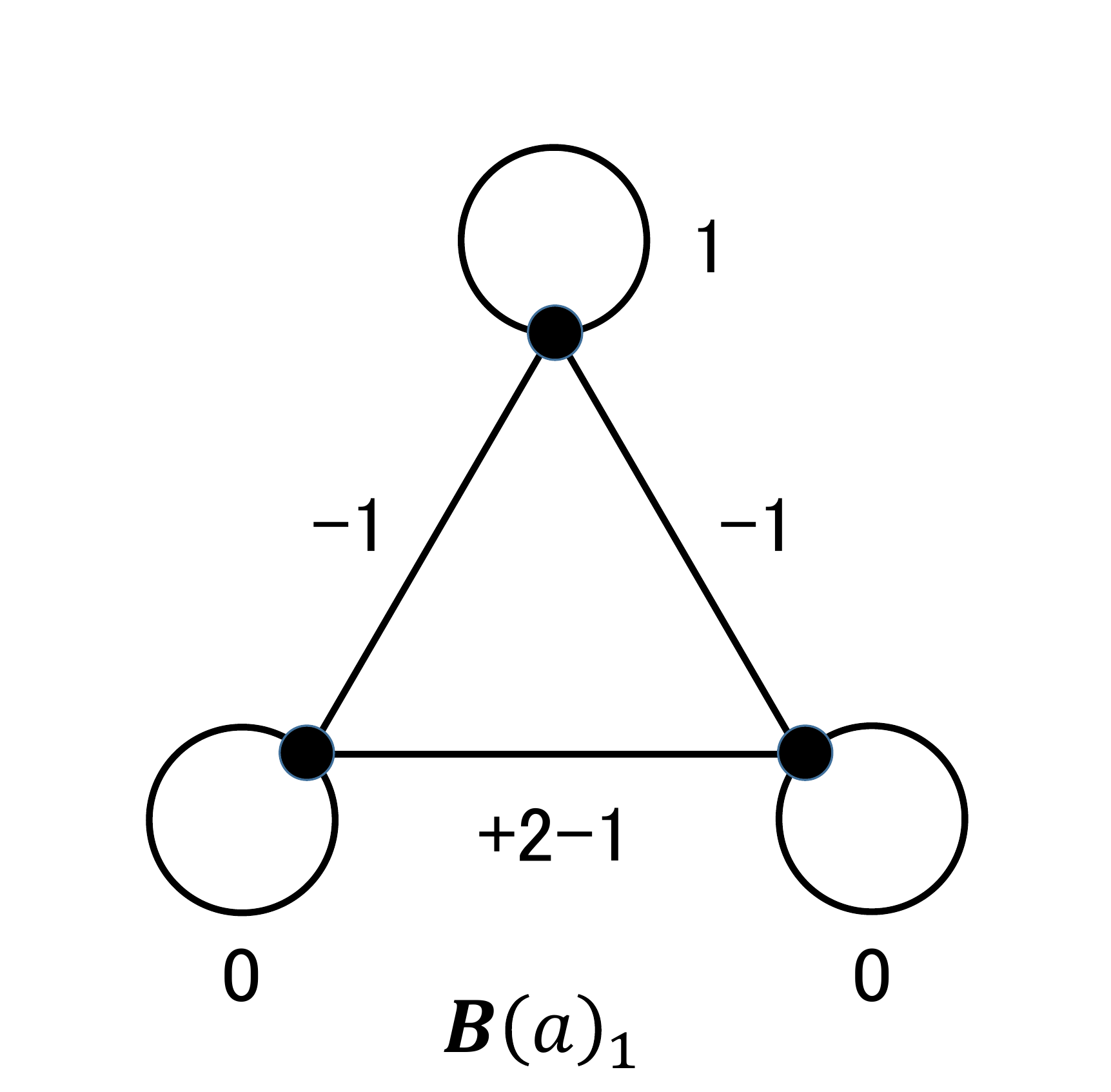}&
  \includegraphics[width=40mm]{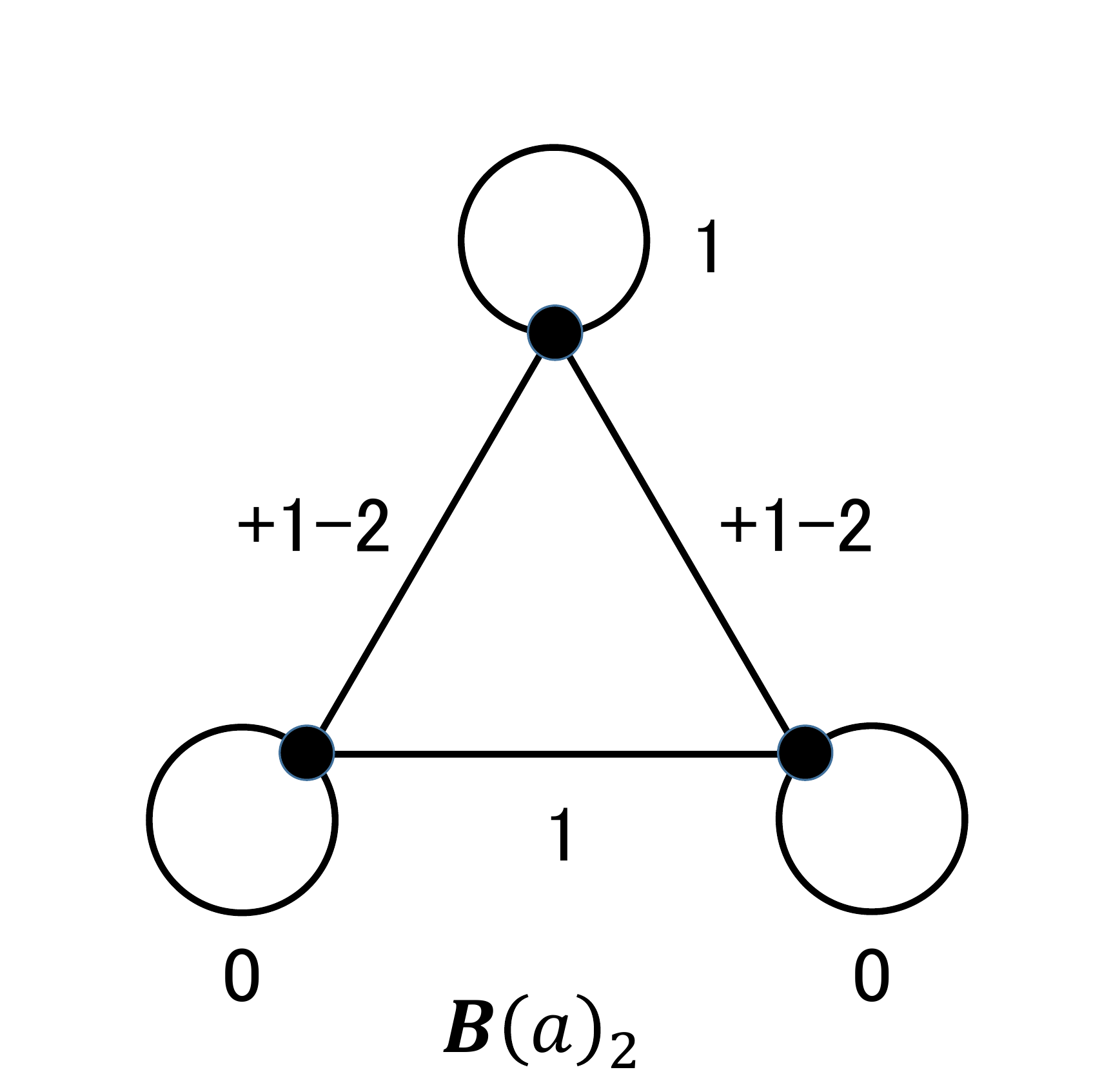}&
  \includegraphics[width=40mm]{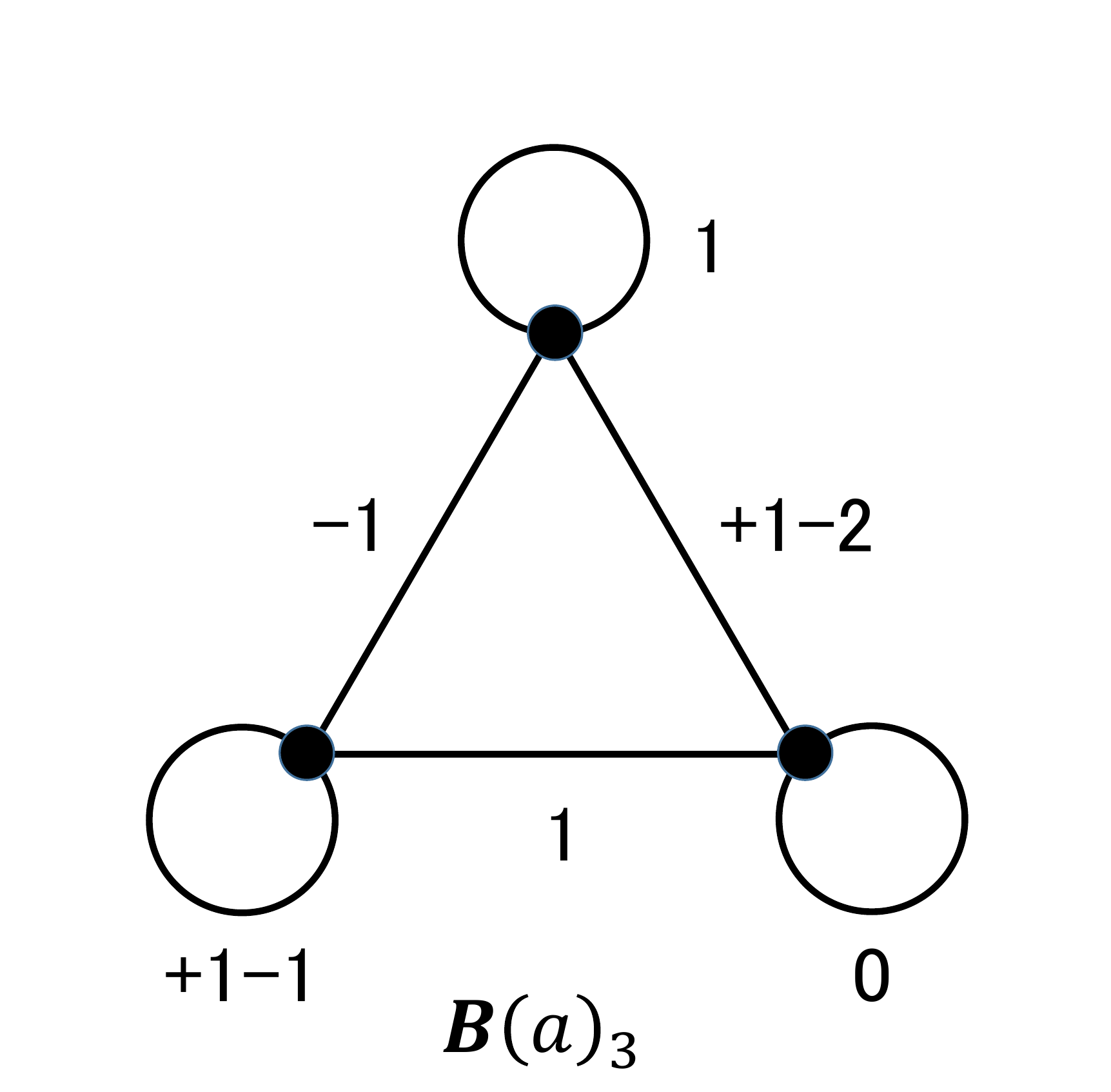}\\
   \includegraphics[width=40mm]{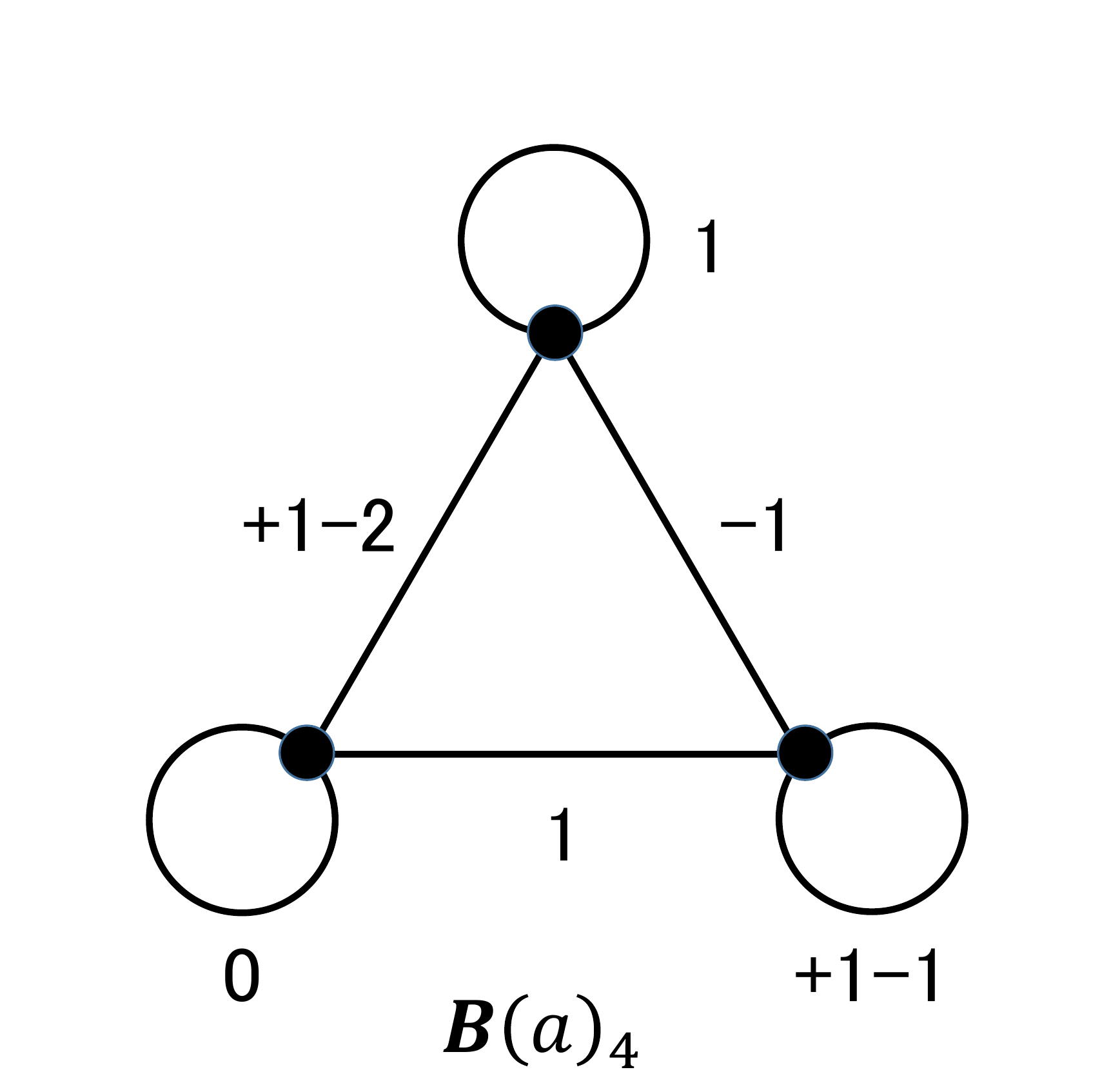}&
  \includegraphics[width=40mm]{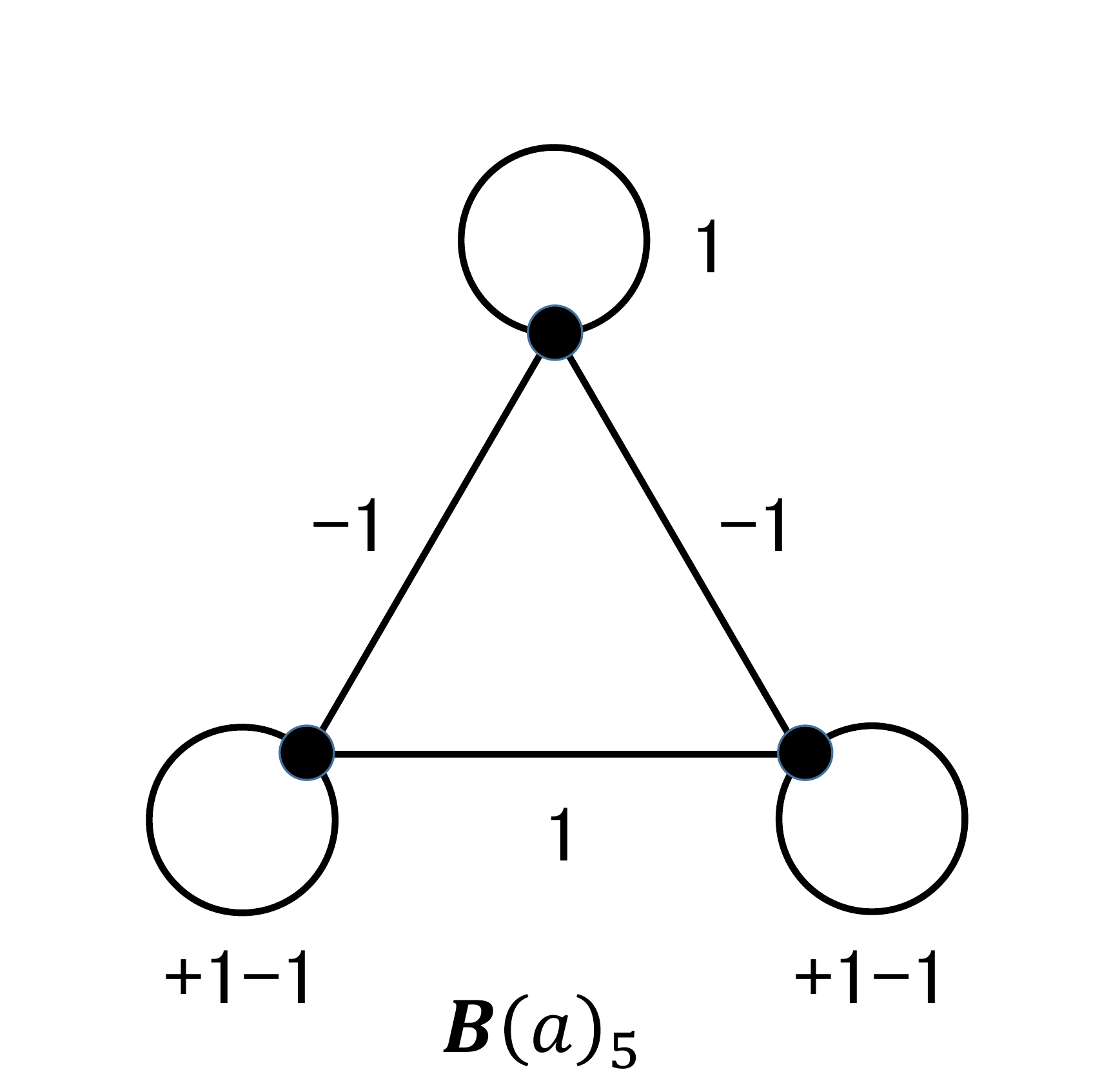}&
   \includegraphics[width=40mm]{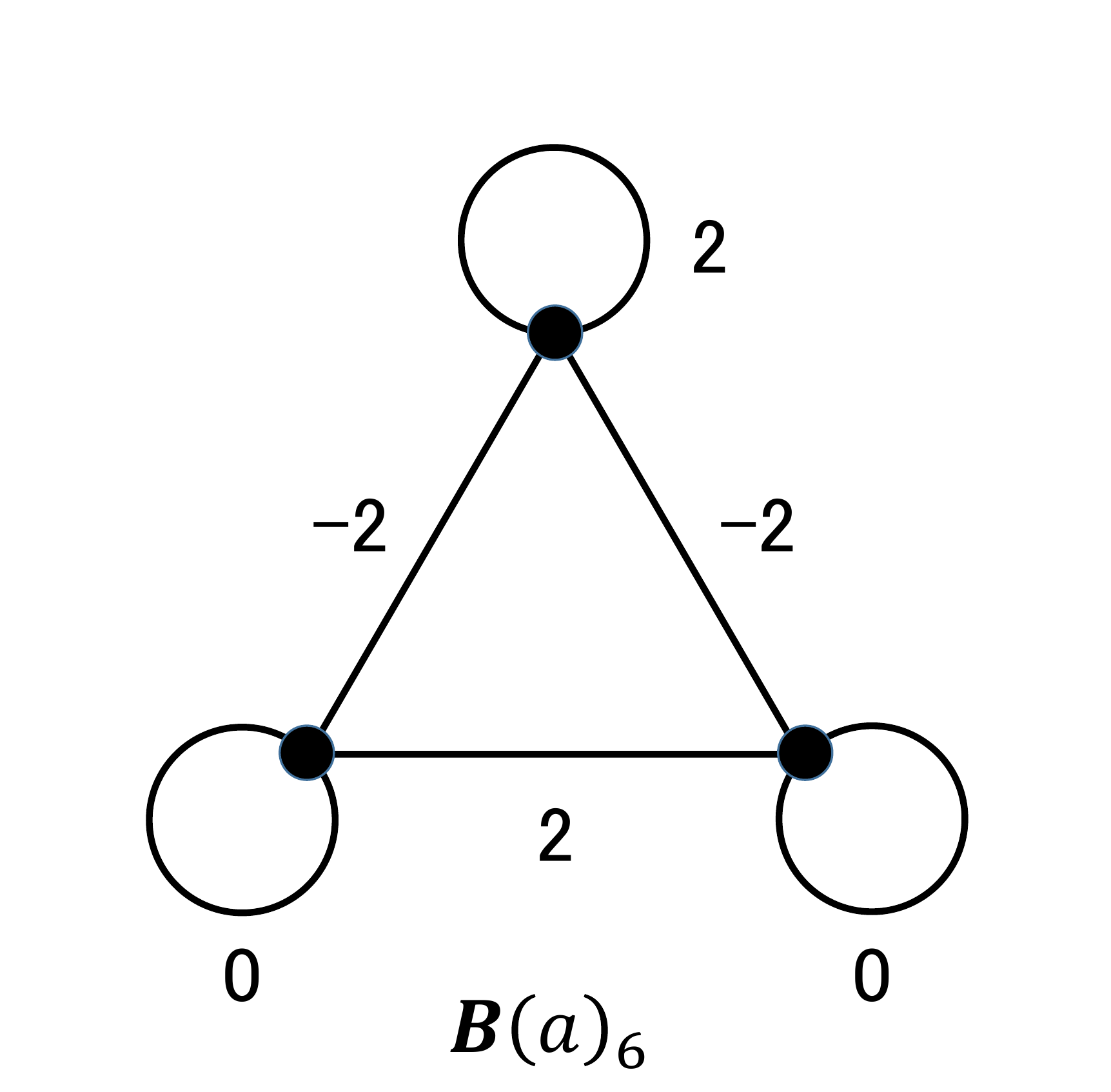}\\
\end{tabular}
\caption{$\bm B(a)_1$, $\bm B(a)_2$, $\bm B(a)_3$, $\bm B(a)_4$, $\bm B(a)_5$, $\bm B(a)_6$}
\label{B(a)}
\end{center}
\end{figure}

\begin{description}
\item[Case 4-1]$\bm B(a)$, $-\bm C(c)$, $\bm C(a)$, $-\bm C(b)$ and  $\bm B(a)_1$  exist.\\
$S$ is decreased by 
\begin{align*}
\bm g_{\bm B(a)_1}'=\bm g_{\bm B(a)_1}+\bm C(a),\qquad 
\bm g_{\bm C(a)}'=\bm g_{\bm C(a)}-\bm C(a).
\end{align*}

\item[Case 4-2] $\bm B(a)$, $-\bm C(c)$, $\bm C(a)$, $-\bm C(b)$ and  $\bm B(a)_2$  exist.\\
By 
\begin{align*}
\bm g_{\bm B(a)_2}'=\bm g_{\bm B(a)_2}+\bm B(a),\qquad
\bm g_{\bm B(a)}'=\bm g_{\bm B(a)}-\bm B(a),
\end{align*}
$S$ is not changed, but $\bm g_{\bm B(a)_2}$ now has $\bm B(a)_6$. Then we will decrease $S$ in 
Case 4-6 below.

\item[Case 4-3] $\bm B(a)$, $-\bm C(c)$, $\bm C(a)$, $-\bm C(b)$ and  $\bm B(a)_3$  exist.\\
$S$ is decreased by 
\begin{align*}
\bm g_{\bm B(a)_3}'=\bm g_{\bm B(a)_3}-\bm C(c),\qquad
\bm g_{-\bm C(c)}'=\bm g_{-\bm C(c)}+\bm C(c).
\end{align*}

\item[Case 4-4] $\bm B(a)$, $-\bm C(c)$, $\bm C(a)$, $-\bm C(b)$ and  $\bm B(a)_4$  exist.\\
Because of the symmetry of $b$ and $c$, we can decrease $S$ as in Case 4-3.

\item[Case 4-5] $\bm B(a)$, $-\bm C(c)$, $\bm C(a)$, $-\bm C(b)$ and  $\bm B(a)_5$  exist.\\
$S$ is decreased by 
\begin{align*}
\bm g_{\bm B(a)_5}'=\bm g_{\bm B(a)_5}-\bm C(c),\qquad
\bm g_{-\bm C(c)}'=\bm g_{-\bm C(c)}+\bm C(c).
\end{align*}

\item[Case 4-6] $\bm B(a)_6$, $-\bm C(c)$, $\bm C(a)$ and $-\bm C(b)$ exist.\\
$S$ is decreased by 
\begin{align*}
\bm g_{\bm B(a)_6}'=\bm g_{\bm B(a)_6}+\bm C(a),\qquad
\bm g_{\bm C(a)}'=\bm g_{\bm C(a)}-\bm C(a).
\end{align*}
\end{description}
\end{description}

Now we look at Case 5.

\begin{description}
\item[Case 5] $\bm B(a)$, $-\bm C(c)$, $\bm C(a)$ and $-\bm B(c)$ exist.\\
As in Case 4 $\deg(c)\ge 2$ by the existence of $\bm C(a)$.
Then $\Bg_{-\bm C(d)}$ has additional edges connecting to $c$. The possible cases
are,  \ 1) $cc$ alone, \ 2) at least one $ac$, or \ 3) 2 $bc$'s.
These three cases are depicted in Figure \ref{-C(c)}.

\begin{figure}[htbp]
  \begin{center}
  \begin{tabular}{ccc}
  \includegraphics[width=40mm]{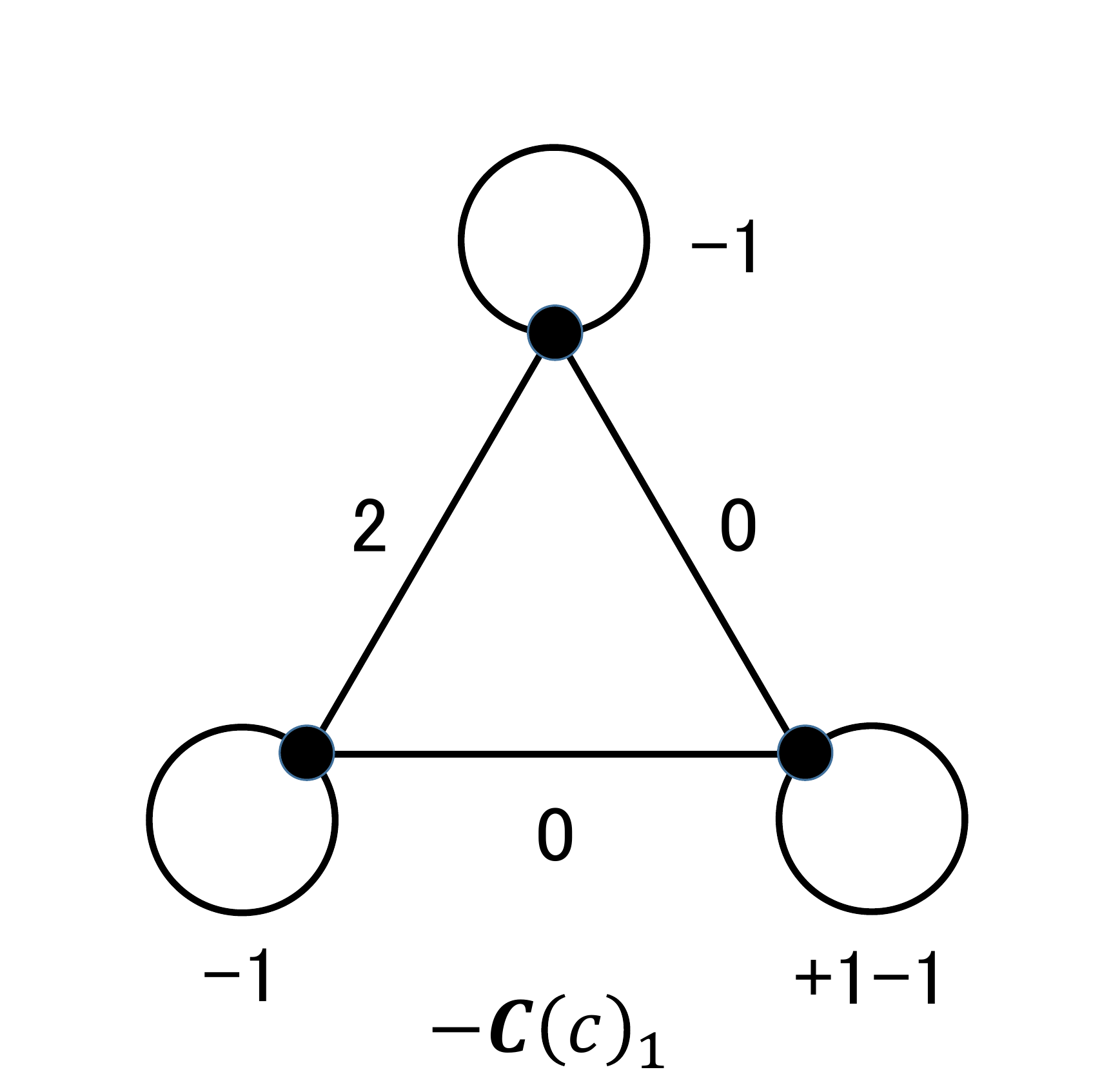} &
  \includegraphics[width=40mm]{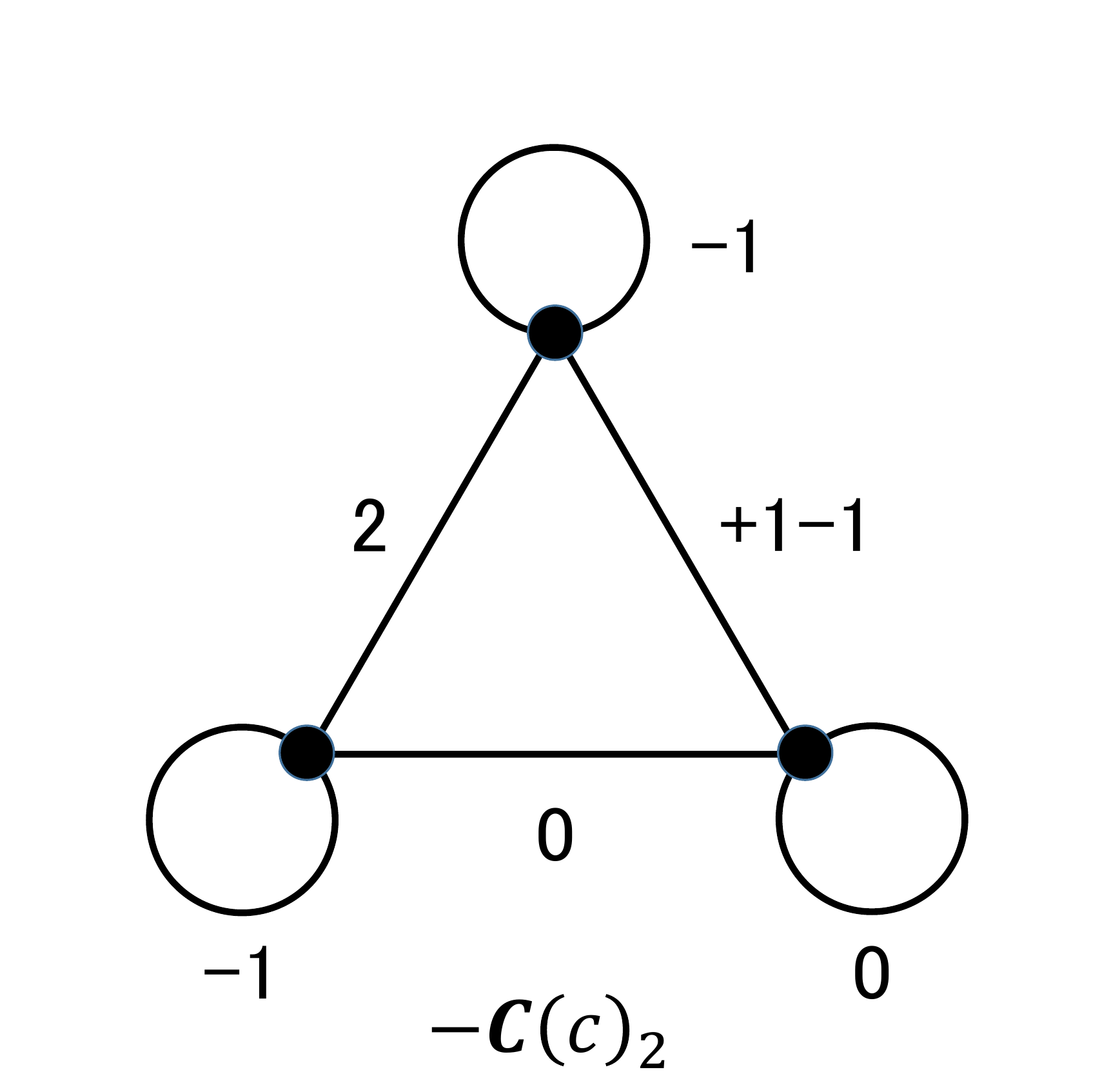} &
  \includegraphics[width=40mm]{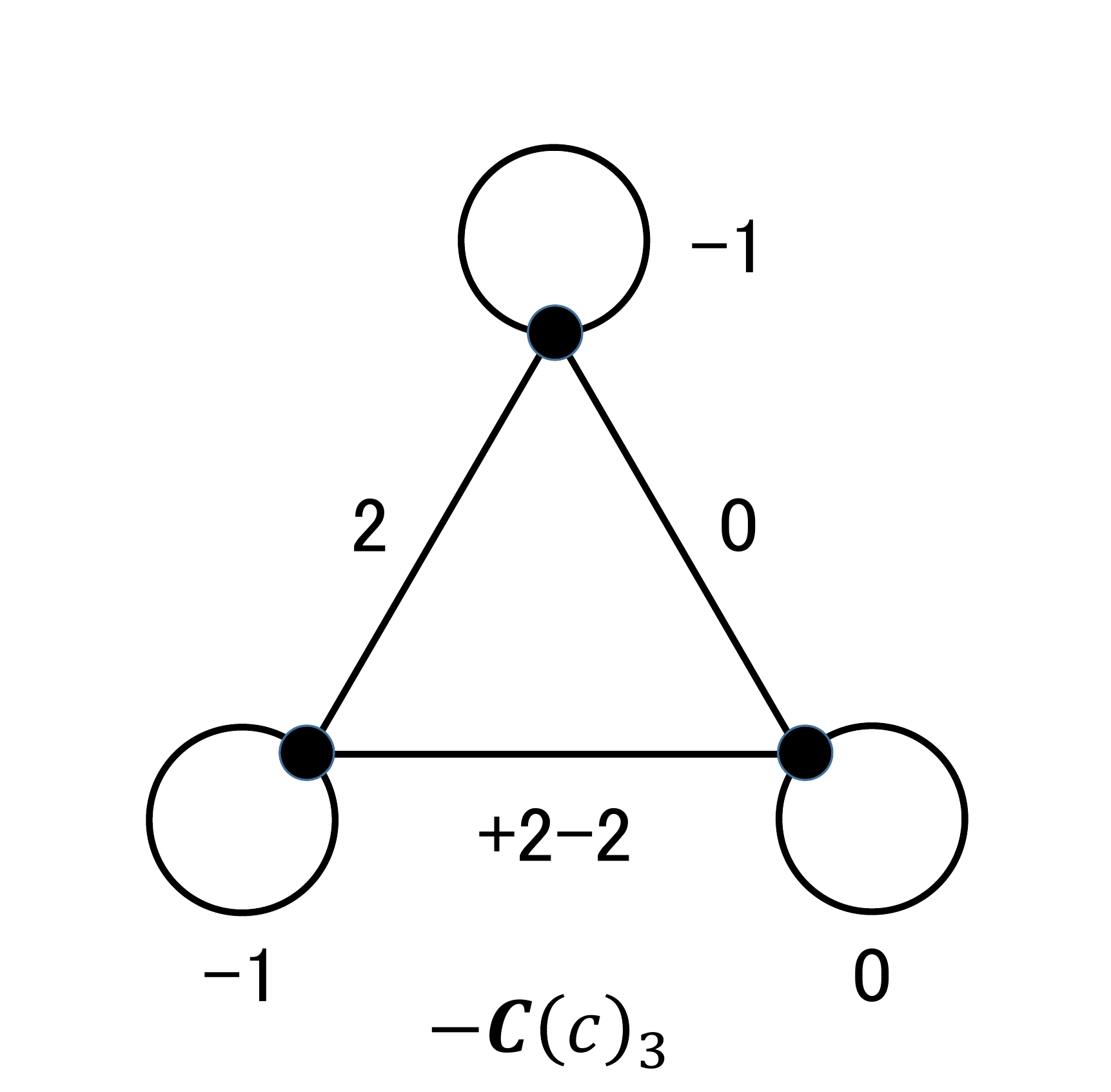}
  \end{tabular}
  \caption{$-\bm C(c)_1$, $-\bm C(c)_2$, $-\bm C(c)_3$}
  \label{-C(c)}
  \end{center}
\end{figure}

\begin{description}
\item[Case 5-1] $\bm B(a)$, $-\bm C(c)_1$, $\bm C(a)$ and $-\bm B(c)$ exist.\\
$S$ is decreased by 
\begin{align*}
\bm g_{-C(c)_1}'=\bm g_{-C(c)_1}-\bm B(c),\qquad
\bm g_{-B(c)}'=\bm g_{-B(c)}+\bm B(c).
\end{align*}

\item[Case 5-2] $\bm B(a)$, $-\bm C(c)_2$, $\bm C(a)$ and $-\bm B(c)$ exist.\\
$S$ is decreased by
\begin{align*}
\bm g_{-\bm C(c)_2}'=\bm g_{-\bm C(c)_2}+\bm B(a),\qquad
\bm g_{\bm B(a)}'=\bm g_{\bm B(a)}-\bm B(a).
\end{align*}

\item[Case 5-3] $\bm B(a)$, $-\bm C(c)_3$, $\bm C(a)$ and $-\bm B(c)$ exist.\\
$S$ is decreased by 
\begin{align*}
\bm g_{-\bm C(c)_3}'=\bm g_{-\bm C(c)_3}+\bm C(a),\qquad
\bm g_{\bm C(a)}'=\bm g_{\bm C(a)}-\bm C(a).
\end{align*}
\end{description}
\end{description}

We have eliminated patterns $\bm A$, $\bm D$ and $\bm B$.  The remaining pattern is $\bm C$.
\begin{description}
\item[Case 6] $\bm C$ exists.\\
Suppose that $\bm C(a)$ exists.
In the absence of $\pm \bm A$, $\pm \bm D$ and $\pm \bm B$ and the pair
$(\bm C(a), -\bm C(a))$, the excess of $bc$ can not be canceled.
Hence this case is impossible.
\end{description}

We have now eliminated all the patterns.  We now review the moves we needed to decrease $S$.
Except for Case 3-1, all the moves were  exchanges of edges between two graphs, which correspond 
to moves of degree two.  In Case 3-1 we needed a move of degree three. Hence
$\max_{\Bb\in \N A} {\rm MD}(A_\Bb)  \le  3$. Together with \eqref{eq:a222}
we have $\max_{\Bb\in \N A} {\rm MD}(A_\Bb)  = 3$.

\bigskip
\noindent
II. \ Proof of ${\rm MC}(A) = 5$.

Next we show ${\rm MC}(A) = 5$.
As discussed above, our argument before Case 3-2 applies also to the
higher Lawrence lifting $A^{(N)}$.
$\Bb$'s can be different in different layers in \eqref{omomiwa=b}. 
Therefore we need to check  Case 4 and Case 5 again for higher Lawrence lifting.
The argument is actually simple.
In Case 4, we consider at most five patterns (at most five graphs) which consist of two $\bm B(a)$'s, 
$-\bm C(c)$, $\bm C(a)$ and $-\bm C(b)$, whose sum is the zero vector.
This shows that a move of type at most five decreases $S$ in the Case 4 for $A^{(N)}$.
In Case 5 we consider at most four patterns (at most four graphs),  whose sum is the zero vector.
Hence a move of type at most four decreases $S$ in the Case 5 for $A^{(N)}$.
This proves  ${\rm MC}(A) \le 5$.

To establish the equality, we construct an indispensable move whose type is five.
Let ${\bm g}_1,\dots,{\bm g}_5$ be graphs displayed in the upper row
and %
let $\hat{\bm g}_1,\dots,\hat{\bm g}_5$ be graphs displayed in the lower row of
Figure \ref{g_5}.
We show that 
\[
\Bz = (\Bz_1, \dots,\Bz_5)=(\bm g_1,\dots,\bm g_5) -(\hat{\bm g}_1,\dots,\hat{\bm g}_5)
\]
is  an indispensable move $A^{(5)}$ by Proposition \ref{lem:indispensable-for-lawrence}.
First, $\Bz_i$, $i=1,\dots,5$, are patterns $\bm B$ or $\bm C$ and they are
indispensable moves for $A$.  By the argument after Proposition \ref{lem:indispensable-for-lawrence}
we can start from arbitrary slice $\Bz_k$.
\begin{figure}[htbp]
  \begin{center}
\includegraphics[width=37mm]{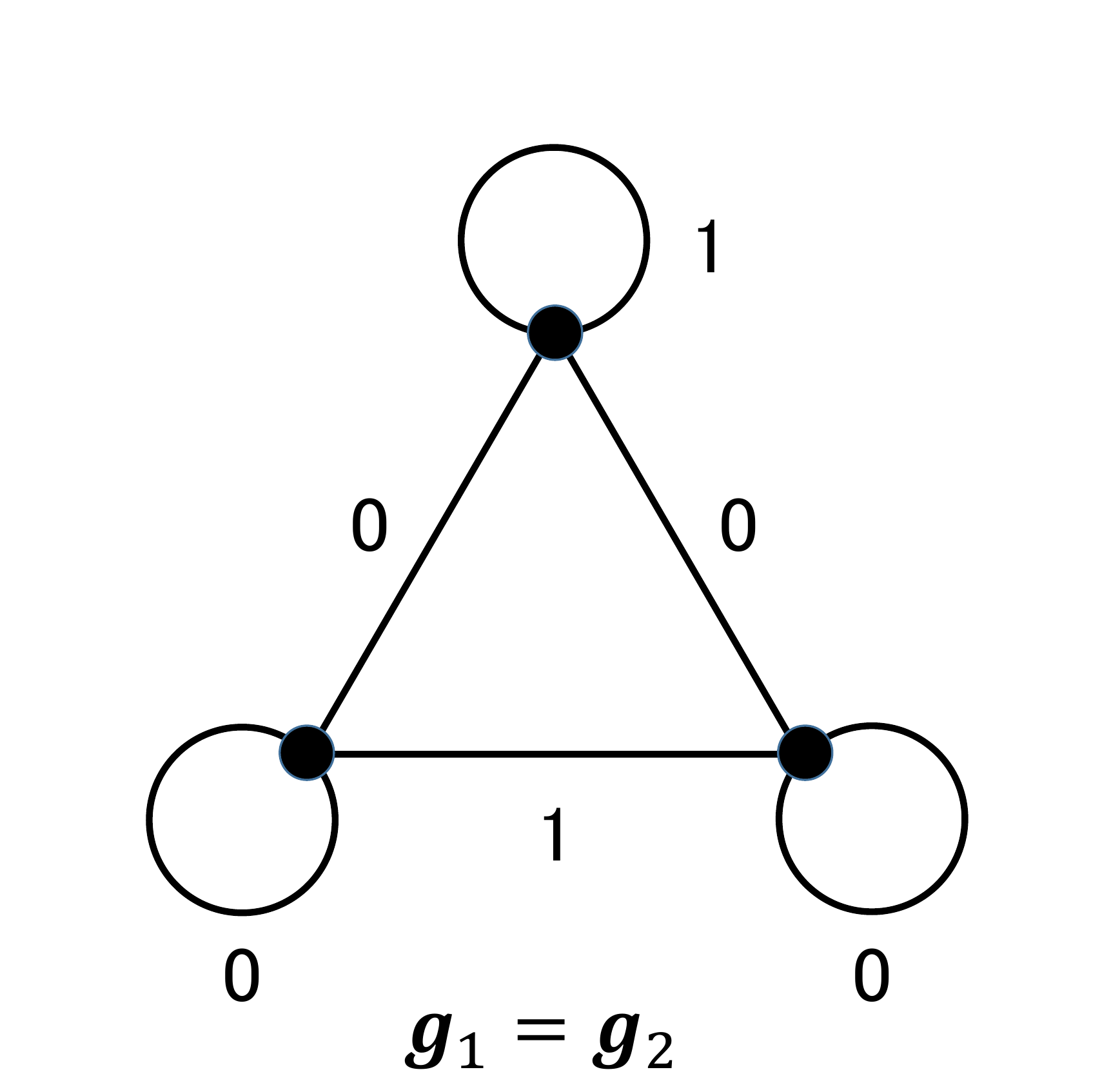}
\includegraphics[width=37mm]{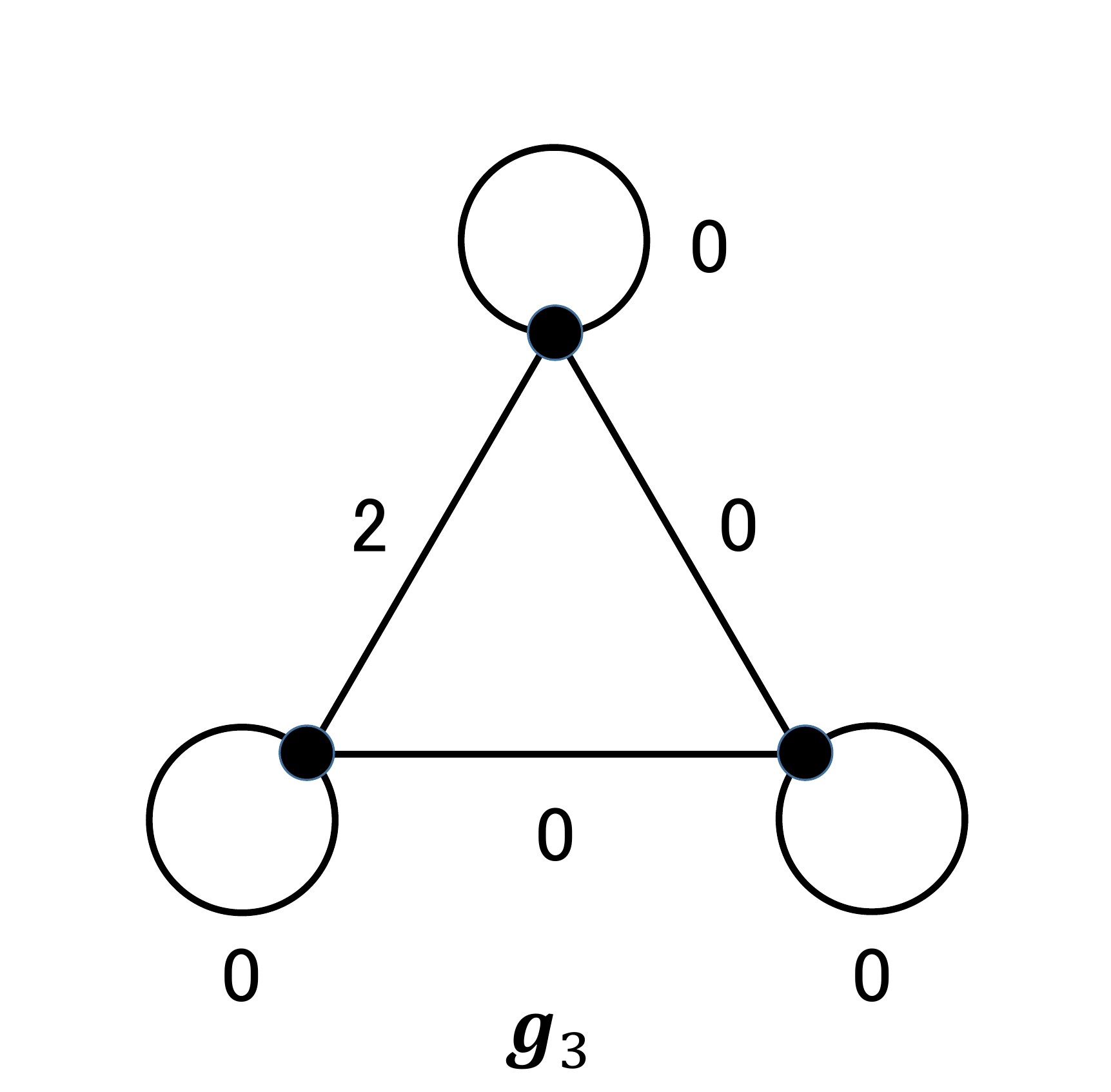}
\includegraphics[width=37mm]{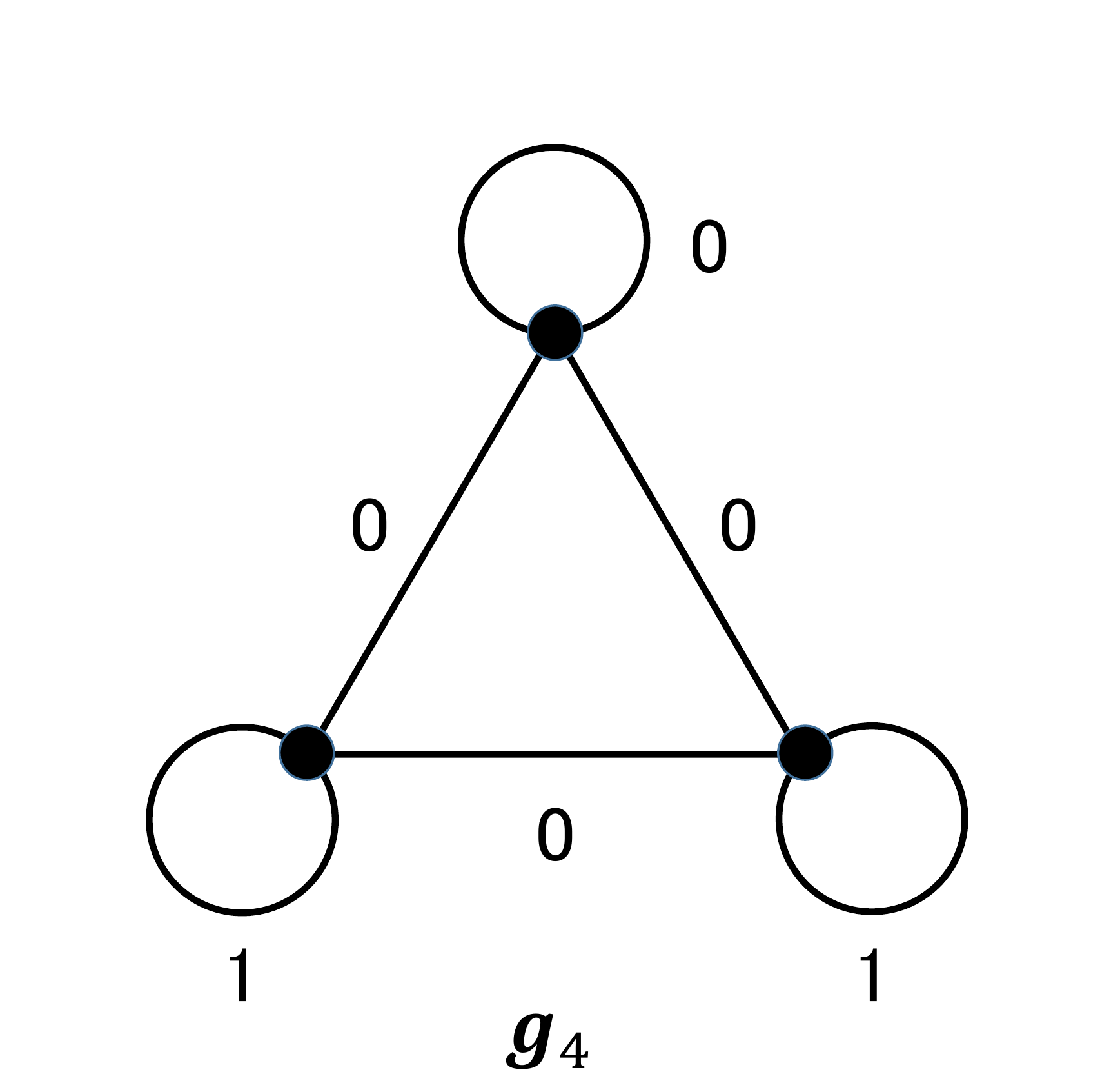}
\includegraphics[width=37mm]{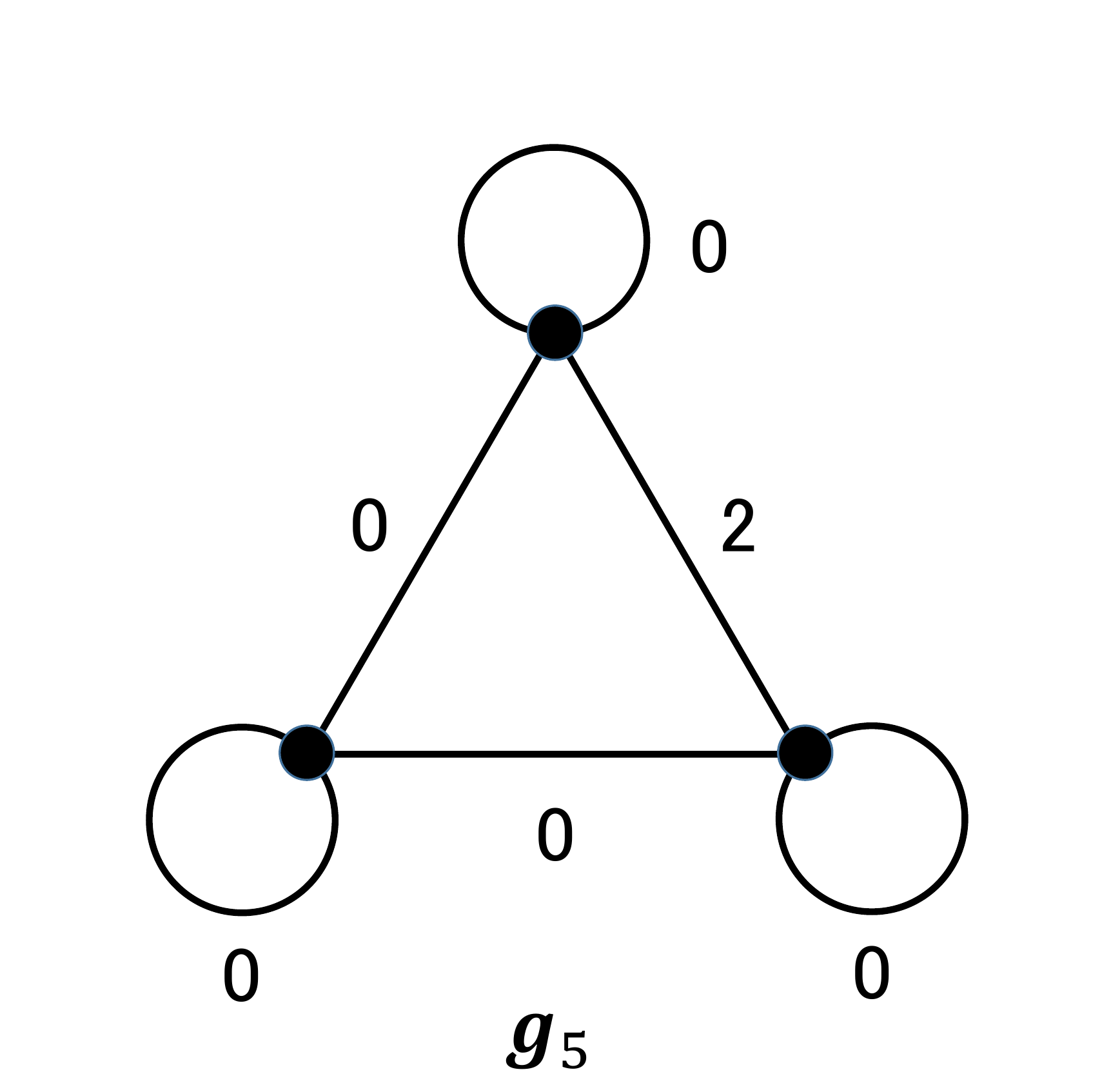}
\includegraphics[width=37mm]{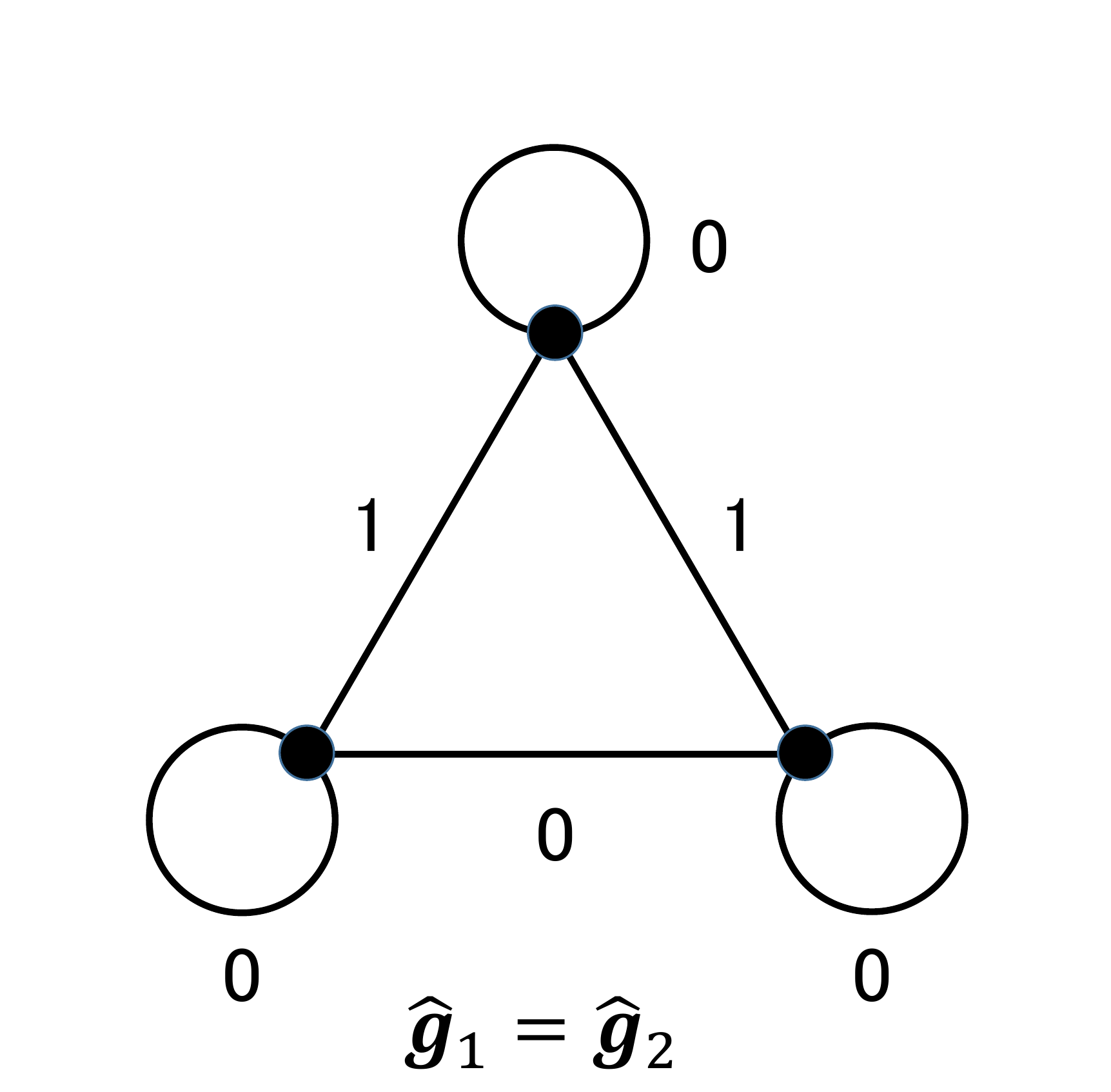} 
\includegraphics[width=37mm]{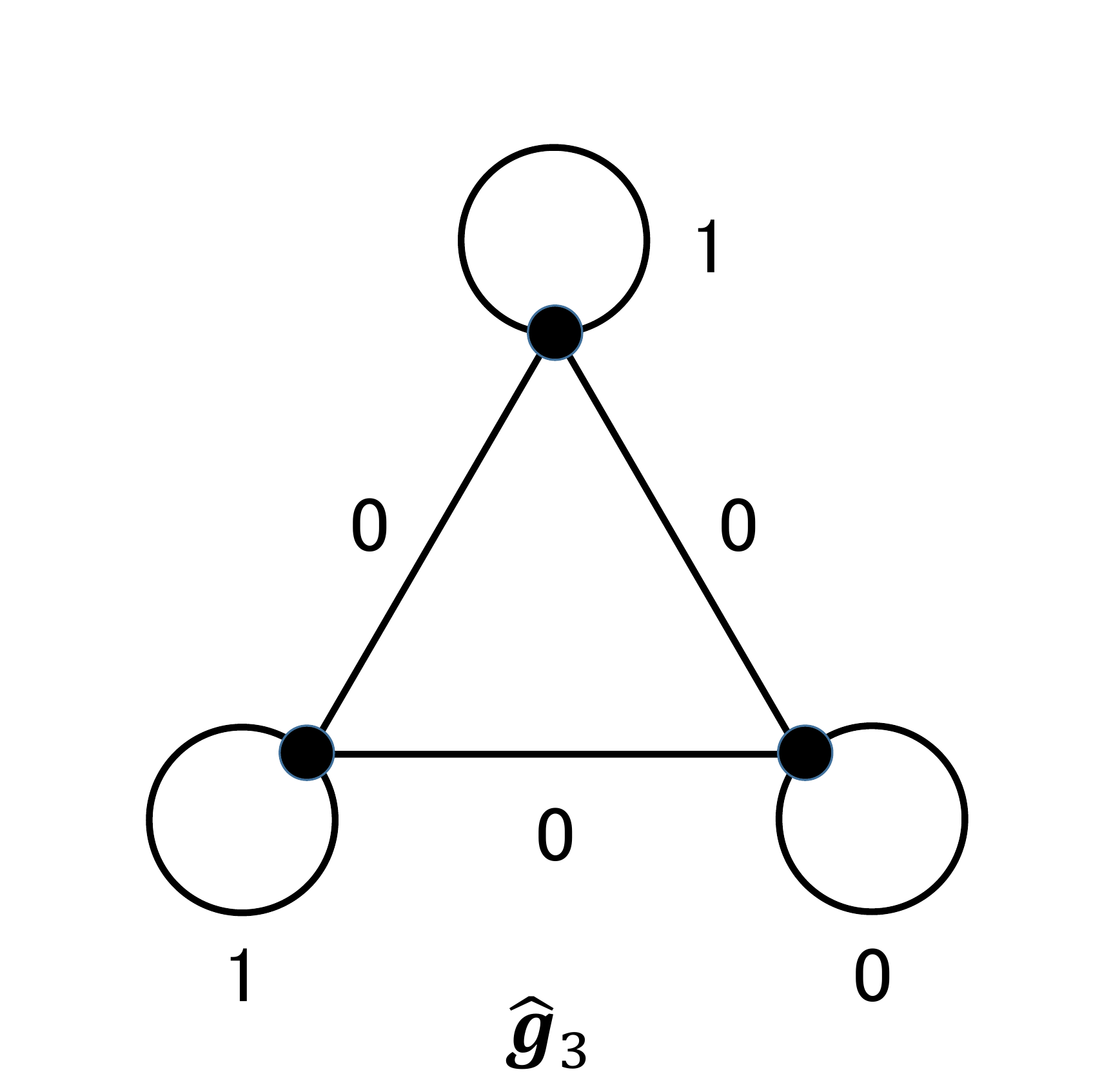} 
\includegraphics[width=37mm]{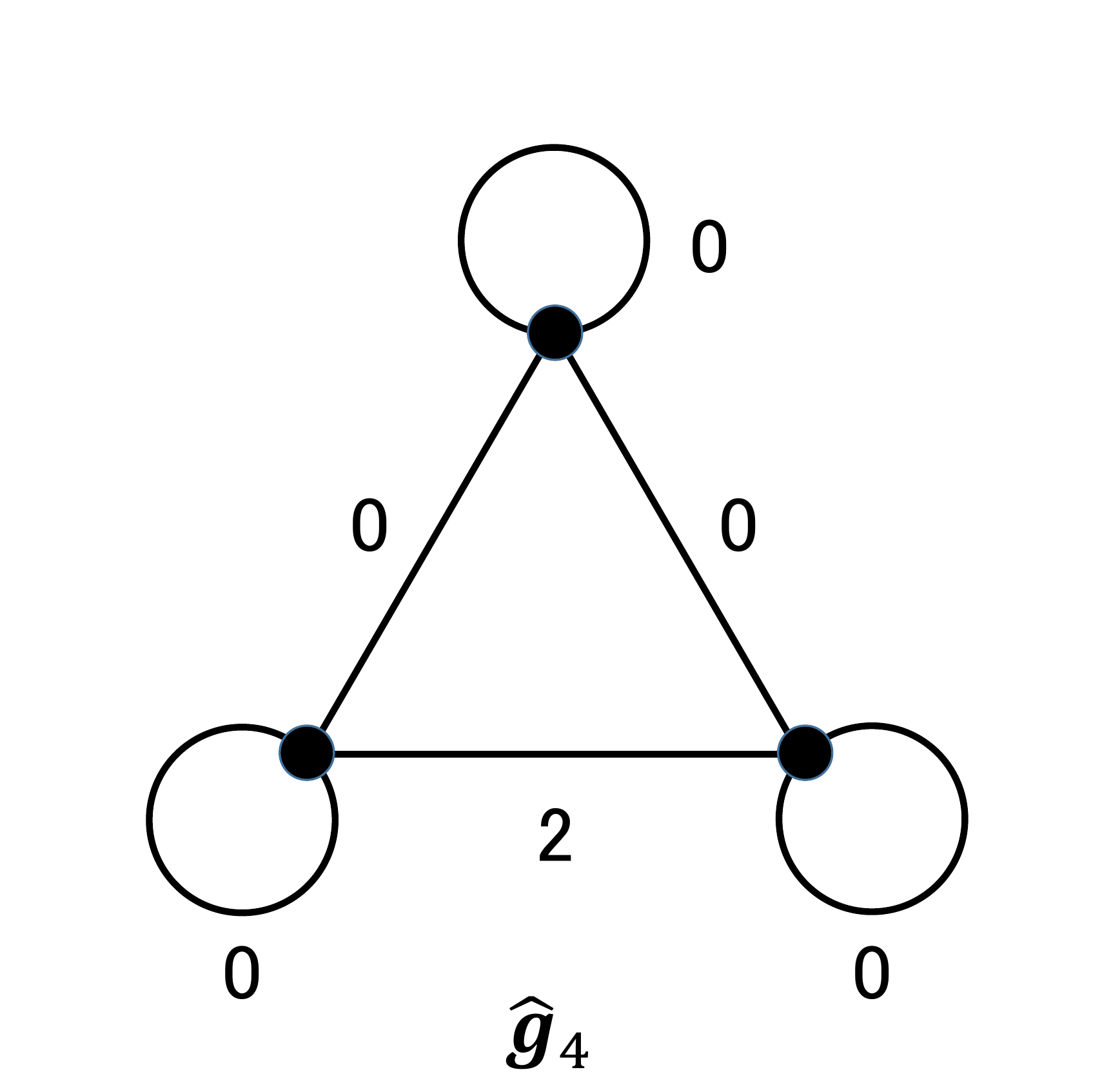} 
\includegraphics[width=37mm]{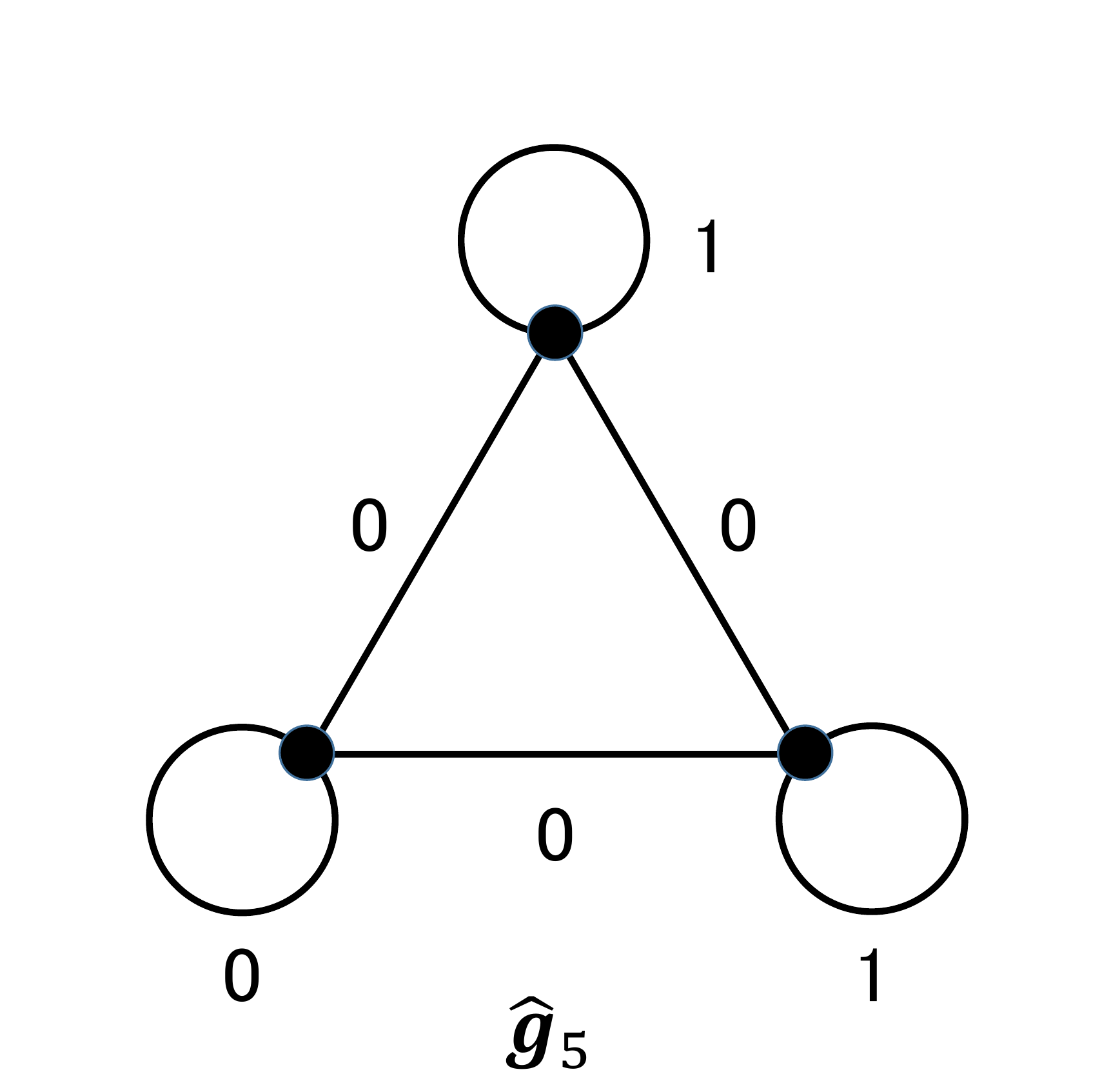} 
 \caption{Graphs $\bm g_1=\bm g_2$, $\bm g_3$, $\bm g_4$, $\bm g_5$ and 
$\hat{\bm g}_1=\hat{\bm g}_2$, $\hat{\bm g}_3$, $\hat{\bm g}_4$, $\hat{\bm g}_5$}
  \label{g_5}
 \end{center}
\end{figure}
  
We start with $\Bz_3$.
Since edges $aa, bb$ in $\Bz_3$ %
have to be canceled,  we need %
$\Bz_1$ and $\Bz_4$.
Since the edge $bc$ in $\Bz_4$ has to be canceled, we need $\Bz_2$.
Also, since the edge $ac$ in $\Bz_1$ has to be canceled we need $\Bz_5$.
Hence we need all slices and this proves that $\Bz$ is indispensable.

\bigskip
\noindent
III. \ Proof of ${\rm MD}(A_\Bb)  = 2$ for  $\Bb\neq (2,2,2)$. 

Recall that only Case 3-1 
needed a degree-three move.  We show that this move is not needed if $\Bb\neq (2,2,2)$, by a series of lemmas.

We write elements of the Graver basis by their positive part and their negative part, e.g., $\bm A=\bm A^+-\bm A^-$. 
We only need to consider the condition on $\Bb$ such that we need degree-three moves to decrease 
$S$  for the case
\begin{equation}
\label{eq:degree-3-special-case}
S=\sum_{i=1}^3 |\bm g_i - \hat\Bg_i|, \quad 
\hat{\bm g}_1=\bm g_1-\bm B(a), \ \hat{\bm g}_2=\bm g_2+\bm C(c), \ \hat{\bm g}_3=\bm g_3-\bm B(b)
\end{equation}
and $\bm b=A\bm g_i=A\hat{\bm g}_i$, $i=1,2,3$.
Note there is the symmetry of vertex $a$ and $b$.

\begin{lemma}
\label{lem:cc}
If degree-two moves do not decrease $S$ in \eqref{eq:degree-3-special-case}, then
\begin{align*}
\bm e_{cc} \le\bm g_2-\bm C(c)^-,\quad\bm e_{bc}\not\le \bm g_2-\bm C(c)^-,
\quad\bm e_{ca}\not\le\bm g_2-\bm C(c)^-.
\end{align*}
\end{lemma}
\begin{proof}
Since the degree of vertex $c$ of $\bm C^-$ is less than that of $\bm B(a)^+$ by one,  
the degree of vertex $c$ of $\bm g_2-\bm C(c)^-$ is greater than one.
Then $\bm e_{cc}\le \bm g_2-\bm C(c)^-$, $\bm e_{bc}\le \bm g_2-\bm C(c)^-$, or $\bm e_{ca}\le\bm g_2-\bm C(c)^-$. %
If $\bm e_{bc}\le \bm g_2-\bm C(c)^-$, then $S$ is decreased by the following exchange of edges:
\begin{align*}
\bm g_2'=\bm g_2+\bm B(b),\quad\bm g_3'=\bm g_3-\bm B(b).
\end{align*}
Hence $\bm e_{bc}\not\le \bm g_2-\bm C(c)^-$. 
We also have $\bm e_{ca}\not\le\bm g_2-\bm C(c)^-$
by the symmetry between $a$ and $b$.
\end{proof}

\begin{lemma}
\label{lem:ba}
If degree-two moves do not decrease $S$ in \eqref{eq:degree-3-special-case}, then
\begin{align*}
\bm e_{bc}\le \bm g_1-\bm B(a)^+,\quad\bm e_{bb}\not\le \bm g_1-\bm B(a)^+,\quad\bm e_{ab}\not\le\bm g_1-\bm B(a)^+ .
\end{align*}
\end{lemma}
\begin{proof}
Since the degree of vertex $b$ of $\bm B(a)^+$ is less than that of $\bm C(a)^-$ by one,  
the degree of vertex $b$ of $\bm g_1-\bm B(a)^+$ is greater than one.
Then $\bm e_{bc}\le \bm g_1-\bm B(a)^+$, $\bm e_{bb}\le \bm g_1-\bm B(a)^+$, or $\bm e_{ab}\le \bm g_1-\bm B(a)^+$.
If $\bm e_{bb}\le \bm g_1-\bm B(a)^+$, then $S$ is decreased by the following exchange of edges:
\begin{align*}
\bm g_1'=\bm g_1-\bm C(c),\quad\bm g_2'=\bm g_2+\bm C(c).
\end{align*}
Similarly if $\bm e_{ab}\le \bm g_1-\bm B(a)^+$, $S$ is decreased by the following  exchange of edges:
\begin{align*}
\bm g_1'=\bm g_1+\bm B(b),\quad\bm g_3'=\bm g_3-\bm B(b).
\end{align*}
\end{proof}

By the symmetry of $a$ and $b$, the following lemma also holds.
\begin{lemma}
\label{cor:bb}
If degree-two moves do not decrease $S$ in \eqref{eq:degree-3-special-case}, then
\begin{align*}
\bm e_{ca}\le \bm g_3-\bm B(b)^+,\quad\bm e_{aa}\not\le \bm g_3-\bm B(b)^+,\quad\bm e_{ab}\not\le\bm g_3-\bm B(b)^+ .
\end{align*}
\end{lemma}

\begin{lemma}
\label{lem:cgeq3}
Suppose that degree-two moves do not decrease $S$ in \eqref{eq:degree-3-special-case} and
$\deg(a)\ge 3$ or $\deg(b)\ge 3$. Then $\deg(c)\ge 3$.
\end{lemma}
\begin{proof}
By symmetry let $\deg(a)\ge 3$.
By Lemma \ref{cor:bb}, in this case, $\bm 2e_{ca}\le \bm g_3-\bm B(b)^+$. Hence $\deg(c)\ge 3$.
\end{proof}

By this lemma we can assume that $\deg(c)\ge 3$ if $\Bb\neq (2,2,2)$.
Hence our proof is completed by the following lemma.
\begin{lemma}
\label{lem:d3not}
If $\deg(c)\ge 3$, then  $S$ in \eqref{eq:degree-3-special-case} can be decreased by degree-two moves.
\end{lemma}
\begin{proof}
By Lemma \ref{lem:cc}, if $\deg(c)\ge 3$, then $\bm 2e_{cc}\le\bm g_3-\bm B(b)^+$.
Then the following series of exchanges of edges decreases $S$:
\begin{align*}
\bm g_2'&=\bm g_2-\bm D(a),\quad\bm g_3'=\bm g_3+\bm D(a), \\
\bm g_1'&=\bm g_1+\bm C(a),\quad\bm g_2'=\bm g_2-\bm C(a),\\
\bm g_2'&=\bm g_2+\bm B(c),\quad\bm g_3'=\bm g_3-\bm B(c),\\
\bm g_1'&=\bm g_1-\bm A,\quad\bm g_2'=\bm g_2+\bm A.
\end{align*}
\end{proof}

\section{Complete bipartite graphs as base configurations}
\label{sec:complete-bipartite-graphs}

In this section we  take incidence matrices $A(I,J)$ of complete bipartite graphs $K_{I,J}$ as base configurations and study the maximum Markov degree of the configurations defined by their fibers.
The fibers correspond to two-way transportation polytopes.  
In algebraic statistics, $A(I,J)$ is the design matrix specifying the row sums and the column sums of an $I\times J$ two-way contingency table and the $N$-th Lawrence lifting $A(I,J)^{(N)}$ is the design matrix for no-three-factor interaction model for $I\times J\times N$ three-way contingency tables.

A remarkable fact for the case of complete bipartite graphs 
is that the maximum Markov degree is three irrespective of $I$ and $J$ as we show in
Section \ref{subsec:MD-for-two-way}.  On the other hand
the Markov complexity grows with $I$ and $J$.  Lower bound for the Graver complexity has been
obtained by \cite{berstein09}, \cite{kudo-takemura}.  In Section 
\ref{subsec:lower-MC} we give a lower bound for the Markov complexity, which appears on the 
right-hand side of \eqref{eq:main-result} in our main theorem.

\subsection{Markov degree for two-way transportation polytopes}
\label{subsec:MD-for-two-way}

In this section we prove that the Markov degree of configurations for two-way transportation polytopes is at most three. As discussed in Section \ref{sec:intro}, recently this fact was proved by 
Domokos and  Jo\'o (\cite{domokos-joo}) in a more general setting.  However in this section
we give a proof, which is a direct extension of a proof in \cite{YamaguchiOgawaTakemura}.

Let $\Br \in \N^{I}$ and $\Bc\in\N^{J}$ be two non-negative integer vectors with $\sum_{i=1}^{I}r_{i}=\sum_{j=1}^{J}c_{j}$.
The two-way transportation polytope is the set of all non-negative matrices $\Bx=(x_{ij})$ whose row sum vector is $\Br$ and column sum vector is $\Bc$.
Let  $T_{\Br,\Bc}$  be the set of integral matrices in the transportation polytope.
Then 
\[
T_{\Br,\Bc}= {\mathcal F}_{A(I,J),(\Br, \Bc)}
\]
is the the $(\Br,\Bc)$-fiber for the incidence matrix $A(I,J)$ of the complete bipartite graph $K_{I,J}$.
We regard an element in $T_{\Br,\Bc}$ as complete bipartite graph with non-negative integral weights on edges, which is denoted by $\Bg = (g(ij)\mid (i,j)\in[I]\times[J])$.
Set $\Be=(e_{ij}) \in \N A(I,J)_{(\Br, \Bc)}$ arbitrarily.
Then an element of the corresponding fiber $\cF_{A(I,J)_{(\Br, \Bc)},\Be}$ can be identified with some multiset $\{ \Bg_{1},\ldots,\Bg_{N}\}$ satisfying $\Bg_{k} \in T_{\Br,\Bc}, k=1,\ldots, N$, and $\sum_{k}g_{k}(ij)=e_{ij}, (i,j)\in[I]\times[J]$.
Haase and Paffenholz \cite{haase2009poly} studied the $3\times3$ transportation polytopes.
When $I=J$ and $\Br=\Bc=(1,\ldots,1)^{\top}$, the corresponding transportation polytope is the Birkhoff polytope.

\begin{theorem}
\label{thm:tpoly} 
The toric ideal associated with the transportation polytope is generated by binomials of degree two and three, i.e., $\max_{(\Br,\Bc)\in \N A(I,J)} {\rm MD}(A(I,J)_{(\Br,\Bc)})=3$.
\end{theorem}

The rest of this subsection is devoted to the proof of Theorem \ref{thm:tpoly}.
Our proof is a direct extension of the proof for the Birkhoff polytope in \cite{YamaguchiOgawaTakemura}.
We modify the terminologies in \cite{YamaguchiOgawaTakemura} to be suitable for our setting.

\begin{definition}
\label{def:proper}
An $I\times J$ integer matrix $\Bg=(g(ij))$ is a {\em proper graph} if $\Bg$ is an element of $T_{\Br,\Bc}$.
A multiset $\{ \Bg_{1},\ldots,\Bg_{N}\}$ is {\em proper} if each $\Bg_{k}, k=1,\ldots,N$, is a proper graph.
\end{definition}
For two proper graphs $\Bg$ and $\hat{\Bg}$, we call $D_{\Bg,\hat{\Bg}}:=\sum_{i,j}|g(ij)-\hat{g}(ij)|$ the {\em size of differences}.
\begin{definition}
\label{def:improper}
An $I\times J$ integer matrix $\Bg=(g(ij))$ is an {\em improper graph} if $\Bg$ has the row sum $\Br$ and column sum $\Bc$, and there exists a unique edge $(i^* , j^*) \in [I] \times [J]$ such that
\begin{align*}
g(i^* j^*)=-1, \quad g(ij) \geq 0 , \ \forall (i, j) \neq (i^* , j^*).
\end{align*}
We call $g(i^* j^*)$ an {\em improper edge} of $\Bg$.
A multiset $\{ \Bg_{1},\ldots,\Bg_{N}\}$ is {\em improper} if one of $\{ \Bg_{1},\ldots,\Bg_{N}\}$ is an improper graph, the others are proper graphs, and $\sum_{k=1}^N g_{k}(ij) \geq 0, \forall i,j$. 
\end{definition}

\begin{definition}
\label{def:collision}
An $I\times J$ integer matrix $\Bg=(g(ij))$ is a {\em graph with collision} if $g(ij)\geq 0, \forall i,j$, the column sum of $\Bg$ is $\Bc$ and there exists $i^* \in [I]$ such that 
\begin{align*}
\sum_{j=1}^J g(i^{*}j)=r_{i^{*}}+1, \quad \sum_{j=1}^J g(ij)\leq r_{i}+1, \ \forall i\neq i^*.
\end{align*}
In this case we also say that the graph $\Bg$ contains a {\em collision} or the vertex $i^*$ {\em collides} in $\Bg$.
\end{definition}

We often denote a multiset $\{ \Bg_{1},\ldots,\Bg_{N} \}$ of $I\times J$ integer matrices by $\cS$ if $\sum_{k=1}^{N}g_{k}(ij) \geq 0, \forall i,j$, and each element $\Bg_{k},\forall k$, is one of the graphs defined in Definitions \ref{def:proper}--\ref{def:collision}.
The multiset $\cS$ is denoted by $\cP$ (resp. $\cI$) when $\cS$ is proper (resp. improper) and we want to emphasize it.

We now introduce some operations. 
Let $\cS=\{ \Bg_{1},\ldots,\Bg_{N}\}$ be a multiset of graphs in Definitions \ref{def:proper}--\ref{def:collision}.
Consider a pair of distinct graphs in $\cS$, say $\Bg_{k_1}=(g_{k_{1}}(ij))$ and $\Bg_{k_2}=(g_{k_{2}}(ij))$. 
Fix $i_{1},i_{2}\in[I]$ and $j^{*}\in[J]$ arbitrarily and set the two matrices $\Bz_{k_{1}}=(z_{k_{1}}(ij))$ and $\Bz_{k_{2}}=(z_{k_{2}}(ij))$ as
\begin{align*}
z_{k_{1}}(ij) = 
\begin{cases}
+1,& (i,j)=(i_{2},j^{*}), \\
-1,& (i,j)=(i_{1},j^{*}), \\
0,& \text{otherwise},
\end{cases}
\qquad
z_{k_{2}}(ij) = 
\begin{cases}
+1,& (i,j)=(i_{1},j^{*}), \\
-1,& (i,j)=(i_{2},j^{*}), \\
0,& \text{otherwise}.
\end{cases}
\end{align*}
The {\em swap} $\{k_1,k_2 \} : i_1 \swapj{j^*} i_2$ for $\cS$ is an operation transforming $\cS$ into another multiset $\cS^{\prime}$ of matrices defined by
\begin{align*}
\cS^{\prime} = ( \cS \setminus \{ \Bg_{k_{1}} , \Bg_{k_{2}} \} )
\cup \{ \Bg_{k_{1}}+\Bz_{k_{1}}, \Bg_{k_{2}}+\Bz_{k_{2}} \}.
\end{align*}
Note that the resulting $\cS^{\prime}$ has the same sums of weights of each edge as the original $\cS$, although the elements $\cS^{\prime}$ may not be graphs in Definitions \ref{def:proper}--\ref{def:collision}.

Let us consider $n\in\N$ swaps on the same pair of graphs $g_{k_{1}},g_{k_{2}}\in\cS$ and denote them as
\begin{align*}
(\Bz_{k_{1}}^{(1)},\Bz_{k_{2}}^{(1)}),\ldots,(\Bz_{k_{1}}^{(n)},\Bz_{k_{2}}^{(n)}).
\end{align*}
Consider the following operation, which transforms a multiset $\cS$ into another multiset $\cS^{\prime}$ without changing sums of weights of each edge:
\begin{align*}
\cS^{\prime} = ( \cS \setminus \{ \Bg_{k_{1}} , \Bg_{k_{2}} \} )
\cup \{ \Bg_{k_{1}}+\sum_{l=1}^{n}\Bz_{k_{1}}^{(l)}, \Bg_{k_{2}}+\sum_{l=1}^{n}\Bz_{k_{2}}^{(l)} \}.
\end{align*}
We call this operation a {\em swap operation among two graphs of $\cS$} and denote it as $\cS\swapoperation{\{k_{1},k_{2}\}}\cS^{\prime}$ or merely $\cS\swapoperation{}\cS^{\prime}$.
If both of $\cS$ and $\cS^{\prime}$ are proper, the operation is nothing but the move of degree two.

\begin{lemma}
\label{lem:collision}
Let $\cS=\{ \Bg_{1},\ldots,\Bg_{N}\}$ be a multiset of graphs without any improper edge and suppose that the $k$th and the $k^{\prime}$th graphs contain some collisions.
If $\sum_{j=1}^{J} (g_{k}(ij)+g_{k^{\prime}}(ij))=2r_{i}$ for each $i\in[I]$, we can resolve all the collisions by a swap operation among these two graphs.
\end{lemma}
\begin{proof}
We may assume $g_{k}(ij)=0$ or $g_{k^{\prime}}(ij)=0$ for each $i\in[I], j\in[J]$.
Let $\bar{\Bg}_{k}=(\bar{g}_{k}(j))$ and $\bar{\Bg}_{k^{\prime}}=(\bar{g}_{k^{\prime}}(j))$ be the $J$-dimensional row vectors whose $j$th elements $\bar{g}_{k}(j)$ and $\bar{g}_{k^{\prime}}(j)$ are the multisets of symbols defined by
\begin{align*}
\bar{g}_{k}(j) :=
\{ \underbrace{1,\ldots,1}_{g_{k}(1j)},\ldots, \underbrace{I,\ldots,I}_{g_{k}(Ij)} \}, \quad
\bar{g}_{k^{\prime}}(j) :=
\{ \underbrace{1,\ldots,1}_{g_{k^{\prime}}(1j)},\ldots, \underbrace{I,\ldots,I}_{g_{k^{\prime}}(Ij)} \},
\qquad j\in[J].
\end{align*}
Suppose that the vertex $i\in[I]$ collides in $\Bg_{k}$.
This means that the symbol $i$ appears $(r_{i}+1)$ times in $\bar{\Bg}_{k}$ and $(r_{i}-1)$ times in $\bar{\Bg}_{k^{\prime}}$.
To resolve the collision of $i$, we temporarily assign the different labels to vertices as follows.
First, we assign $i_{1},\ldots,i_{r_{i}-1}$ to $(r_{i}-1)$ $i$'s in each of $\bar{\Bg}_{k}$ and $\bar{\Bg}_{k^{\prime}}$.
Second, for each vertex not colliding in these two graphs, say $i^{\prime}$, we assign $i_{1}^{\prime},\ldots,i_{r_{i^{\prime}}}^{\prime}$ to $r_{i^{\prime}}$ $i^{\prime}$'s in each of $\bar{\Bg}_{k}$ and $\bar{\Bg}_{k^{\prime}}$.
Finally, for each colliding vertex different from $i$, say $\idp$, we assign $\idp_{1},\ldots,\idp_{r_{\idp}-1}$ to $(r_{\idp}-1)$ $\idp$'s in each graph and $\hat{i}_{1}^{\prime\prime},\hat{i}_{2}^{\prime\prime}$ to the remaining two $\idp$'s.
At this point, each symbol except $i$ appears once in each of $\bar{\Bg}_{k}$ and $\bar{\Bg}_{k^{\prime}}$.

Let $s=\sum_{j=1}^{J}c_{j}$ and define the $2\times s$ matrix $D=(d_{K\alpha})$ satisfying the following equations as multisets:
\begin{align*}
\{d_{K\alpha},\ldots,d_{K(\alpha+c_{j}-1)}\}=\bar{g}_{K}(j),\quad K=k,k^{\prime}, \ \alpha=\sum_{m=1}^{j-1}c_{m}+1, \ j=1,\ldots,J.
\end{align*}
Let $G$ be a graph on the vertex set $[s]$ defined as follows: a directed edge $(\alpha,\beta )$ exists for $\alpha,\beta \in[s]$ if and only if $d_{k\beta}=d_{k^{\prime}\alpha}$.
The graph $G$ consists of disjoint directed paths and cycles.
Then, there exists a path starting from a vertex $\gamma\in[s]$ with $d_{k\gamma}=i$.
This path defines the swap operation among $\Bg_{k}$ and $\Bg_{k^{\prime}}$ with the original labels of vertices, which resolves the collision of $i$ without causing any new collision.
Repeatedly applying this discussion, we obtain the sequence of the swap operations resolving collisions among the two graphs.
Combining them into one swap operation, we obtain the desired swap operation among $\Bg_{k}$ and $\Bg_{k^{\prime}}$.
\end{proof}

\begin{lemma}
\label{lem:resolve-improper}
Let $\cI=\{ \Bg_{1}, \ldots, \Bg_{N} \}$ be an improper multiset with $g_{k}(ij)=-1$.
Then, by a swap operation among two graphs, $\cI$ can be transformed to a proper multiset.
\end{lemma}
\begin{proof}
Choose $i^{\prime}\in[I]$ with $i^{\prime}\neq i$ and $g_{k}(i^{\prime}j)>0$.
Since $\sum_{l=1}^{N}g_{l}(ij)\geq 0$, there exists $k^{\prime}\in[N]$ with $g_{k^{\prime}}(ij)>0$.
Perform a swap $\{k,k^{\prime}\}:i\swapj{j}i^{\prime}$ to resolve the improper element.
Then, $i$ collides in $\Bg_{k^{\prime}}$ and $i^{\prime}$ collides in $\Bg_{k}$.
Since $i$ (resp. $i^{\prime}$) appears $2r_{i}$ (resp. $2r_{i^{\prime}}$) times in the first two graphs in total, we can resolve these collisions by Lemma \ref{lem:collision} by a swap operation among these two graph.
Combining the process, we obtain a swap operation among two graphs transforming $\cI$ to a proper multiset.
\end{proof}

\begin{definition}
\label{def:resolvable}
We call the pair of two graphs $\Bg_{k}$ and $\Bg_{k^{\prime}}$ in Lemma \ref{lem:resolve-improper} a {\em resolvable pair} and denote it as $[\kim,\kpr]$.
\end{definition}

\begin{definition}
\label{def:compatible}
A swap operation among two graphs labeled by $A=\{k,k^{\prime}\}$ in $\cI \swapoperation{A} \cI^{\prime}$ is {\em compatible} with improper multisets $\cI$ and $\cI^{\prime}$ if there exists a common resolvable pair $[\kim,\kpr]$ of $\cI$ and $\cI^{\prime}$ such that $A\cap\{\kim,\kpr\}\neq\emptyset$.
\end{definition}

\begin{lemma}
\label{lem:key1}
Let $\cP=\{ \Bg_{1}, \ldots, \Bg_{N} \},\hat{\cP}=\{ \hat{\Bg}_{1}, \ldots, \hat{\Bg}_{N} \}$ be two proper multisets in $\cF_{A(I,J)_{(\Br, \Bc)},\Be}$ and suppose $\Bg_{k}\neq \hat{\Bg}_{k^{\prime}}$ for some $k,k^{\prime}$.
Then, their size $D_{\Bg_{k}, \hat{\Bg}_{k^{\prime}}}$ of differences can be decreased by a swap operation among two graphs of $\cP$, such that if the resulting multiset is not proper, it is improper and its improper graph and the $k$th graph form a resolvable pair.
\end{lemma}
\begin{proof}
Since $\Bg_{k}\neq \hat{\Bg}_{k^{\prime}}$, there exist $i\in[I]$ and  $j,j^{\prime}\in[J]$ satisfying $g_{k}(ij)<\hat{g}_{k^{\prime}}(ij)$ and $g_{k}(ij^{\prime})>\hat{g}_{k^{\prime}}(ij^{\prime})$.
Since $\cP$ and $\cP^{\prime}$ belong to $\cF_{A(I,J)_{(\Br, \Bc)},\Be}$, there exists $\kdp\in [N]$ with $\kdp \neq k$ and $g_{\kdp}(ij)>0$.
Choose $i^{\prime}\in[I]$ satisfying $i^{\prime}\neq i$ and $g_{k}(i^{\prime}j)>0$ and consider a swap operation $\{k,\kdp\} : i^{\prime}\swapj{j} i \swapj{j^{\prime}} i^{\prime}$ to $\cP$.
This operation decreases $D_{\Bg_{k}, \hat{\Bg}_{k^{\prime}}}$.
When $g_{\kdp}(i^{\prime}j^{\prime})>0$, the resulting multiset is proper.
Otherwise, the resulting multiset is improper where $[\kdp,k]$ forms a resolvable pair.
This proves the claim.
\end{proof}

\begin{lemma}
\label{lem:key2}
Let $\cI=\{ \Bg_{1}, \ldots, \Bg_{N} \}$ be an improper multiset and $\hat{\cP}=\{ \hat{\Bg}_{1}, \ldots, \hat{\Bg}_{N} \}$ be a proper multiset with the same sums $\Be$ of weights of edges.
Consider the $k^{\prime}$th graph $\hat{\Bg}_{k^{\prime}}$ of $\hat{\cP}$ and choose any resolvable pair $[\kim,\kpr]$ of $\cI$.
Then, by at most two swap operations among two graphs of $\cI$, we can 
(i) decrease the size $D_{\Bg_{\kpr},\hat{\Bg}_{k^{\prime}}}$ of differences, or 
(ii) make $\cI$ proper without changing $\Bg_{\kpr}$.
Furthermore, if the resulting multiset is not proper, then it is an improper multiset with a resolvable pair consisting of its improper graph and the $\kpr$th graph, and each intermediate swap operation between two consecutive improper multisets is compatible with them.
\end{lemma}

\begin{proof}
We may suppose $g_{\kim}(ij)=-1$ and $g_{\kpr}(ij)>0$ for some $i\in[I]$ and $j\in[J]$.
In the cases below, when a resulting multiset is improper, $[\kim,\kpr]$ will be a resolvable pair.
\begin{description}
\item[Case 1] $\hat{g}_{k^{\prime}}(ij)\geq g_{\kpr}(ij)$.\\
Since $\sum_{l=1}^{N}g_{l}(ij) = \sum_{l=1}^{N}\hat{g}_{l}(ij)>g_{\kpr}(ij)$, there exists $k\in[N]$ such that $k\neq \kim,\kpr$ and $g_{k}(ij)>0$.
Then, $[\kim,k]$ is a resolvable pair and $\cI$ can be transformed to a proper multiset without changing $\Bg_{\kpr}$ by Lemma \ref{lem:resolve-improper}.
This corresponds to (ii) of the lemma and summarized as $\cI \swapoperation{\{ \kim,k \} } \cP$.

\item[Case 2] $\hat{g}_{k^{\prime}}(ij)<g_{\kpr}(ij)$.\\
Since $\sum_{t=1}^{J}\hat{g}_{k^{\prime}}(it)=\sum_{t=1}^{J}\hat{g}_{\kpr}(it)$, there exists $j^{\prime}\in[J]$ with $\hat{g}_{k^{\prime}}(ij^{\prime})>g_{\kpr}(ij^{\prime})$.
Fix some $i^{\prime}\in[I]$ with $g_{\kpr}(i^{\prime}j^{\prime})>\hat{g}_{k^{\prime}}(i^{\prime}j^{\prime})$ arbitrarily.
\begin{description}
\item[Case 2-1]  $g_{\kim}(ij^{\prime})>0$.\\
We perform the swap operations $\{\kpr,\kim\}: i\swapj{j}i^{\prime}$ and $\{\kim,\kpr\}: i^{\prime}\swapj{j^{\prime}}i$ to $\cI$ at the same time, which decrease $D_{\Bg_{\kpr},\hat{\Bg}_{k^{\prime}}}$.
If $g_{\kim}(i^{\prime}j)>0$, the resulting multiset is proper.
Otherwise, the resulting multiset is improper.
This corresponds to (i) of the lemma and is summarized as $\cI \swapoperation{\{ \kpr,\kim \} } \cP$ or $\cI \swapoperation{\{ \kpr,\kim \} } \cI$.
\item[Case 2-2] $g_{\kim}(ij^{\prime})=0$.\\
Since $\hat{g}_{k^{\prime}}(ij^{\prime})>g_{\kpr}(ij^{\prime})$, there exists $\kdp\in[N]$ such that $\kdp\neq \kpr,\kim$ and $\hat{g}_{k^{\prime}}(ij^{\prime})>0$.
Fix $\idp\in[I]$ with $x_{\idp\kim}>0$ arbitrarily.
Consider the swap $\{\kdp,\kim\}:i\swapj{j^{\prime}}\idp$.
Then, $i$ collides in the $\kim$th graph and $\idp$ collides in the $\kdp$th graph.
By the similar argument as the proof of Lemma \ref{lem:collision}, we can resolve these collisions by a swap operation among the $\kim$th and the $\kdp$th graphs, which leaves $g_{\kim}(ij)=-1$ and makes $g_{\kim}(ij^{\prime})$ positive.
These operation can be done by a singe swap operation among the two graphs.
After that, this case reduces to Case 2-1.
Together with the subsequent operation of Case 2-1, Case 2-2 is summarized as $\cI \swapoperation{\{ \kim,\kdp \} } \cI \swapoperation{\{ \kpr,\kim \} } \cP$ or $\cI \swapoperation{\{ \kim,\kdp \} } \cI \swapoperation{\{ \kpr,\kim \} } \cI$.
\end{description}
\end{description} 
\end{proof}

We now give a proof of Theorem \ref{thm:tpoly} by the similar argument as \cite{YamaguchiOgawaTakemura}.
Let $\cP$ and $\hat{\cP}$ be two proper multisets belonging to the same fiber $\cF_{A(I,J)_{(\Br, \Bc)},\Be}$.
Choose any $k$th graph $\Bg_{k}$ of $\cP$ and any $k^{\prime}$th graph $\hat{\Bg}_{k^{\prime}}$ of $\hat{\cP}$ with $\Bg_{k}\neq \hat{\Bg}_{k^{\prime}}$.
Thanks to Lemmas \ref{lem:key1} and \ref{lem:key2}, allowing some intermediate improper multisets, we can make $\Bg_{k}$ identical with $\hat{\Bg}_{k^{\prime}}$ by a sequence of swap operations among two graphs of $\cP$.
We throw away this common graph from the two multisets and repeat the procedure.
In the end, $\cP$ can be fully transformed to $\hat{\cP}$.
Let us decompose the whole process of transforming $\cP$ to $\hat{\cP}$ into segments that consist of transformations from a proper multiset to another proper multiset with improper intermediate steps.
One segment is depicted as
$\cP_{1}\longleftrightarrow \cI_{1}\longleftrightarrow \cdots \longleftrightarrow \cI_{m}\longleftrightarrow \cP_{m}$
where each $\longleftrightarrow$ denotes a swap operation among two graphs in Lemmas \ref{lem:key1} or \ref{lem:key2}.
Then, for any consecutive multisets $\cI_{l}$ and $\cI_{l+1}, l=1,\ldots,m-1$, there exist proper multisets $\cP_{l}, \cP_{l}^{\prime}, l=1,\ldots, m-1$, satisfying
\begin{align*}
& \cP_{l}\longleftrightarrow \cI_{l}\longleftrightarrow \cI_{l+1} \longleftrightarrow \cP_{l+1}^{\prime}, \\
& \cP_{l}^{\prime}\longleftrightarrow \cI_{l} \longleftrightarrow \cP_{l}.
\end{align*}
By the compatibility of the swap operation in $\cI_{l}\longleftrightarrow \cI_{l+1}$, $\cP_{l}$ can be transformed to $\cP_{l+1}^{\prime}$ by a swap operation among three graphs.
Since $\cP_{l}^{\prime}\longleftrightarrow \cI_{l}$ and $\cI_{l} \longleftrightarrow \cP_{l}$ involve a common improper graph, 
$\cP_{l}^{\prime}$ can also be transformed to $\cP_{l}$ by a swap operation among three graphs.
Therefore, the process from $\cP_{1}$ to $\cP_{m}$ is realized by swap operations among three graphs as
\begin{center}
\includegraphics[width=6.5cm]{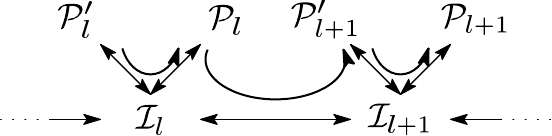}
\end{center}
This proves Theorem \ref{thm:tpoly}.

\subsection{Lower bound of Markov complexity for complete bipartite graphs}
\label{subsec:lower-MC}

In this section we give a lower bound for ${\rm MC}(A(I,J))$, $3\le I\le J$.

\begin{proposition}
\label{prop:lower-bound-IJ} For $3\leq I\leq J$,
\begin{align}
{\rm MC}(A(I,J))\geq(I-2)(J^2-1)/4+J-1 .\label{lower bound}
\end{align}
\end{proposition}

For the rest of this subsection we give a proof of Proposition \ref{prop:lower-bound-IJ}.
Let $d=\lfloor J/2\rfloor$.
We display $I\times J$ two-dimensional slice as follows:
\begin{align}
\begin{array}{c|c|c|c|c|c|c|c|}
 \multicolumn{1}{c}{} & \multicolumn{1}{c}{1} & \multicolumn{1}{c}{\cdots} & \multicolumn{1}{c}{d} & \multicolumn{1}{c}{(r)} & \multicolumn{1}{c}{J-d+1} & \multicolumn{1}{c}{\cdots} & \multicolumn{1}{c}{J}\\\cline{2-8}
   1 & \multicolumn{1}{c}{z_{1,1}} & \multicolumn{1}{c}{\cdots} & \multicolumn{1}{c}{z_{1,d}} & \multicolumn{1}{c}{(z_{1,r})} & \multicolumn{1}{c}{z_{1,J-d+1}} & \multicolumn{1}{c}{\cdots} & z_{1,J}   \\
   \vdots & \multicolumn{1}{c}{\vdots} & \multicolumn{1}{c}{} & \multicolumn{1}{c}{} & \multicolumn{1}{c}{} & \multicolumn{1}{c}{} & \multicolumn{1}{c}{} & \vdots   \\
   I & \multicolumn{1}{c}{z_{I,1}} & \multicolumn{1}{c}{\cdots} & \multicolumn{1}{c}{z_{I,d}} & \multicolumn{1}{c}{(z_{I,r})} & \multicolumn{1}{c}{z_{I,J-d+1}} & \multicolumn{1}{c}{\cdots}   & z_{I,J} \\ \cline{2-8}   
\end{array}\ ,
\end{align}
where, $r=d+1$ if $J$ is odd and $r$ does not exist if $J$ is even.
We  define $\Bz(i_1, i_2; j_1, ,j_2)$ as the following $I\times J$ table:
\begin{align}
\begin{array}{c|c|c|}
 \multicolumn{1}{c}{} & \multicolumn{1}{c}{j_1}&\multicolumn{1}{c}{j_2} \\\cline{2-3}
   i_1 & \multicolumn{1}{c}{+1}  & -1   \\
   i_2 & \multicolumn{1}{c}{-1}   & +1  \\ \cline{2-3}
\end{array}\  , 
\end{align}
where other entries are $0$.

We give an indispensable move $\Bz^*=\{z^*(i,j,k)\}$ of for $A(I,J)^{(N)}$, where 
$N=(I-2)(J-d)d+2d$ is the type of $\Bz^*$. The $I\times J$ slices of $\Bz^*$ as follows:
\begin{align*}
&\Bz(1, I; j, J-d+j),
\quad j=1,\cdots,d, \\
&\Bz(I-1, I ;J-d+j+1, j),
\quad j=1,\cdots,d-1, \\
&\Bz(I-1, I ;j+1, j), 
\quad j=d,(r),\\
&\Bz(i, i+1; j+1, j) 
\times j,\quad i=1,\dots,I-2,\ j=1,\dots,d,\\
&\Bz(i, i+1 ;r+1, r) 
\times d, \quad i=1,\dots I-2,\\
&\Bz(i, i+1; J-j+1, J-j)
\times j, \quad i=1,\dots,I-2,\ j=1,\dots,d-1.
\end{align*}
It is easy checked that $\sum_{k=1}^N z^*(i,j,k)=0$ for all $i,j$ and $\Bz^*$ is 
a move for $A(I,J)^{(N)}$. Also all slices of $\Bz^*$ are indispensable.
Therefore, if we can show that $\Bz^*$ is indispensable move, then
\[
{\rm MC}(A(I,J)) \geq (I-2)(J-d)d+2d \geq(I-2)(J^2-1)/4+J-1.
\]

Now we again use the argument after Proposition \ref{lem:indispensable-for-lawrence}.
We start with the slice $\Bz(1, I; 1, J-d+1)$. 
Since the (sum of) $(I,1)$-element is $-1$,  we need a slice whose $(I,1)$-element  is $+1$. 
Therefore we need $\Bz(I-1, I; J-d+2, 1)$. Since the sum of $(I,J-d+2)$-elements is $-1$, 
we need $\Bz(1, I; 2, J-d+2)$. In the same way, 
we find that $\Bz(1, I; j, J-d+j)$, $j=1,\dots,d$, \ $\Bz(I-1, I; J-d+j+1, j)$, $j=1,\dots,d-1,$ \  and $\Bz(I-1, I; j+1, j)$, $j=d,r$, are needed.

The sum of slices so far is as follows:
\begin{align*}
\begin{array}{c|c|c|c|c|c|c|c|}
 \multicolumn{1}{c}{} & \multicolumn{1}{c}{1} & \multicolumn{1}{c}{\cdots} & \multicolumn{1}{c}{d} & \multicolumn{1}{c}{(r)} & \multicolumn{1}{c}{J-d+1} & \multicolumn{1}{c}{\cdots} & \multicolumn{1}{c}{J}\\\cline{2-8}
  1 & \multicolumn{1}{c}{+1} & \multicolumn{1}{c}{\cdots} & \multicolumn{1}{c}{+1} & \multicolumn{1}{c}{(0)} & \multicolumn{1}{c}{-1} & \multicolumn{1}{c}{\cdots} & -1   \\
  I-1 & \multicolumn{1}{c}{-1} & \multicolumn{1}{c}{\cdots} & \multicolumn{1}{c}{-1} & \multicolumn{1}{c}{(0)} & \multicolumn{1}{c}{+1} & \multicolumn{1}{c}{\cdots}   & +1  \\ \cline{2-8}   
\end{array}\ .
\end{align*}

Since the sum of $(I-1, 1)$-elements is $-1$, we need $\Bz(I-2, I-1; 2, 1)$.
Since the sum of $(I-1, 2)$-elements is $-2$, we need $\Bz(I-2, I-1; 3, 2)\times 2$.
In the same way, we find that $\Bz(I-2, I-1; j+1, j)\times j$, \ $j=1,\dots,r,$ and $\Bz(I-2, I-1; J-j+1, J-j)\times j$, $j=1,\dots,d-1,$ are needed.

The sum of slices so far is as follows:
\begin{align*}
\begin{array}{c|c|c|c|c|c|c|c|}
 \multicolumn{1}{c}{} & \multicolumn{1}{c}{1} & \multicolumn{1}{c}{\cdots} & \multicolumn{1}{c}{d} & \multicolumn{1}{c}{(r)} & \multicolumn{1}{c}{J-d+1} & \multicolumn{1}{c}{\cdots} & \multicolumn{1}{c}{J}\\\cline{2-8}
  1 & \multicolumn{1}{c}{+1} & \multicolumn{1}{c}{\cdots} & \multicolumn{1}{c}{+1} & \multicolumn{1}{c}{(0)} & \multicolumn{1}{c}{-1} & \multicolumn{1}{c}{\cdots} & -1   \\
  I-2 & \multicolumn{1}{c}{-1} & \multicolumn{1}{c}{\cdots} & \multicolumn{1}{c}{-1} & \multicolumn{1}{c}{(0)} & \multicolumn{1}{c}{+1} & \multicolumn{1}{c}{\cdots}   & +1  \\ \cline{2-8}   
\end{array}\ .
\end{align*}

In the way, we find that $\Bz(i, i+1, j+1, j)\times j\ (i=1,\dots,I-3,\ j=1,\dots,r)$ and $\Bz(i, i+1; J-j+1, J-j)\times j\ (i=1,\dots,I-3,\ j=1,\dots,d-1)$ are needed.
Hence all slices are needed for cancellation and 
this implies that $\Bz^*$ is an indispensable move.

\medskip

\begin{remark}
There are indispensable moves whose types are larger than the one in \eqref{lower bound} 
for specific $I$ and $J$.
One example is the following move $\Bz$ of $5\times5\times 32$ table for the case $I=J=5$.
Each $5\times 5$ slice is a move of degree two of the form $\Bz(i_1, i_2; j_1,j_2)$.
We now list these 32 slices. In the list,  $-(i_1, i_2; j_1, j_2)$ denotes
$-\Bz(i_1, i_2; j_1, j_2)=\Bz(i_1, i_2; j_2, j_1)$.

\bigskip
{\small
\setlength{\tabcolsep}{3pt}
\begin{tabular}{|c|cccccccc|} \hline
slice & 1 & 2 & 3 & 4 & 5 & 6 & 7 & 8\\
move  & (1,5;1,5) & $-$(1,2;1,2) & $-$(1,3;2,3) & $-$(1,2;3,4) & $-$(1,3;4,5) & $-$(2,3;1,3) & (2,4;2,4) & $-$(2,4;3,5)\\ \hline
slice & 9 & 10 & 11 & 12 & 13 & 14 & 15 & 16\\
move & $-$(2,4;3,5) & (2,5;4,5) & (2,5;4,5) & $-$(3,4;1,2) & $-$(3,5;2,4) & $-$(3,5;2,4) & (3,5;3,5) & (3,5;3,5)\\ \hline
slice & 17 & 18 & 19& 20& 21& 22& 23& 24\\
move & $-$(3,4;4,5) & $-$(3,4;4,5) & $-$(3,4;4,5) & $-$(4,5;1,3) & (4,5;2,5) & (4,5;2,5) & $-$(4,5;3,4) & $-$(4,5;3,4) \\ \hline
slice & 25 & 26 & 27& 28& 29& 30& 31& 32\\
move & $-$(4,5;3,4) & $-$(4,5;4,5) & $-$(4,5;4,5) & $-$(4,5;4,5) & $-$(4,5;4,5) & $-$(4,5;4,5) & $-$(4,5;4,5) & $-$(4,5;4,5)  \\ \hline
\end{tabular}
}
\bigskip

\noindent
\begin{itemize}
\setlength{\itemsep}{-1pt}
\item Since $(1,1)$-element of slice 1 is $+1$, we need slice 2.
\item Since the sum of $(1,2)$-elements from slice 1 and 2 is $+1$, we need slice 3. \\
$\dots$ 
\item Since the sum of $(2,2)$-elements from slice 1 through slice 6 is $-1$, we need slice 7. 
\item Since the sum of $(2,3)$-elements from slice 1 through slice 7 is $+2$, we need slice 7 and 8.\\
$\dots$
\item Since the sum of $(4,4)$-elements from slice 1 through slice 25 is $+7$, we need slice 26 through slice 32.
\end{itemize}
Therefore, this move is indispensable.
\end{remark}

\section{Discussion}
\label{sec:discussion}
In this paper we investigated a series of configurations $A_\Bb$
defined by fibers of a given base configuration $A$. We proved that the
maximum Markov degree of the configurations is bounded from above by the Markov complexity of $A$.
From our examples, the equality between the maximum Markov degree and the Markov complexity
seems to hold only in special simple cases.  As discussed after the statement of Theorem \ref{thm:main}, 
the equality holds because the set of fibers for $A_\Bb$ is a subset of fibers of the
higher Lawrence lifting $A^{(N)}$ of $A$.  The strict inequality suggests that the former is a small
subset of the latter.
In particular for the case of incidence matrix complete bipartite graphs $K_{m,n}$, the maximum Markov 
degree for $A_b$ is three independently of $m$ and $n$, whereas the Markov complexity grows at least 
polynomially in $m$ and $n$ as shown in Section 4.2. 
Hence the discrepancy is large for this case.

Another interesting topic to investigate is the dependence of the Markov degree of $A_\Bb$ on $\Bb$.
The results of Haase and Paffenholz (\cite{haase2009poly}) suggest that for generic $\Bb$,
the Markov degree of $A_\Bb$ may be smaller than the maximum Markov degree.
The result of our \ref{thmmd3} on the specific $\Bb=(2,2,2)$ suggests that 
this may a general phenomenon.

\section*{Acknowledgment}
This work is partially supported by Grant-in-Aid for JSPS Fellows (No. 12J07561) from Japan Society for the Promotion of Science (JSPS).  We are grateful to Hidefumi Ohsugi 
for very useful suggestions and a proof mentioned in Remark \ref{rem:ohsugi}.

\bibliographystyle{abbrv}
\bibliography{fiber-configuration}

\end{document}